\documentclass[a4paper]{article}

\usepackage{amsfonts}
\usepackage{graphicx}
\usepackage{amsthm}
\usepackage{amsmath}
\usepackage{amssymb}
\usepackage{enumerate}
\usepackage{dsfont}
\usepackage{abstract}
\usepackage{verbatim}
\usepackage{caption}

\allowdisplaybreaks

\usepackage{sagetex}

\newtheorem{proposition}{Proposition}[section]
\newtheorem{lemma}[proposition]{Lemma}
\newtheorem{corollary}[proposition]{Corollary}
\newtheorem{theorem}[proposition]{Theorem}

\newtheorem*{theorem*}{Theorem}
\theoremstyle{definition}

\numberwithin{equation}{section}

\newenvironment{proofbold}[1]{\paragraph{Proof of {#1}.}}{\hfill$\square$}

\newenvironment{mysage}{\sagesilent}{\endsagesilent}

\title{\begin{center}
\LARGE
\textbf{Explicit $L^2$ bounds for the Riemann $\zeta$ function}
\bigskip
\end{center}}
\author{Daniele Dona, Harald A. Helfgott, Sebastian Zuniga Alterman}
\date{\today}

\begin{document}

\maketitle


\begin{abstract}
Explicit bounds on the tails of the zeta function $\zeta$
are needed for applications, notably for integrals involving $\zeta$
on vertical lines or other paths going to infinity.
Here we bound weighted $L^2$ norms of
tails of $\zeta$.

Two approaches are followed, each giving the better result on a different range.
The first one is inspired by the proof of the standard mean value theorem
for Dirichlet polynomials. 
The second approach, superior for large $T$, is based on classical lines, starting with an approximation to $\zeta$ via Euler-Maclaurin.

Both bounds give main terms of the correct order for $0<\sigma\leq 1$ and are strong enough to be of practical use for the rigorous computation of improper integrals.

We also present bounds for the $L^{2}$ norm of $\zeta$ in $[1,T]$ for $0\leq\sigma\leq 1$.
\end{abstract}
	
\tiny
\begin{mysage}
### Rounding ####
def roundup(x,d):
    return float(ceil(x*10^d)/10^d)
def rounddown(x,d):
    return float(floor(x*10^d)/10^d)

### Thresholds ###
TH=200 #For the numerical first approach (approximately)
THTrue=192 #For the numerical first approach (actually)
TZ=4 #For the integral of |zeta|^2
T0_012=2221 #For |zeta/s|^2 and 0<sigma<1/2 (least threshold for some sigma)
T0_121=639 #For |zeta/s|^2 and 1/2<sigma<1 (least threshold for some sigma)
T0_1=TH #For |zeta/s|^2 and sigma=1
T0_12_exact=2.26085*10^39 #For |zeta/s|^2 and sigma=1/2 #
Exp_12=ceil(log(T0_12_exact,10)) #10^(Exp_12-1)<T0_12_exact<10^Exp_12
T0_12=10^Exp_12
\end{mysage}
\normalsize

\newcommand{\Addresses}{{
  {\bigskip\footnotesize
\ \\
  D.~Dona, \textsc{Mathematisches Institut, Georg-August-Universit\"at G\"ottingen, Bunsenstrasse 3-5, 37073 G\"ottingen, Germany.}\\
  \texttt{daniele.dona@mathematik.uni-goettingen.de}

\medskip
\ \\
  H.~A. Helfgott, \textsc{Mathematisches Institut, Georg-August-Universit\"at G\"ottingen, Bunsenstrasse 3-5, 37073 G\"ottingen, Germany.}\\
  \texttt{harald-andres.helfgott@mathematik.uni-goettingen.de}

\medskip
\ \\
  S.~Zuniga Alterman, \textsc{Institut de Math\'{e}matiques de Jussieu, Universit\'{e} Paris Diderot P7, B\^ati-ment Sophie Germain, 8 Place Aur\'elie Nemours, 75013 Paris, France.}\\
  \texttt{sebastian.zuniga-alterman@imj-prg.fr}

}}}

\section{Introduction}\label{Int}

\subsection{Motivation}
Say we want to compute a line integral from $\sigma-i \infty$ to
$\sigma + i\infty$ involving the zeta function. Such integrals arise
often in work in number theory as inverse Mellin transforms. For example,
during his work on \cite{Hel19},
the second author had to estimate the double sum
\[D_{\alpha_1,\alpha_2}(y)=
\sum_{d\leq y} \sum_{l\leq y/d} \frac{\log\left(\frac{y}{d l}\right)}{d^{\alpha_1} l^{\alpha_2}},\]
and others of the same kind. Now, it
is not hard to show that
\[D_{\alpha_1,\alpha_2}(y) = \frac{1}{2\pi i} \int_{\sigma-i\infty}^{\sigma+i\infty}
\frac{\zeta(s+\alpha_1) \zeta(s+\alpha_2)}{s^2} y^s ds\]
for $\sigma>1$. Let  $0<\alpha_1, \alpha_2< 1$.
Shifting the line of integration to the left, we obtain
main terms coming from the poles at $s=1-\alpha_2$ and $s=1-\alpha_2$, and,
as a remainder term, the integral
\[\frac{1}{2\pi i} \int_{R_\beta} \frac{\zeta(s+\alpha_1) \zeta(s+\alpha_2)}{s^2}
y^s ds,\]
where $R_\beta$ is some contour to the left of the poles going
from $\beta-i\infty$ to $\beta+i \infty$.

It is possible to do rigorous numerical integration on bounded contours
in the complex plane, using, for instance, the ARB package \cite{Joh18}.
It then remains to bound the integral
\[\int_{\beta+i T}^{\beta+i\infty}
\frac{|\zeta(s+\alpha_1)| |\zeta(s+\alpha_2)|}{|s|^2}  ds,\]
the integral from $\beta-i\infty$ to $\beta - i T$ having the same absolute value. By the Cauchy-Schwarz inequality, the problem reduces to that of giving explicit bounds for
the integral
\begin{equation}\label{eq:terencehill}\int_{\beta+i T}^{\beta+i\infty}
\frac{|\zeta(s+\alpha_1)|^2}{|s|^2}  ds.\end{equation}
Finding such bounds is the main subject of this paper. 

\subsection{Methods and results}

Convexity bounds on $\zeta$ have been known explicitly
for more than 100 years \cite{Ba18}.
Since they are of the form
$\zeta(\sigma + i t) =O\left(t^{\frac{1-\sigma}{2}} \log t\right)$
for $0\leq \sigma \leq 1$, they imply that \eqref{eq:terencehill}
converges for $0< \sigma \leq 1$. There are also explicit
subconvexity bounds (that is, bounds stronger than convexity) for $\sigma=\frac{1}{2}$ (\cite{Leh70}, \cite{CG04}, \cite{PT15}, \cite{Hia16}) and for $\frac{1}{2}\leq\sigma\leq 1$ \cite{For02}.

Here, we produce better results in the $L^2$ norm
than can be obtained from such $L^\infty$ bounds.
Non-explicit bounds on the $L^2$-norm of $\zeta(\sigma+i t)$
are well known (\cite[Vol. 2, 806--819, 905--906]{Lan09},
\cite{HL18}, \cite{HL22}, \cite{Li24}; see the introduction to
\cite{Ing26} for an exposition).

\tiny
\begin{mysage}
#For checks about the following constants, see at the end of the file.
C_1_main=round( 28.29 ,ndigits=2)
C_121_num_a_main=round( 18.98 ,ndigits=2)
C_121_num_b_main=round( 0.61 ,ndigits=2)
C_121_opt_main=round( 12.95 ,ndigits=2)
C_12_num_main=round( 7.72 ,ndigits=2)
C_12_opt_main=round( 9.19 ,ndigits=2)
C_012_num_a_main=round( 0.5 ,ndigits=2)
C_012_num_b_main=round( 0.95 ,ndigits=2)
C_012_num_c_main=round( 5.62 ,ndigits=2)
C_012_num_d_main=round( 2.55 ,ndigits=2)
C_012_opt_main=round( 20.68 ,ndigits=2)
\end{mysage}
\normalsize

Our main result collects in a simplified form the results in Theorems \ref{thzeta} and \ref{th:z/s-all}.
\begin{theorem}\label{thm:main}
Let $0<\sigma\leq 1$. Then, the integral $\int_{T}^{\infty}\left|\frac{\zeta(\sigma+it)}{\sigma+it}\right|^{2}dt$ is bounded as follows
\begin{enumerate}[(1)]
\item\label{bigthm1} if $\sigma=1$, by
\begin{align*}
& \frac{\pi^{2}}{6}\cdot\frac{1}{T}+\sage{C_1_main}\cdot\frac{\log T}{T^2} & & \text{ \ for \ } T\geq\sage{T0_1};
\end{align*}
\item\label{bigthm121} if $\frac{1}{2}<\sigma<1$, by
\begin{align*}
& \frac{3\pi\zeta(2\sigma)}{5}\cdot\frac{1}{T}+\left(\sage{C_121_num_a_main}-\frac{\sage{C_121_num_b_main}}{\sigma-\frac{1}{2}}\right)\cdot\frac{1}{T^{2\sigma}} & & \text{ \ for \ } T\geq\sage{TH}, \\
& \zeta(2\sigma)\cdot\frac{1}{T}+\frac{\sage{C_121_opt_main}}{\left(\sigma-\frac{1}{2}\right)(1-\sigma)}\cdot\frac{1}{T^{2\sigma}} & & \text{ \ for \ } T\geq\sage{TZ};
\end{align*}
\item\label{bigthm12} if $\sigma=\frac{1}{2}$, by
\begin{align*}
& \frac{3\pi}{5}\cdot\frac{\log T}{T}+\sage{C_12_num_main}\cdot\frac{1}{T} & & \text{ \ for \ } T\geq\sage{TH}, \\
& \frac{\log T}{T}+\sage{C_12_opt_main}\cdot\frac{\sqrt{\log T}}{T} & & \text{ \ for \ } T\geq 10^{\sage{Exp_12}};
\end{align*}
\item\label{bigthm012} if $0<\sigma<\frac{1}{2}$, by
\begin{align*}
& \left(\frac{\sage{C_012_num_a_main}}{\sigma}+\frac{\sage{C_012_num_b_main}}{\frac{1}{2}-\sigma}+\sage{C_012_num_c_main}\right)\cdot\frac{1}{T^{2\sigma}}-\sage{C_012_num_d_main}\zeta(2\sigma)\cdot\frac{1}{T}& & \text{for \ }T\geq\sage{TH}, \\ 
& \frac{\zeta(2-2\sigma)}{2\sigma(2\pi)^{1-2\sigma}}\cdot\frac{1}{T^{2\sigma}}+\frac{\sage{C_012_opt_main}}{\sigma^{2}\left(\frac{1}{2}-\sigma\right)}\cdot\frac{1}{T} & & \text{for \ }T\geq\sage{TZ}.
\end{align*}
\end{enumerate}
\end{theorem}

In each pair of bounds above, the second one is stronger for large $T$ and fixed $\sigma$. 
The first bounds in cases \eqref{bigthm121}, \eqref{bigthm12}, \eqref{bigthm012} are obtained by a method explained in \S\ref{Exp}, based on the fact that the
Mellin transform is an isometry.  The second set of bounds and the single bound in case \eqref{bigthm1} use a different approach, explained in \S\ref{Opt}; it is based on the following explicit bounds
on the $L^{2}$ norm of the restriction of $\zeta(\sigma+it)$ to a segment. 

\tiny
\begin{mysage}
#For checks about the following constants, see at the end of the file.
Cz_1_main=round( 18.48 ,ndigits=2)
Cz_121_main=round( 5.22 ,ndigits=2)
Cz_12_a_main=round( 2.0 ,ndigits=2)
Cz_12_b_main=round( 23.06 ,ndigits=2)
Cz_012_main=round( 10.34 ,ndigits=2)
Cz_0_main=round( 9.37 ,ndigits=2)
\end{mysage}
\normalsize

\begin{theorem}\label{thm:mainz}
Let $0\leq\sigma\leq 1$ and $T\geq\sage{TZ}$. Then, the integral $\int_{1}^{T}|\zeta(\sigma+it)|^{2}dt$ is bounded from above by
\begin{align*}
& \frac{\pi^{2}}{6}\cdot T+\sage{Cz_1_main}\cdot\sqrt{T} & & \text{if $\sigma=1$,} \\
& \zeta(2\sigma)\cdot T+\frac{\sage{Cz_121_main}}{\left(\sigma-\frac{1}{2}\right)(1-\sigma)^{2}}\cdot\max\{T^{2-2\sigma}\log T,\sqrt{T}\} & & \text{if $\frac{1}{2}<\sigma<1$,} \\
& T\log T+\sage{Cz_12_a_main}\cdot T\sqrt{\log T}+\sage{Cz_12_b_main}\cdot T & & \text{if $\sigma=\frac{1}{2}$,} \\
& \frac{\zeta(2-2\sigma)}{(2\pi)^{1-2\sigma}(2-2\sigma)}\cdot T^{2-2\sigma}+\frac{\sage{Cz_012_main}}{\sigma^{2}\left(\frac{1}{2}-\sigma\right)}\cdot T & & \text{if $0<\sigma<\frac{1}{2}$,} \\
& \frac{\pi}{24}\cdot T^{2}+\sage{Cz_0_main}\cdot T\log T & & \text{if $\sigma=0$.}
\end{align*}
\end{theorem}

The error terms above are not optimal: bounds with the correct coefficient for the second-order term (and a non-explicit lower-order term) are known; for $\sigma=\frac{1}{2}$, see Ingham \cite{Ing26}, Titchmarsh \cite{Tit34}, Atkinson \cite{At49}, and Balasubramanian \cite{Ba78} (vd.\ Heath-Brown \cite{HB78} for an $L^{2}$ estimate of the lower-order term, and Good \cite{Go77} for a lower bound on its order). For $\frac{1}{2}<\sigma<\frac{3}{4}$, an estimate was given by Matsumoto \cite{Ma89}, later extended by Matsumoto and Meurman \cite{MM93} to $\frac{1}{2}<\sigma<1$. We will be more precise in Thms.~\ref{th:z->=1/2} and~\ref{asymptotic<1/2}.

It would seem feasible to improve on Theorem~\ref{thm:mainz} by starting from
Atkinson's formula for $\sigma=\frac{1}{2}$, or Matsumoto-Meurman's for
$\frac{1}{2}<\sigma<1$,
estimating all terms while foregoing cancellation.  One could then
deduce a bound for $0<\sigma<\frac{1}{2}$ by the functional equation, as in
Theorem~\ref{asymptotic<1/2} here.
For $\sigma=\frac{1}{2}$, yet another possibility would be to attempt to make
the work of Titchmarsh or Balasubramanian explicit.

An exposition of these alternative procedures -- in their current
non-explicit versions -- can be found in \cite[\S 1]{Ma00}. They are
based on the approximate functional equation,
or the Riemann-Siegel formula, which is closely related. Shortly after the appearance of the original version of the present paper, Simoni\v{c} provided an explicit bound in \cite[Cor.~5]{Sim20} for the case $\sigma=\frac{1}{2}$, improving on Theorems~\ref{thm:mainz} and~\ref{th:z->=1/2} for $T$ large enough.

For the sake of rigor, we have used interval arithmetic throughout,
implemented by ARB \cite{Joh18}, which we used via Sage.

\section{Classical foundations revisited}

\subsection{$O$ and $O^*$ notation}

When we write $f(x)=O(g(x))$ as $x\to a$ ($a=\pm\infty$ is allowed)
for a real or complex valued function $f$ and a real valued function $g$,
we mean that there is a constant $C$ such that $|f(x)|\leq Cg(x)$ in a
neighborhood of $a$. We write $f(x)=O^*(h(x))$ to mean that $|f(x)|\leq h(x)$ (either for all $x$ or in an explicitly stated neighborhood of $a$).

\subsection{Bernoulli polynomials}

We define the Bernoulli polynomials $B_k:\mathbb{R}\to\mathbb{R}$ inductively: $B_0(x)=1$ and for $k\geq 1$, $B_{k}(x)$ is determined by $B_k'(x)=kB_{k-1}(x)$ and $\int_0^1B_k(x)=0$. The $k$-th Bernoulli number $b_k$ is the constant term of $B_k(x)$. 
In particular, $B_1(x)=x-\frac{1}{2}$ and $B_2(x)=x^2-x+\frac{1}{6}$. 

\begin{lemma}[{\cite[Cor.~B.4, Exer.~B.5(e)]{MV07}}]\label{lem:BoundBernoulli}
  For $k\geq 1$, $\max_{x\in[0,1]}|B_{2k}(x)|=|b_{2k}|$ and $\max_{x\in[0,1]}|B_{2k+1}(x)|<\frac{2(2k+1)!}{(2\pi)^{2k+1}}$. In general, for every $k\geq 2$,
\begin{equation}\label{BernoulliBounds}
\max_{x\in[0,1]}|B_{k}(x)|\leq\frac{2\zeta(k)k!}{(2\pi)^k}.
\end{equation}
\end{lemma}

\subsection{Euler-Maclaurin summation formula}

Bernoulli polynomials appear naturally in the Euler-Maclaurin summation formula.

\begin{theorem}[\textbf{Euler-Maclaurin}]\label{EMSF}
Let $K$ be a positive integer. Let $X<Y$ be two real numbers such that the function $f:[X,Y]\rightarrow\mathbb{C}$ has continuous derivatives up to the $K$-th order on the interval $[X,Y]$. Then
\begin{equation}
\sum_{X<n\leq Y}f(n)=\int_X^Yf(x)dx+S(K)-\frac{(-1)^K}{K!}\int_X^Y B_K(\{x\})f^{(K)}(x)dx,
\end{equation}
where
\begin{equation}
S(K)=\sum_{k=1}^K\frac{(-1)^k}{k!}\left(B_k(\{Y\})f^{(k-1)}(Y)-B_k(\{X\})f^{(k-1)}(X)\right),
\end{equation}
and $B_k:[0,1]\rightarrow\mathbb{R}$ is the $k$-th Bernoulli polynomial.
\end{theorem}

The reader may refer to \cite[Appendix B]{MV07} for a proof of Theorem \ref{EMSF}.

\begin{corollary}\label{Zeta}
Let $X\geq 1$ be an arbitrary real number. Let $K$ be a positive integer. For every $s=\sigma+it\in\mathbb{C}$ such that $\sigma>1-K$ and $s\neq 1$, we have
\begin{align*}
\zeta(s)= & \ \sum_{n\leq X}\frac{1}{n^s}+\frac{X^{1-s}}{s-1}+\left(\{X\}-\frac{1}{2}\right)\frac{1}{X^s} \\
 & \ +\sum_{k=2}^K\frac{a_k(s)B_k(\{X\})}{k! X^{s+k-1}}-\frac{a_{K+1}(s)}{K!}\int_{X}^\infty \frac{B_K(\{X\})}{x^{s+K}}dx,
\end{align*}
where $a_k(s)=s(s+1)...(s+k-2)$ for $k\geq 2$.
\end{corollary}

For $\sigma>1$,
Corollary \ref{Zeta} is a direct application of Theorem \ref{EMSF} upon defining $f:[X,Y]\to\mathbb{C}$, as $x\mapsto x^{-s}$, $\Re(s)>1$, and letting $Y\to\infty$. We extend the statement to $\sigma>1-K$ by analytic continuation.

We consider Theorem \ref{EMSF} into a broader class of functions than $C^K$.
The following formulation (from \cite[\S 3.1]{Hel19}) improves slightly
on a constant value: it replaces the factor $\frac{1}{12}$,
coming from a direct application of Theorem \ref{EMSF} with $K=2$, by a
factor of $\frac{1}{16}$.

\begin{lemma}[\textbf{Improved Euler-Maclaurin summation formula of second order}]\label{le:EulerMc}
Let $f:[0,\infty)\to\mathbb{C}$ be a continuous, piecewise $C^1$ function such that $f$, $f'$, $f''$ are in $L^{1}([0,\infty))$. Then 
\begin{equation}\label{eq:EulerMc}
  \sum_{n=1}^{\infty}f(n)=\int_{0}^{\infty}f(x)dx -\frac{f(0)}{2} -
  \lim_{t\to 0^+}
  \frac{f'(t)}{16}+
  O^{*}\left(\frac{1}{16}\|f''\|_{1}\right).
\end{equation}
\end{lemma}
Here and elsewhere (for instance in Proposition~\ref{prop:hippogryph}), we mean $f''$ and $\|f''\|_1$ in the sense of
distributions or measures,
so that $\|f''\|_1$ stands for the total variation of the function $f'$
on the interval $\lbrack 0,\infty)$. If $f$ is in $C^2$, this is the same as the usual meaning.

\begin{proof}
As $f$ has bounded total variation, $f(x)$ converges to a real number
  $R$ as $x\to \infty$.
  If $R$ were non-zero, then $f$ could not be in $L^1$;
thus $\lim_{x\to\infty}f(x)=0$. By the same reasoning, since $f'$ is differentiable and $f',f''$ are in $L^{1}$, we have $\lim_{x\to\infty}f'(x)=0$.

Suppose first that $f'$ is continuous at the positive integers. Let $F(x)$ be a differentiable function with $F'(x)=x-\frac{1}{2}$. Then
$\int_{0}^{1}F'(x)dx=0$, $F(0)=F(1)$, and so, by integration by parts, 
\[\begin{aligned}
& \ \int_{n-1}^{n}f(x)dx = \frac{f(n)}{2}-\frac{f(n-1)}{2}-\int_{n-1}^{n} f'(x) F'(\{x\}) dx \\
= & \ \frac{f(n)}{2}-\frac{f(n-1)}{2}-(f'(n)-f'(n-1))F(0) +\int_{n-1}^{n}f''(x)F(\{x\})dx,
 \end{aligned}\]
where we write $f'(0)$ for $\lim_{t\to 0^+} f'(t)$.
Therefore, $\int_0^{n}f(x)dx$ equals
\begin{equation*}
\sum_{k=1}^{n}f(k)-\frac{f(n)}{2}+\frac{f(0)}{2}-f'(n)F(0)+f'(0)F(0)+\int_{0}^{n}f''(x)F(\{x\})dx.
\end{equation*}
By using the fact that $\lim_{n\to\infty}f(n)=\lim_{n\to\infty}f'(n)=0$, we obtain finally that
\begin{equation}\label{eq:EulerMc*}
\sum_{n=1}^{\infty}f(n)=\int_{0}^{\infty}f(x)dx-\frac{f(0)}{2}-f'(0)F(0)+O^{*}\left(\int_0^{\infty}|f''(x)||F(\{x\})|dx\right).
\end{equation}
It remains to choose $F$ with $F'(x)=x-\frac{1}{2}$ such that $\max_{x\in[0,1]}|F(x)|$ is minimal. We take
$F(x)=\frac{1}{2}\left(x^2-x+\frac{1}{8}\right)$, in which case
$\max_{x\in[0,1]}|F(x)|=\frac{1}{16}$. We obtain \eqref{eq:EulerMc}.

Finally, suppose that $f'$ not continuous at all the positive integers. Since there are countably many points in which $f$ is not differentiable,
there are only countably many $x\in \mathbb{R}$ such that $f$ is not differentiable at $n+x$ for at least one $n\in\mathbb{Z}_{>0}$. Thus, there is a sequence $\{\varepsilon_{k}\}_{k=1}^{\infty}$ with $\lim_{k\rightarrow\infty}\varepsilon_{k}=0$ such that the functions $f_{k}:x\mapsto f(x+\varepsilon_{k})$ are differentiable at all positive integers. Then, by the above,
\eqref{eq:EulerMc} holds for all of these
functions. Since $f'\in L^{1}([0,\infty))$, dominated convergence gives us that
 $\sum_{n=1}^\infty f_k(n) \to \sum_{n=1}^\infty f(n)$ as $k\to\infty$. It is clear that $\lim_{k\to\infty} \lim_{t\to 0^+} f_k'(t) =
    \lim_{k\to\infty} \lim_{t\to 0^+} f'(t+\varepsilon_k) =
    \lim_{t\to 0^+} f'(t)$, because the last limit exists.
     Obviously, $\int_0^\infty f_k(x) dx \to \int_0^\infty f(x) dx$ as $k\to \infty$
      and $\|f_k''\|_1\leq \|f''\|_1$ for all $k$.
    We let $k\to\infty$ and
    obtain that $f$ satisfies  \eqref{eq:EulerMc}.
\end{proof}

\subsection{The Mellin transform}\label{sub:Mellin}
Let $f:[0,\infty)\rightarrow\mathbb{C}$. Its Mellin transform is defined as $\mathcal{M}f(s)=\int_{0}^{\infty}f(x)x^{s-1}dx$ for all $s$ such that the integral
  converges absolutely. It is a Fourier transform up to changing variables, so a version of Plancherel's identity holds:
\begin{equation}\label{eq:Plancherel}
\int_{0}^{\infty}|f(x)|^{2}x^{2\sigma-1}dx=\frac{1}{2\pi}\int_{-\infty}^{\infty}|\mathcal{M}f(\sigma+it)|^{2}dt,
\end{equation}
provided that
$f(x) x^{\sigma-\frac{1}{2}}$ is in $L^2([0,\infty))$ and
  $f(x)  x^{\sigma-1}$ is in $L^1([0,\infty))$.
  
    For $f$ continuous and piecewise $C^1$, by integration by parts,
\begin{equation}\label{eq:mellinder}
\mathcal{M}f'(s)=-(s-1)\mathcal{M}f(s-1).
\end{equation}

In particular, we have that $\mathcal{M}\mathds{1}_{(0,a]}(s)=\frac{a^{s}}{s}$, where $\mathds{1}_{S}$ denotes the indicator function of a set $S$. Considering now $f(x)=\sum_{n=1}^{\infty}a_{n}\mathds{1}_{(0,1/n]}(x)$, where $A(s)=\sum_{n=1}^{\infty}\frac{a_{n}}{n^{s}}$ is a Dirichlet series converging in the half-plane $\{s\in\mathbb{C}|\Re(s)>\sigma_c\}$, we observe that
\begin{equation*}
\mathcal{M}f(s)=\sum_{n=1}^{\infty}\int_{0}^{\infty}a_{n}\mathds{1}_{(0,1/n]}(x)x^{s-1}dx=\sum_{n=1}^{\infty}a_{n}\int_{0}^{1/n}x^{s-1}dx=\frac{A(s)}{s},
\end{equation*}
in the set $\{s\in\mathbb{C}|\Re(s)>\max\{0,\sigma_c\}\}$. As the above holds for every Dirichlet series, we have, for the function $J(x)=\sum_{n=1}^{\infty}\mathds{1}_{(0,1/n]}(x)=\left\lfloor\frac{1}{x}\right\rfloor$, the equality
\begin{equation}\label{eq:Melchi1}
\mathcal{M}J(s)=\frac{\zeta(s)}{s},
\end{equation}
which is valid for the set $\{s\in\mathbb{C}|\Re(s)>1\}$. Moreover, for a general function $f$, the function $\tilde{f}:x\mapsto f(nx)$ has Mellin transform $\mathcal{M}\tilde{f}(s)=\frac{\mathcal{M}f(s)}{n^{s}}$ for all $s$ in the domain of definition of $\mathcal{M}f$. Thus, for every well-defined function $F(x)=\sum_{n=1}^{\infty}f(nx)$, by considering
\begin{equation*}
h(x)=\left\lfloor\frac{1}{x}\right\rfloor-F(x)=\sum_{n=1}^{\infty}\mathds{1}_{(0,1/n]}(x)-\sum_{n=1}^{\infty}f(nx),
\end{equation*}
\begin{flalign}\label{eq:MelwithL} \text{we obtain}
&& \mathcal{M}F(s)&=\mathcal{M}f(s)\zeta(s),&\\
\label{identity} &&\mathcal{M}h(s)&=\left(\frac{1}{s}-\mathcal{M}f(s)\right)\zeta(s),&
\end{flalign}
for all $s$ in the domain of definition of $\mathcal{M}f$ such that $\Re(s)>1$.

The following can be readily proved by induction.
\begin{lemma}\label{le:mellinpoly}
For every $a\in\mathbb{R}$, $j\in\mathbb{N}\cup\{0\}$ and $s\in\mathbb{C}$ such that $\Re(s)>0$, we have
\begin{equation}\label{eq:mellinpoly}
\mathcal{M}\left((a-x)^{j}\mathds{1}_{(0,a]}(x)\right)(s)=\frac{j!a^{s+j}}{s(s+1)\ldots(s+j)}.
\end{equation}
\end{lemma}

\subsection{The Gamma function}

The Gamma function $\Gamma$ is defined for all $s\in\mathbb{C}$ such that $\Re(s)>0$ as $\Gamma:s\mapsto\int_0^\infty t^{s-1}e^{-t}dt$. This function can be extended meromorphically to $\mathbb{C}$, with poles on the set $\{0,-1,-2,-3,\ldots\}$ and vanishing nowhere. Where well-defined, it satisfies the relationship $\Gamma(s+1)=s\Gamma(s)$. This function is closely related to the $\zeta$ function, by means of the functional equation, valid for all $s\in\mathbb{C}\backslash\{0,1\}$,
\begin{equation}\label{functional}\zeta(s)=2(2\pi)^{s-1}\sin\left(\frac{\pi s}{2}\right)\Gamma(1-s)\zeta(1-s).
\end{equation} 

\begin{theorem}[\textbf{Stirling's formula, explicit form}]\label{Stirling}Let $0<\theta<\pi$. Let $s\in\mathbb{C}\backslash(-\infty,0]$
  such that $|\arg(s)|\leq\pi-\theta$, where $\arg(s)$ is the principal argument of $s$. Then
\begin{equation*}
\Gamma(s)=\sqrt{2\pi}s^{s-\frac{1}{2}}e^{-s}\exp\left(O^*\left(\frac{F}{|s|}\right)\right),
\end{equation*}
where $F=F_\theta=\frac{1}{12\sin^2\left(\frac{\theta}{2}\right)}$.
\end{theorem}
\begin{proof} 
Since $\Gamma(s)$ has neither zeroes nor poles in
the simply connected domain $\mathbb{C}\backslash(-\infty,0]$,
  $\log\Gamma(s)$ is a well-defined
  analytic function on $\mathbb{C}\backslash(-\infty,0]$.
By \cite[Thm.~1.4.2]{AAR99} with $m=1$,
\begin{equation}\label{log}
\log\Gamma(s)=\frac{1}{2}\log(2\pi)+\left(s-\frac{1}{2}\right)\log s-s+\mu(s),
\end{equation}
where $\log$ is the principal branch of the logarithm defined on $\mathbb{C}\backslash(-\infty,0]$ and $\mu(s)=\frac{1}{12s}-\frac{1}{2}\int_0^\infty\frac{B_2(\{x\})}{(s+x)^2}dx$. Moreover, as explained in \cite[\S 2.4.4]{Re98}, $\mu$ can be expressed as a Gudermann series so that, for all $s\in\mathbb{C}\backslash(-\infty,0]$,
\begin{equation}\label{trig}|\mu(s)|\leq\frac{1}{12\cos^{2}\left(\frac{1}{2}\arg(s)\right)|s|}.
\end{equation}
Now, if $|\arg(s)|\leq\pi-\theta$ then $\cos\left(\frac{1}{2}\arg(s)\right)=\cos\left(\frac{1}{2}|\arg(s)|\right)\geq\cos\left(\frac{\pi-\theta}{2}\right)=\sin\left(\frac{\theta}{2}\right)$. Thus, upon exponentiating both sides of \eqref{log} and implementing the final bound for \eqref{trig}, we derive the result.
\end{proof}

\begin{corollary}[\textbf{Rapid decay of $\Gamma$ in non-negative vertical strips}]\label{decay} Let $T\geq 1$ and $\sigma\geq 0$. Then, for every complex number $s=\sigma+it$ such that $|t|\geq T$,
\begin{equation*}|\Gamma(\sigma+it)|=\sqrt{2\pi}|t|^{\sigma-\frac{1}{2}}e^{-\frac{\pi}{2}|t|}\exp\left(O^*\left(\frac{G_\sigma}{T}\right)\right),
\end{equation*}
where $G_\sigma=\frac{\sigma^3}{3}+\frac{\sigma^2}{2}\left|\sigma-\frac{1}{2}\right|+\frac{1}{6}$.
\end{corollary}
\begin{proof}

As $s$ is such that $|\arg(s)|\leq\frac{\pi}{2}$, we use Theorem \ref{Stirling} with $\theta=\frac{\pi}{2}$, and obtain
\begin{align*}
\Re(\log\Gamma(s))=&\ \frac{\log(2\pi)}{2}+
\left(\sigma-\frac{1}{2}\right) \frac{\log(\sigma^2+t^2)}{2} \\
 & \ -t\arg(\sigma+it)-\sigma+O^*\left(\frac{1}{6T}\right).
\end{align*}
As $\log(1+x)\leq x$ for all $x\geq 0$, we have that $\log(\sigma^2+t^2)=2\log|t|+O^*\left(\frac{\sigma^2}{T^2}\right)$. Furthermore, observe that $\arg(\sigma+it) =\arctan\left(\frac{t}{\sigma}\right)$, and that
\begin{align*}
\arctan\left(\frac{t}{\sigma}\right) & =\int_0^{\frac{t}{\sigma}}\frac{dx}{1+x^2} =\pm\frac{\pi}{2}-\int_0^{\frac{\sigma}{t}}(1+O^*(x^2))dx \\
 & =\pm\frac{\pi}{2}-\frac{\sigma}{t}+O^*\left(\frac{\sigma^3}{3t^3}\right),
\end{align*}
where the sign $\pm$ corresponds to the sign of $t$. Putting everything together, we obtain that $\Re(\log\Gamma(s))$ equals
\begin{equation*}
\frac{1}{2}\log(2\pi)+\left(\sigma-\frac{1}{2}\right)\log|t|-\frac{\pi |t|}{2}+O^*\left(\left(\frac{\sigma^3}{3}+\frac{\sigma^2}{2}\left|\sigma-\frac{1}{2}\right|\right)\frac{1}{T^2}+\frac{1}{6T}\right).
\end{equation*}
As $\frac{1}{T^2}\leq\frac{1}{T}$, the above error term can thus be compressed to $O^*\left(\frac{G_\sigma}{T}\right)$. By exponentiating the above equation, we obtain the result.
\end{proof}

\subsection{Bounds on some sums}\label{sub:rec}

\tiny
\begin{mysage}
# Constant in the lemma below.
g=RIF(euler_gamma)
lower_harm=2*((RIF(2)).log()+g-1)
\end{mysage}
\normalsize

\begin{lemma}\label{le:harmonic}
For any $X\geq 1$ we have
\begin{equation}\label{eq:harmonicbounds}
\log X+\gamma-\frac{c}{X}\leq\sum_{n\leq X}\frac{1}{n}\leq\log X+\gamma+\frac{1}{2X},
\end{equation}
where $c = 2 (\log 2 + \gamma-1)$ and
$\gamma=0.5772\dotsc$ is the Euler-Mascheroni constant.
\end{lemma}

The constant $c$ in the lower bound was pointed out in
\cite[Lemma~2.1]{RA17}.

\begin{proof}
By applying Theorem \ref{EMSF} with $K=2$ to the function $x\mapsto x^{-1}$, we obtain
\begin{equation}\label{eq:frommaclau}\sum_{n\leq X}\frac{1}{n}=\log X+\frac{7}{12}-\int_1^\infty\frac{B_2(\{x\})}{x^3}dx+R(X),
\end{equation}
where $R(X)=-\frac{B_1(\{X\})}{X}-\frac{B_2(\{X\})}{2X^2}+\int_X^\infty\frac{B_2(\{x\})}{x^3}dx$. By \eqref{BernoulliBounds}, $B_1$ and $B_2$ are bounded functions on $[0,1]$; hence, $R(X)=O\left(\frac{1}{X}\right)$ and the integral
in \eqref{eq:frommaclau} is convergent. We conclude that $\gamma$,
defined as
$\lim_{X\to\infty}\sum_{n\leq X}\frac{1}{n}-\log X$, equals
$\frac{7}{12}-\int_1^\infty\frac{B_2(\{x\})}{x^3} dx$. Therefore,
$\sum_{n\leq X} \frac{1}{n} = \log X +\gamma + R(X).$

Since $B_1(t) = t-\frac{1}{2}$, $B_2(t)=t^2-t+\frac{1}{6}$ and $\max|B_2(\{x\})|=\frac{1}{6}$, 
\begin{equation*}
  R(X)=\frac{1}{2X}
-\frac{\{X\}}{X}\left(1-\frac{1-\{X\}}{2X}\right)
  -\frac{1}{12X^2}+O^*\left(\frac{1}{12X^2}\right).
\end{equation*}
Since $1-(1-\{X\})/(2 X)\geq 0$, the upper bound in \eqref{eq:harmonicbounds}
follows immediately. We also obtain that $R(X)\geq -1/(2 X) -1/(6 X^2)$,
and so the lower bound in \eqref{eq:harmonicbounds} holds for
$X\geq 5$; we check it for $1\leq X \leq 5$ by hand.
\end{proof}

\begin{lemma}
\label{recip} 
Let $\alpha\in \mathbb{R}^+\setminus \{1\}$ and $X>0$. Then
\begin{equation*}
\zeta(\alpha)-\frac{1}{(\alpha-1)X^{\alpha-1}}-\frac{1}{2X^{\alpha}}\leq\sum_{n\leq X}\frac{1}{n^\alpha}\leq\zeta(\alpha)-\frac{1}{(\alpha-1)X^{\alpha-1}}+\frac{1}{X^{\alpha}}.
\end{equation*}
\end{lemma} 

\begin{proof}
By definition of $\zeta(s)$ for $\Re(s)>1$, and by analytic continuation
for $\Re(s)>0$,
\begin{equation}\label{EulerMac}
\zeta(s) - \frac{1}{(s-1) X^{s-1}} - \sum_{n\leq X} \frac{1}{n^s} =
\sum_{n=1}^\infty 
\left(\int_{n-1}^{n} \frac{dx}{(X+x)^s} - \frac{1}{(\lfloor X\rfloor+n)^s}
\right)\end{equation}
for $s\ne 1$. Set $s=\alpha$.
 Since $t\mapsto t^{-\alpha}$ is decreasing, the right side of \eqref{EulerMac} is at least
\[\sum_{n=1}^\infty \left(\frac{1}{(X+n)^\alpha} -\frac{1}{(X+n-1)^\alpha}\right)
= - \frac{1}{X^\alpha}.\]
By 
the convexity of
$t\mapsto t^{-\alpha}$ and $(\lfloor X\rfloor+n)^{-\alpha} \geq (X+n)^{-\alpha}$,
\[\int_{n-1}^{n}\!\! \frac{dx}{(X+x)^\alpha}
- \frac{1}{(\lfloor X\rfloor+n)^s}
\leq \frac{1}{2} 
\left(\frac{1}{(X+n-1)^\alpha} +  \frac{1}{(X+n)^\alpha}\right) -
 \frac{1}{(X+n)^\alpha}
.\]
Telescoping again, we see that the right side of \eqref{EulerMac} is at most $1/(2 X^\alpha)$.
\end{proof}

\tiny
\begin{mysage}
def Dfunction(C):
    if C>10000:
        C=10000
    D=RIF( 1/2 + sqrt(1+1/C^2)*( 1/12 + sqrt(1+4/C^2)*( 1/(72*sqrt(3)) + sqrt(1+4/C^2)*1/360 ) ) )
    return D
Dcake = Dfunction(1)
\end{mysage}
\normalsize

The non-explicit form of the lemma below is classical: see for instance \cite[Thm.~4.11]{Tit86b}.

\begin{lemma}\label{le:353}
Let $s=\sigma+it$. Suppose that $X\geq 1$, $s\neq 1$, $0<\sigma\leq 1$ and $|t|\leq X$. Then
\begin{equation*}
  \zeta(s)=\sum_{n\leq X}\frac{1}{n^{s}}+\frac{X^{1-s}}{s-1}+O^{*}\left(
  \frac{D}{X^{\sigma}}\right),
\end{equation*}
where $D=\sage{roundup(Dcake.upper(),5)}$. If we assume $X\geq C$ for some $C>1$,
we may use
\begin{equation*}
D=\frac{1}{2}+\left(\frac{1}{12}+\left(\frac{1}{72\sqrt{3}}+\frac{1}{360}\sqrt{1+\frac{9}{C^2}}\right)\sqrt{1+\frac{4}{C^2}}\right) \sqrt{1+\frac{1}{C^2}}.
\end{equation*}
\end{lemma}

\begin{proof} Write $m_k$ for $\max_{x\in [0,1]} |B_k(x)|$.
  By Corollary \ref{Zeta} with $K=4$,
  \[\begin{aligned}
 & \ \zeta(s) - \sum_{n\leq X} \frac{1}{n^s} + \frac{X^{1-s}}{s-1} \\
= & \ O^*\bigg(
\frac{1}{2X^\sigma} + \sum_{2\leq k\leq 4}\frac{|a_{k}(s)|m_{k}}{k!X^{\sigma+k-1}} + \frac{|a_{5}(s)|m_{4}}{24}
\left|\int_X^\infty \frac{dx}{x^{s+4}}\right|\bigg)
  \\
= & \ \frac{1}{X^\sigma}
O^*\left(\frac{1}{2} + \frac{|s| m_2}{2 X} + \frac{|s||s+1|m_{3}}{6X^{2}} + 2\cdot \frac{|s||s+1||s+2|m_{4}}{24X^{3}}
\right).\end{aligned}\]
We know that $m_2=\frac{1}{6}$, and, since $B_3(x) = x^3-\frac{3}{2} x^2 + \frac{x}{2}$ and $B_{4}(x)=x^{4}-2x^{3}+x^{2}-\frac{1}{30}$, also that $m_3 = \frac{1}{12 \sqrt{3}}$ and $m_{4}=\frac{1}{30}$.
By $X\geq |t|$ and $\sigma\in [0,1]$, we have
$|s+k|/X\leq \sqrt{1+(k+1)^{2}/X^2}$ for all $k\geq 0$. Substituting these values inside the error term above, we get the result.
\end{proof}

\subsection{Further results}

The following is an explicit {\em mean value estimate}.

\tiny
\begin{mysage}
E=RIF(2*pi*sqrt(1+2/3*sqrt(6/5)))
\end{mysage}
\normalsize

\begin{proposition}\label{pr:brudern}
For any $X,T>0$ and any sequence of complex numbers $\{a_{n}\}_{n=1}^{\infty}$,
\begin{equation*}
\int_0^{T}\bigg|\sum_{n\leq X}a_{n}n^{it}\bigg|^{2}dt=\left(T+\frac{E}{2}\right)\sum_{n\leq X}|a_{n}|^{2}+O^{*}\bigg(E\sum_{n\leq X}n|a_{n}|^{2}\bigg),
\end{equation*}
where $E$ can be chosen to be equal to $2\pi\sqrt{1+2/3\sqrt{6/5}}\leq\sage{roundup(E.upper(),5)}$.
\end{proposition}

\begin{proof}
We use the main theorem in \cite{Pr84}, which improves on \cite[Cor.~2]{MV74} (the theorem states $C=\frac{4}{3}$, which yields $E=\frac{8}{3}\pi$, but it is proved with a lower $C$ that yields our $E$). We apply it then as in \cite[Cor.~3]{MV74}, with a numerical improvement given by $\log^{-1}\left(\frac{n+1}{n}\right)<n+\frac{1}{2}$, proved directly by calculus. See also \cite[Satz 4.4.3]{Br95} for an older explicit result that used $15n$ instead of $\frac{8}{3}\pi\left(n+\frac{1}{2}\right)$.
\end{proof}

If $\{a_n\}_{n=1}^\infty$ is a real sequence then the error term factor may be improved to $\frac{E}{2}$. As pointed out in \cite[Lemma 6.5]{Ra16}, a
term cancels out, allowing us to gain a factor of $2$ inside the error term.

\tiny
\begin{mysage}
H=RIF(zeta(2)-1) #This H is the constant in the statement of {le:stieltjes}, the bound on zeta(s)-1/(s-1) for 1<s<2
\end{mysage}
\normalsize

\begin{lemma}\label{le:stieltjes}
For any $1<\sigma<2$ we have $
\frac{1}{\sigma-1}<\zeta(\sigma)<\frac{1}{\sigma-1}+\zeta(2)-1.
$
The lower bound holds also for $0<\sigma<1$.
\end{lemma}

See also \cite[Lemma~5.4]{Ra16} for a better upper bound than the above for $\sigma$ close to $1$.

\begin{proof}
The Laurent expansion of $\zeta$ is
\begin{equation}\label{laurent}
f(\sigma)=\zeta(\sigma)-\frac{1}{\sigma-1}=\sum_{n=0}^{\infty}\frac{(-1)^{n}\gamma_{n}}{n!}(\sigma-1)^{n}
\end{equation}
where the $\gamma_{n}$ are the Stieltjes constants. For the upper bound, it suffices to prove that $f'(\sigma)$ is positive for $\sigma\in(1,2)$, so that $f(\sigma)<f(2)$: one can use
\begin{equation*}
f'(\sigma)=\sum_{n=0}^{\infty}\frac{(-1)^{n+1}\gamma_{n+1}}{n!}(\sigma-1)^{n}>-\gamma_{1}-\sum_{n=1}^{\infty}\frac{|\gamma_{n+1}|}{n!},
\end{equation*}
compute the first $10$ constants directly and then use the bound $|\gamma_{n}|\leq\frac{n!}{2^{n+1}}$ (for $n\geq 1$) given by Lavrik in \cite[Lemma 4]{Lav76}, so that $\sum_{n=10}^{\infty}\frac{|\gamma_{n+1}|}{n!}\leq\frac{1}{2}\sum_{n=11}^{\infty}\frac{n}{2^{n}}<10^{-2}$. The lower bound is even simpler to obtain: in order to prove that $f(\sigma)>0$ for $0<\sigma<2$ and $\sigma\neq 1$, we compute directly $\gamma_{0}=\gamma$ and then we bound the absolute value of the rest of the series in \eqref{laurent} by using again Lavrik's estimations.
\end{proof}

\begin{lemma}\label{lem:handy}
Let $A,B\geq 0$. Then, for any $\rho>0$,
\begin{align*}
(A + B)^2 &\leq (1 + \rho) A^2 + \left(1 + \frac{1}{\rho}\right) B^2, \\
(A - B)^2 &\geq (1 - \rho) A^2 + \left(1 - \frac{1}{\rho}\right) B^2.
\end{align*}
\end{lemma}
Note that the inequalities are tight when $\rho = \frac{B}{A}$.
\begin{proof}
Expand the square. By the arithmetic-geometric mean inequality, $2 |A B| = 2 (\sqrt{\rho} |A|) \cdot \frac{|B|}{\sqrt{\rho}}
  \leq \rho A^2 + \frac{B^2}{\rho}$.
\end{proof}

\section{First approach: as in a mean value theorem}\label{Exp}

We will first bound (Proposition~\ref{prop:hippogryph}) the $L^2$ norm of
the function $t\mapsto \left((\sigma+i t)^{-1}-G(\sigma + i t)\right) \zeta(\sigma + i t)$,
where $G$ is the Mellin transform of a function
$g:\lbrack 0,\infty)\to \mathbb{R}$. Then we will choose $g$ so
that $G(\sigma + i t)$ is close to $0$ for $|t|\geq T$, while keeping the aforementioned $L^2$ bound small.

We will first give a general treatment for $g$ arbitrary (\S \ref{L2}).
It will turn out to be easy to choose a $g$ that is optimal within our general
statement (\S \ref{subs:optim}). However, that optimality will turn out to be
an artifact of the form of our general statement. We will be able to do better
(at least for $\sigma\geq 1/2$) by choosing a different $g$, whose transform $G$ we can compute explicitly (\S\ref{subs:better}). Our final estimates are as follows.

\tiny
\begin{mysage}
#For checks about the following constants, see at the end of the file.
kpar_main=round( 27.8821 ,ndigits=5)
k111_main=round( 0.15659 ,ndigits=5)
k112_main=round( 0.15655 ,ndigits=5)
k113_main=round( 0.00979 ,ndigits=5)
k114_main=round( 0.07407 ,ndigits=5)
k12x_main=round( 0.60031 ,ndigits=5)
c21x_main=round( 2.4476 ,ndigits=5)
c22x_main=round( 1.58493 ,ndigits=5)
k314_main=round( 0.11361 ,ndigits=5)
c30x_main=round( 0.39113 ,ndigits=5)
\end{mysage}
\normalsize

\begin{theorem}\label{thzeta}
Let $0<\sigma\leq 1$ and $T\geq T_{0}=\sage{TH}$. Then the integral
\begin{equation*}
\frac{1}{2\pi i}\left(\int_{\sigma-i\infty}^{\sigma-iT}+\int_{\sigma+iT}^{\sigma+i\infty}\right)\left|\frac{\zeta(s)}{s}\right|^{2}ds
\end{equation*}
is bounded by
\begin{align*}
& \frac{3\zeta(2\sigma)}{5T}+\left(\frac{c_{111}}{\sigma}+\frac{c_{112}}{2\sigma+1}+\frac{c_{113}}{\sigma+1}-\frac{c_{114}}{2\sigma-1}\right)\frac{1}{T^{2\sigma}}+\frac{c_{12*}}{T^{2\sigma+1}}, & & \text{if $\sigma>\frac{1}{2}$,} \\
& \frac{3\log T}{5T}+\frac{c_{21*}}{T}+\frac{c_{22*}}{T^{2}}, & & \text{if $\sigma=\frac{1}{2}$,} \\
& \left(\frac{c_{311}}{\sigma}+\frac{c_{312}}{2\sigma+1}+\frac{c_{313}}{\sigma+1}+\frac{c_{314}}{1-2\sigma}\right)\frac{1}{T^{2\sigma}}+\frac{c_{30*}\zeta(2\sigma)}{T}+\frac{c_{32*}}{T^{2\sigma+1}}, & & \text{if $\sigma<\frac{1}{2}$,}
\end{align*}
where
\begin{align*}
c_{11i}= & \ \kappa^{\sigma}\kappa_{11i} \ (i=1,2,3,4), & \kappa= & \ \sage{kpar_main}, & \kappa_{12*}= & \ \sage{k12x_main}, \\
c_{12*}= & \ \kappa^{\sigma}\kappa_{12*}, & \kappa_{111}= & \ \sage{k111_main}, & c_{21*}= & \ \sage{c21x_main}, \\
c_{31i}= & \ c_{11i} \ (i=1,2,3), & \kappa_{112}= & \ \sage{k112_main}, & c_{22*}= & \ \sage{c22x_main}, \\
c_{314}= & \ \kappa^{\sigma}\kappa_{314}, & \kappa_{113}= & \ \sage{k113_main}, & \kappa_{314}= & \ \sage{k314_main}, \\
c_{32*}= & \ c_{12*}, & \kappa_{114}= & \ \sage{k114_main}, & c_{30*}= & \ \sage{c30x_main}.
\end{align*}
\end{theorem}

We have chosen $T_0=\sage{TH}$ for simplicity. In actual fact,
$T_0=\sage{THTrue}$ is the least $T$ for which we are able to reach $\frac{3}{5}$ as a main term coefficient for $\sigma=\frac{1}{2}$.

\subsection{Basic estimate}\label{L2}

Let us first give a bound valid for a function $g$ that satisfies a number of general conditions.
The proof is in parts close to, and in fact inspired by, proofs of classical mean value theorems, such as \cite[Thm.~6.1]{Mo71} (see in particular the exposition
in \cite[Thm.~9.1]{IK04}).

There are differences all the same. First, in a mean value theorem, we typically 
work with a finite sum $\sum_{n\leq X} a_n n^{i t}$, and obtain a bound that contains a term proportional to $X$, whereas here
we work directly with $\zeta$ and thus with an infinite sum.

Secondly, the proof in \cite[Thm.~9.1]{IK04} (or \cite[Thm.~6.1]{Mo71}) majorizes the characteristic function of a vertical interval by a continuous function of compact support, and then uses the decay in the inverse Mellin transform to bound the contribution of off-diagonal terms. On the vertical line, we choose to work
with a function of the form $1- G(s) s$, where $G$ is the Mellin transform
of a function $g$ satisfying certain properties. As a consequence,
off-diagonal terms vanish, outside an initial interval $\lbrack 0,\delta\rbrack$
that makes a small contribution.

\begin{proposition}\label{prop:hippogryph}
  Let $g:\lbrack 0,\infty) \to \mathbb{R}$ be a continuous,
  piecewise $C^1$ function such that $g$ and $g'$ have bounded total variation.
  Assume that \textbf{(a)} $\int_0^\infty g(t) dt = 1$, \textbf{(b)}  $0\leq g(t)\leq 1$
  for all $t$, \textbf{(c)}  $g(t)=1$ for $0\leq t\leq 1-\delta$ and
  $g(t)=0$ for $t\geq 1+\delta$, where $0<\delta\leq\frac{1}{2}$, \textbf{(d)} 
$g(1+t)=1-g(1-t)$ for $0\leq t\leq\delta$.
  Let \begin{equation}
    I(\sigma) = \frac{1}{2\pi i}\int_{\sigma-i\infty}^{\sigma+i\infty}\left|\frac{1}{s}-G(s)\right|^{2}|\zeta(s)|^{2}ds,
  \end{equation}
where $G$ is the Mellin transform of $g$.
Then, for any $\sigma>0$,
 \begin{equation}\label{eq:dusino} I(\sigma)\leq
c(\sigma,\alpha)\cdot \delta^{2\sigma} + 2 \beta \delta\cdot \begin{cases}
  \zeta(2\sigma) - \frac{\delta^{2\sigma-1}}{2\sigma-1} +
  \frac{\delta^{2\sigma}}{1-\delta^2}
  &\text{if $\sigma \ne \frac{1}{2}$,}\\
  \log\left(\frac{1}{\delta}\right) + \gamma + \frac{\delta}{2(1-\delta^2)}
    &\text{if $\sigma=\frac{1}{2}$,} 
\end{cases}
   \end{equation}
 where $\alpha = \frac{\delta}{16} \int_0^\infty |g''(t)| dt$,
 $\beta = \frac{1}{\delta} \int_1^{1+\delta} |g(y)|^2 dy$
 and
 $c(\sigma,\alpha) = \frac{1}{8\sigma} + \frac{\alpha}{2\sigma+1} + \frac{\alpha^2}{2\sigma+2}$.
\end{proposition}

\begin{proof}
Since $g$ is bounded, $G(s)$ is well-defined when $\Re(s)>0$.
For $\Re(s)>1$, we know from \eqref{eq:MelwithL} that 
$G(s)\zeta(s)$ is the Mellin transform of the function $x\mapsto\sum_{n=1}^{\infty}g(nx)$ (well-defined by (c)) and from \eqref{eq:Melchi1} that $\frac{\zeta(s)}{s}$ is the Mellin transform of $x\mapsto\sum_{n=1}^{\infty}\mathds{1}_{(0,1/n]}(x)$.

Let
\begin{equation}\label{eq:hdef}
h(x)=\sum_{n=1}^{\infty}\left(\mathds{1}_{[0,1/n]}(x)-g(nx)\right)=\left\lfloor\frac{1}{x}\right\rfloor-\sum_{n=1}^{\infty}g(nx).
\end{equation}
Then
\begin{equation}\label{eq:Mh}
\mathcal{M}h(s)=\left(\frac{1}{s}-G(s)\right)\zeta(s),
\end{equation}
for $\Re(s)>1$. On one hand, by \eqref{eq:hsmallbound}, $h$ is bounded, and thus
$\mathcal{M}h(s)$ is well-defined for $\Re(s) >0$. On the other hand, by condition (a),
$G(1)=1$ and thus the right side of \eqref{eq:Mh} is holomorphic for $\Re(s) > 0$. Hence,
by analytic continuation, \eqref{eq:Mh} holds for $\Re(s) > 0$ and therefore, by \eqref{eq:Plancherel},
\begin{equation}\label{mell}
\frac{1}{2\pi i}\int_{\sigma-i\infty}^{\sigma+i\infty}\left|\frac{1}{s}-G(s)\right|^{2}|\zeta(s)|^{2}ds=\int_{0}^{\infty}|h(x)|^{2}x^{2\sigma-1}dx,
\end{equation}
for any $s\in\mathbb{C}$ with $\Re(s)>0$, provided that
the integral on the right side converges. Bounding the integral on the right will suffice to derive the result.

Let us first find an upper bound for the value of $|h(x)|$ to use for small values of $x$ (namely, $x\leq\delta$). Using Lemma~\ref{le:EulerMc} and recalling that $g(0)=1$, $g(1)=0$, we obtain that
\begin{align*}
\sum_{n=1}^{\infty}g(nx) 
 = & \ \frac{1}{x}\int_{0}^{\infty}g(t)dt-\frac{1}{2}+O^{*}\left(\frac{1}{16}\int_{0^{+}}^{\infty}|g''(tx)|x^{2}dt\right) \\
 = & \ \frac{1}{x}-\frac{1}{2}+O^{*}\left(\frac{x}{16}\int_{0^{+}}^{\infty}|g''(t)|dt\right).
\end{align*}
By putting the above equality inside \eqref{eq:hdef}, we obtain for any $x\geq 0$ that
\begin{align}
|h(x)|= \ \left|\left\lfloor\frac{1}{x}\right\rfloor-\frac{1}{x}+\frac{1}{2}+O^{*}\left(\frac{x}{16}\int_{0^{+}}^{\infty}|g''(t)|dt\right)\right|
 \leq \ \frac{1}{2}+\frac{x}{16}\int_{0^{+}}^{\infty}|g''(t)|dt, \label{eq:hsmallbound}
\end{align}
since $\left|\lfloor t\rfloor-t+\frac{1}{2}\right|\leq\frac{1}{2}$ for all $t\in\mathbb{R}$.

For $x>\delta$, we bound $h$ in another way; by its definition and condition (c)
\begin{align}
h(x) & = \sum_{nx\leq 1}(1-g(nx))-\sum_{1<nx\leq 1+\delta}g(nx) \nonumber \\
 & = \sum_{1-\delta\leq nx\leq 1 }(1-g(nx))-\sum_{1<nx\leq 1+\delta}g(nx). \label{eq:bound}
\end{align}
  When $x>2\delta$, there is at most one integer $n$ such that
  $n x \in [1-\delta,1+\delta]$, since
  $\frac{1+\delta}{x} - \frac{1-\delta}{x} = \frac{2 \delta}{x} < 1$.
  For the same reason, when $\delta < x\leq 2\delta$, there
  can be at most one integer $n$ (call it $n_{0,x}$) such that
  $n x \in [1-\delta,1]$ and at most one integer $n$
  (call it $n_{1,x}$) such that
  $n x \in [1,1+\delta]$.
  Since $0\leq g(t)\leq 1$ for all $t$, we know that $1-g(n x)\geq 0$ and
  $- g(n x)\leq 0$, and so the last two sums in \eqref{eq:bound}
  have opposite sign. Hence
$|h(x)| \leq \max\left\{|1-g(n_{0,x} x)|,|g(n_{1,x})|\right\}$.
It follows that
\begin{align}
\int_0^\infty |h(x)|^2 x^{2\sigma-1} dx \leq & \ \int_0^{\delta} |h(x)|^{2} x^{2\sigma-1} dx \nonumber \\
 & \ + \sum_{n\leq \frac{1}{\delta}} \int_{\max\left\{\frac{1-\delta}{n},\delta\right\}}^{\frac{1}{n}} |1-g(n x)|^2 x^{2\sigma-1} dx \nonumber \\
 & \ + \sum_{n\leq \frac{1+\delta}{\delta}} \int_{\max\left\{\frac{1}{n},\delta\right\}}^{\frac{1+\delta}{n}} |g(n x)|^2 x^{2\sigma-1} dx. \label{eq:usuru}
\end{align}
 Setting $y= n x$ and changing the order of summation,
 we get
\begin{equation*}
\sum_{n\leq \frac{1}{\delta}}
 \int_{\max\left\{\frac{1-\delta}{n},\delta\right\}}^{\frac{1}{n}}
 |1-g(n x)|^2 x^{2\sigma-1} dx
 = 
 \int_{1-\delta}^1 \left(\sum_{n\leq \frac{y}{\delta}} \frac{1}{n^{2\sigma}}\right)
 |1-g(y)|^2 y^{2\sigma -1} dy,
\end{equation*}
 and, similarly,
\begin{equation*}
\sum_{n\leq \frac{1+\delta}{\delta}} \int_{\max\left\{\frac{1}{n},\delta\right\}}^{\frac{1+\delta}{n}} |g(n x)|^2 x^{2\sigma-1} dx = 
 \int_1^{1+\delta} \left(\sum_{n\leq \frac{y}{\delta}} \frac{1}{n^{2\sigma}}\right)
 |g(y)|^2 y^{2\sigma -1} dy.
\end{equation*}

Using \eqref{eq:hsmallbound} in the first integral on the right hand side of \eqref{eq:usuru}, we obtain
\begin{equation}\label{hxshort}
\int_{0}^{\delta}|h(x)|^{2}x^{2\sigma-1}dx\leq \delta^{2\sigma}\cdot\left(\frac{1}{8\sigma}+\frac{\alpha}{2\sigma+1}+\frac{\alpha^2}{2(\sigma+1)}\right),
\end{equation}
where $\alpha=\alpha_{g,\delta}=\frac{\delta}{16}\int_{0^{+}}^{\infty}|g''(t)|dt$.

As for the remaining terms, we just use the bounds
\begin{equation*}
\sum_{n\leq x} \frac{1}{n^{2\sigma}}
\leq \begin{cases}
  \zeta(2\sigma) + \frac{x^{1-2 \sigma}}{1- 2\sigma} + x^{-2 \sigma} 
  & \text{if $\sigma\ne\frac{1}{2}$,}\\
  \log x + \gamma + \frac{1}{2 x}
  & \text{if $\sigma=\frac{1}{2}$,}
  \end{cases}
\end{equation*} 
which we obtain from Lemmas~\ref{le:harmonic} and \ref{recip}, valid for $x\geq 1$ (for $\delta\leq\frac{1}{2}$ and $y\geq 1-\delta$ we certainly have
$\frac{y}{\delta}\geq 1$). Thus, the second and third terms on the right side of \eqref{eq:usuru} add
up to at most
\begin{align}
& \zeta(2\sigma)\left(\int_{1-\delta}^{1}|1-g(y)|^{2}y^{2\sigma-1}dy+\int_{1}^{1+\delta}|g(y)|^{2}y^{2\sigma-1}dy\right) \nonumber \\
& + \frac{\delta^{2\sigma-1}}{1 - 2\sigma}\left(\int_{1-\delta}^{1}|1-g(y)|^{2}dy+\int_{1}^{1+\delta}|g(y)|^{2}dy\right) \nonumber \\
& +\delta^{2\sigma}\left(\int_{1-\delta}^{1}|1-g(y)|^{2}\frac{dy}{y}+\int_{1}^{1+\delta}|g(y)|^{2}\frac{dy}{y}\right), \label{hxlong1}
\end{align}
if $0<\sigma\leq 1$ with $\sigma\ne\frac{1}{2}$, and
\begin{align}
& \int_{1-\delta}^{1}|1-g(y)|^{2}\log\left(\frac{y}{\delta}\right)dy+\int_{1}^{1+\delta}|g(y)|^{2}\log\left(\frac{y}{\delta}\right)dy \nonumber \\
& +\gamma\left(\int_{1-\delta}^{1}|1-g(y)|^{2}dy+\int_{1}^{1+\delta}|g(y)|^{2}dy\right) \nonumber \\
& +\frac{\delta}{2}\left(\int_{1-\delta}^{1}|1-g(y)|^{2}\frac{dy}{y}+\int_{1}^{1+\delta}|g(y)|^{2}\frac{dy}{y}\right), \label{hxlong2} 
\end{align} 
if $\sigma=\frac{1}{2}$.

When $\frac{1}{2}\leq\sigma\leq 1$, as the functions $f(y)=y^{2\sigma-1}$ and $f(y)=\log\left(\frac{y}{\delta}\right)$ are concave, we have by  condition (d) that
\begin{align}
 & \ \int_{1-\delta}^{1}|1-g(y)|^{2}f(y)dy+\int_{1}^{1+\delta}|g(y)|^{2}f(y)dy \nonumber \\
 = & \ \int_{1}^{1+\delta}|g(y)|^{2}(f(2-y)+f(y))dy \leq 2f(1)\int_{1}^{1+\delta}|g(y)|^{2}dy. \label{con}
\end{align}
In the first line of \eqref{hxlong1}, if $\sigma<\frac{1}{2}$, as $\zeta(2\sigma)<0$ and $f(y) = y^{2\sigma - 1}$ is convex, we employ the following  {\em lower} bound
\begin{equation}\label{eq:conv}
  \int_{1-\delta}^{1}|1-g(y)|^{2}f(y)dy+\int_{1}^{1+\delta}|g(y)|^{2}f(y)dy \geq
2f(1)\int_{1}^{1+\delta}|g(y)|^{2}dy. \end{equation}
To estimate the integrals in \eqref{hxlong1}, \eqref{hxlong2} that have
$\frac{dy}{y}$ in the integrand, we just use the
fact that $y\mapsto y^{-1}$ is convex, so that for all $0\leq t\leq \delta$, $(1-t)^{-1} + (1+t)^{-1} \leq (1-\delta)^{-1} + (1+\delta)^{-1}=
\frac{2}{1-\delta^2}$.
 
Consider now $\beta=\beta_{g,\delta}=\frac{1}{\delta}\int_{1}^{1+\delta}|g(y)|^{2}dy$. Putting together \eqref{hxshort}, the cases \eqref{hxlong1} and \eqref{hxlong2} and the estimates \eqref{con} and \eqref{eq:conv}, we finally obtain the following upper bounds for $\int_{0}^{\infty}|h(x)|^{2}x^{2\sigma-1}dx$:
\label{eq:doneboundL} 
\begin{equation}\label{eqqq}
2\beta\left(\delta\zeta(2\sigma)-\frac{\delta^{2\sigma}}{2\sigma -1}+
  \frac{\delta^{2\sigma+1}}{1-\delta^2}\right)
  +\delta^{2\sigma}\left(\frac{1}{8\sigma}+\frac{\alpha}{2\sigma+1}+\frac{\alpha^2}{2 (\sigma+1)}\right),
\end{equation}
if $0<\sigma\leq 1$ with $\sigma\ne \frac{1}{2}$, and, if $\sigma=\frac{1}{2}$,
\begin{equation}\label{eqqqq}
  2\beta\left(\delta\log\left(\frac{1}{\delta}\right)+
  \gamma \delta+\frac{\delta^2}{2(1-\delta^2)}\right)+\delta\left(\frac{1}{4}+\frac{\alpha}{2}+\frac{\alpha^2}{3}\right).
\end{equation}
\end{proof}

Note that for $0<\sigma<\frac{1}{2}$ the leading term in \eqref{eqqq} is of order $\delta^{2\sigma}$, as then $\zeta(2\sigma)<0$. The bound for $\sigma=\frac{1}{2}$ is what results from \eqref{eqqq} if we let $\sigma\to \frac{1}{2}^-$ or $\sigma\to \frac{1}{2}^+$.

{\bf Remarks.} Note that $\frac{\zeta(s)}{s}$ is the Mellin transform of $x\mapsto\left\lfloor 1/x\right\rfloor$. What we are doing is substract
an approximation $f(x)$ to
$\left\lfloor 1/x\right\rfloor$ such that the difference
$h(x)=\left\lfloor\frac{1}{x}\right\rfloor-f(x)$ has a well-defined Mellin transform throughout $\Re(s)>0$. Then the Mellin transform acts as an isometry throughout that region, and so, for $\Re(s)>0$,
\begin{equation}\label{eq:udururu} 
\frac{1}{2\pi i}\int_{\sigma-i\infty}^{\sigma + i \infty} \left|\frac{\zeta(s)}{s} - F(s)\right|^2 ds
\end{equation}
equals the $L^2$ norm of $h(x) x^{\sigma-\frac{1}{2}}$ on
$\lbrack 0,\infty)$.

In the proof of Proposition~\ref{prop:hippogryph}, we take $f(x) = \sum_{n=1}^\infty g(n x)$ with
$g$ continuous. Then $F(s) = G(s) \zeta(s)$. We need $G$ close to $\frac{1}{s}$ for $|\Im(s)|\leq T$ and close to $0$ for $|\Im(s)| > T$, but not too close or else $g$ would have slow decay, and $f$ would approximate $\left\lfloor\frac{1}{x}\right\rfloor$ poorly. This tension between two sources of error can be seen as reflecting the uncertainty principle.

Our requirement that $g$ be compactly supported is somewhat restrictive, but greatly simplifies the proof of Proposition~\ref{prop:hippogryph}: for $x\geq 2 \delta$, the sum $f(x) = \sum_{n=1}^\infty g(n x)$ contains only one term, and so does its square.

\subsection{An ``optimal'' choice of $g$}\label{subs:optim}

What we want is to bound the integral $\frac{1}{2\pi i}\left(\int_{\sigma-i\infty}^{\sigma-iT}+\int_{\sigma+iT}^{\sigma+i\infty}\right)\left|\frac{\zeta(s)}{s}\right|^{2}ds$,
which is at most
\begin{equation}\label{atmost}\frac{I(\sigma)}{\inf_{|\Im(s)|\geq T}|1-G(s)s|^2}\end{equation}
where $I(\sigma)$ and $G(s)$ are as in Proposition~\ref{prop:hippogryph} and $\Re(s) = \sigma$.

Proposition~\ref{prop:hippogryph} gives us a bound on
$I(\sigma)$, while 
\begin{equation}\label{eq:fight1}
  \inf_{|\Im(s)|\geq T}|1-G(s)s|
  \geq 1-\sup_{|\Im(s)|\geq T}\frac{|G(s)s(s+1)|}{T}\geq 1-\frac{1}{T}\int_{0}^{\infty}|g''(x)|x^{\sigma+1}dx,
\end{equation}
where the second inequality comes from applying \eqref{eq:mellinder} twice.
(Proceeding in this way seems natural, since we already estimated a quantity in terms of $g''$ in Proposition~\ref{prop:hippogryph}. It will later turn out later
that we are losing enough in this step to make the result we would obtain in
this section worse than the one we will get in \S \ref{subs:better}.)

From the conditions on $g$ in Proposition~\ref{prop:hippogryph}
we have $g''=0$ outside $[1-\delta,1+\delta]$ and $g''(1+x)=-g''(1-x)$ for $x\in[0,\delta]$. Since
$x\mapsto x^{\sigma+1}$ is convex in $x$ for $\sigma\geq 0$
and $(1+\delta)^{\sigma+1}+(1-\delta)^{\sigma+1}$ is increasing in $\sigma\geq 0$,
we see that $(1+x)^{\sigma+1} + (1-x)^{\sigma+1} \leq (1+\delta)^2 + (1-\delta)^2 =
2 + 2 \delta^2$,
and so
\begin{equation}\label{eq:fight2}
  1-\frac{1}{T}\int_{0}^{\infty}|g''(x)|x^{\sigma+1}dx \geq
1 - \frac{1+\delta^2}{T} |g''|_1.\end{equation}

We focus only on the main terms in the bound of $I(\sigma)$ given in Proposition~\ref{prop:hippogryph}. Introduce an auxiliary function $\eta:[0,\infty)\rightarrow\mathbb{R}$ defined so that $g(1+x)=\frac{1}{2}\eta\left(\frac{x}{\delta}\right),g(1-x)=1-\frac{1}{2}\eta\left(\frac{x}{\delta}\right)$.
  We then have $\beta=\frac{1}{4}|\eta|_{2}^{2}$, $|g''|_{1} =
  \frac{1}{\delta}|\eta''|_{1}$, and so
  $\alpha=\frac{1}{16}|\eta''|_{1}$. The main terms for $\delta$ small are
\begin{align*}
 & \frac{|\eta|_{2}^{2}\zeta(2\sigma)\delta}{2\left(1-\frac{1}{\delta T}|\eta''|_{1}\right)^{2}} \ \ \ \ \text{for $\frac{1}{2}<\sigma\leq 1$,} & & \frac{|\eta|_{2}^{2}\delta\log\left(\frac{1}{\delta}\right)}{2\left(1-\frac{1}{\delta T}|\eta''|_{1}\right)^{2}} \ \ \ \ \text{for $\sigma=\frac{1}{2}$,} \\
& \frac{\left(\frac{|\eta|_{2}^{2}}{2(1-2\sigma)}+\frac{1}{8\sigma}+\frac{|\eta''|_{1}}{16 (2\sigma+1)}+\frac{|\eta''|_{1}^{2}}{512 (\sigma+1)}\right)\delta^{2\sigma}}{\left(1-\frac{1}{\delta T}|\eta''|_{1}\right)^{2}} & & \text{for $0<\sigma<\frac{1}{2}$.}
\end{align*}
The term $\frac{1}{8\sigma}$ in the case $0<\sigma<\frac{1}{2}$ is not unexpected, as the integral of Theorem~\ref{thzeta} diverges at $\sigma=0$. We will choose $\delta$ so as to minimize the main terms above. For $\frac{1}{2}<\sigma\leq 1$, the minimum of $\frac{x}{(1-ax^{-1})^{2}}$ is at $x=3a$. Therefore we let $\delta=3|\eta''|_{1}T^{-1}$ so that the main term becomes
\begin{equation*}
\frac{3\zeta(2\sigma)}{2 T\left(1-\frac{1}{3}\right)^{2}} |\eta''|_{1}|\eta|_{2}^{2}.
\end{equation*}
For $\sigma=\frac{1}{2}$ we let $\delta = 3 |\eta''|_1T^{-1}$, out of simplicity. 
Then $\log\left(\frac{1}{\delta}\right) = \log T+\log\left(\frac{2}{3|\eta''|_{1}}\right)$, the term with $\log T$, which will be the main term in $T$, contributing
\begin{equation*}
\frac{3\log T}{2T\left(1-\frac{1}{3}\right)^{2}}|\eta''|_{1}|\eta|_{2}^{2}.
\end{equation*}
For $0<\sigma<\frac{1}{2}$, the minimum of $\frac{x^{2\sigma}}{(1-ax^{-1})^{2}}$
is reached at
$x=\left(1+\frac{1}{\sigma}\right)a$, so that we can
choose $\delta=\left(1+\frac{1}{\sigma}\right) |\eta''|_{1}T^{-1}$. The main term in this case is at most
\begin{equation}\label{eq:huantar}
  \frac{\left(1+\frac{1}{\sigma}\right)^{2\sigma}
  }{T^{2\sigma}\left(1-\frac{1}{1+\frac{1}{\sigma}}\right)^{2}}
  \left(\frac{|\eta''|_{1}^{2\sigma}|\eta|_{2}^{2}}{2(1-2\sigma)}+\frac{|\eta''|_{1}^{2\sigma}}{8\sigma}+\frac{|\eta''|_{1}^{2\sigma+1}}{16 (2\sigma+1)}+\frac{|\eta''|_{1}^{2\sigma+2}}{512 (\sigma+1)}\right).
\end{equation}

In all cases, we conclude that we have to select $\eta$ so that the factor $|\eta''|_{1}|\eta|_{2}^{2}$
(or, for $0<\sigma<\frac{1}{2}$, the first term in \eqref{eq:huantar}) is minimal.

\begin{lemma}\label{pr:optimal}
  Let $\eta:[0,\infty)\to\mathbb{R}$ be a decreasing continuous function, continuously differentiable outside a finite number of points, such that $\eta(0)=1$ and $\eta(x)=0$ for all $x\geq 1$. Then there exist $x_0\in(0,1]$ and a function
    $\eta_{x_0}:[0,\infty)\to\mathbb{R}$ of the form
      \[\eta_{x_0}(x) = \begin{cases} 1 - \frac{x}{x_0} &
        \text{for $0\leq x<x_0$,}\\
        0 & \text{for $x\geq x_0$,}\end{cases}\]
such that $|\eta_{x_0}''|_1\leq |\eta''|_1$ and $|\eta_{x_0}|_2\leq |\eta|_2$.
\end{lemma}

\begin{proof} 
If $|\eta''|_1=\infty$, we just take $\eta_{x_0}$ with $x_0>0$ sufficiently small so that $|\eta_{x_0}|_2\leq |\eta|_2$. Otherwise, suppose that $\eta'$ is of bounded variation; then one-sided limits of $\eta'$ always exist.

Since $|\eta'(0^{+})|<\infty$, it is clear that there is a $x_1>0$ such that
$\eta_{x_1}(t)\leq \eta(t)$ for all $t$ in some interval $[0,\delta]$,
$\delta>0$. Since $\eta$ is decreasing and $\eta(x)=0$, $\eta$ is non-negative
on $[0,1]$. Hence, for $x_2=\min(x_1,1/\delta)$, we know that
$\eta_{x_2}(x)\leq \eta(x)$ for all $x\in [0,1]$.
Let $x_{0}$ be the largest element of $[0,1]$
such that $\eta_{x_0}(x)\leq\eta(x)$ for all $x\in[0,1]$; clearly, $x_0\geq x_2>0$.
We readily see that $|\eta_{x_0}|_2\leq |\eta|_2$, so it is sufficient to prove that $|\eta_{x_0}''|_1\leq|\eta''|_1$. 

Clearly, $\eta_{x_0}'(0^+)\leq\eta'(0^+)$. Suppose first that $\eta'_{x_0}(0^{+})=\eta'(0^{+})$. By construction, we have $|\eta_{x_0}''|_1=\frac{1}{x_{0}}=|\eta'_{x_0}(0^{+})|$; furthermore, the total variation $|\eta''|_1$ of $\eta'$ is at least $|\eta'(0^{+})-\eta'(2)|$, which is equal to $|\eta'(0^{+})|$, since $\eta'(2)=0$. Therefore $|\eta_{x_0}''|_1=|\eta'_{x_0}(0^{+})|=|\eta(0^{+})|\leq|\eta''|_1$. Now suppose instead that $\eta'_{x_0}(0^{+})<\eta'(0^{+})$. Let $c\in(0,x_{0}]$ such that $\eta_{x_0}(c)=\eta(c)$, which must exist by definition of $x_{0}$: since $\eta$ is continuous and $\eta_{x_0}(x)\leq\eta(x)$ for all $x\in(0,c)$, there must be some $c'\in(0,c)$ with $\eta'(c')\leq\eta'_{x_{0}}(c')=-\frac{1}{x_{0}}$, and as before we have $|\eta''|_1\geq|\eta'(c')-\eta'(2)|\geq\frac{1}{x_{0}}$, concluding the proof.
\end{proof}

Thanks to Lemma \ref{pr:optimal}, we can assume that $\eta(x)$ is simply
the function given by $\eta(x)=1-x$ for $0\leq x\leq 1$, and by $\eta(x)=0$
for $x\geq 1$; the other functions $\eta_0$ described in the statement of
Lemma \ref{pr:optimal} are just dilations of this one, and can thus be covered
by the fact that we can choose $\delta$ as we wish. 

\begin{corollary}[to Proposition~\ref{prop:hippogryph}]\label{co:quitegood}
Let $0<\sigma\leq 1$, $T>\max\left\{3,1+\frac{1}{\sigma}\right\}$. Then
\[  \frac{1}{2\pi i}\!\left(\int_{\sigma-i\infty}^{\sigma-iT}\!\!\!\!+\int_{\sigma+iT}^{\sigma+i\infty}\right)\!\left|\frac{\zeta(s)}{s}\right|^{2}\!\!\!ds \leq
\rho_{\sigma,T} \cdot
\begin{cases} \frac{\zeta(2\sigma)}{2 T} +  
  \frac{c_{0}(3) - c_{1}(3)
  }{T^{2\sigma}}
& \!\text{if $\sigma>\frac{1}{2}$,}\\  
\frac{\log T}{2 T} + \frac{c_{0}(3) - c_{2}}{T}
& \!\text{if $\sigma=\frac{1}{2}$,}\\ \frac{c_{0}\left(1+\frac{1}{\sigma}\right) \ - \ c_{1}\left(1+\frac{1}{\sigma}\right)}{T^{2\sigma}} + 
  \frac{c_{3}}{T}
& \!\text{if $\sigma<\frac{1}{2}$,}\end{cases}\]
where
\begin{align*}
c_{0}(\kappa) & = \kappa^{2\sigma}\left(c' + \frac{\kappa}{6T(1-\frac{\kappa^2}{T^2})}\right),\ 
\ & \!\!\!\!c_{1}(\kappa) & = \frac{\kappa^{2\sigma}}{6 (2\sigma-1)},\\
c_{2} & = \frac{\log 3- \gamma}{2},
 \ \  & c_{3} & = \frac{(\sigma+1) \zeta(2\sigma)}{6\sigma},  
\\
c' & = \frac{1}{8\sigma} + \frac{1}{16(2\sigma+1)} + \frac{1}{512(\sigma+1)}, \ \  & \rho_{\sigma,T} & = 
  \begin{cases} \frac{9}{
  4 \left(1 - \frac{9}{2T^2}\right)^{2}} &\!\text{if $\sigma\geq\frac{1}{2}$,}\\
   \frac{(1+\sigma)^2}{\left(1 - \frac{(1+\sigma)^2}{\sigma T^2}\right)^2}&\!\text{if $\sigma<\frac{1}{2}$.}\end{cases}
\end{align*}
\end{corollary}

Notice that $c_{0}(\kappa)$ and $c_3$ go to $\infty$ when $\sigma\to 0$. Observe also that the numerical optimization in \S\ref{subs:better}, on which Theorem~\ref{thzeta} is based, yields results that are asymptotically stronger than the ones above only for $\sigma\geq\frac{1}{2}$: the main coefficient of Corollary \ref{co:quitegood} turns out to be better when $\sigma>0$ is close to $0$, starting from around $\sigma=0.044$, although not yet reaching the asymptotically correct value proved later in Theorem~\ref{th:z/s-all}.

\begin{proof}
  As per the discussion above, we let
  \[g(t)= \begin{cases}1 & \text{for $0<t\leq 1 - \delta$,}\\
   \frac{1}{2} - \frac{t-1}{2 \delta} & \text{for $1-\delta\leq t\leq 1+\delta$,}\\
0 & \text{ for $t> 1 + \delta$.}\end{cases}\]
It is clear that $|g''|_1 = \frac{1}{\delta}$; hence,
$\alpha = \frac{1}{16}$ and $\beta = \frac{1}{12}$, for $\alpha$ and $\beta$ as in
  the statement of Proposition~\ref{prop:hippogryph}.
  We let $\delta = \frac{3}{T}$ if $\frac{1}{2}\leq \sigma\leq 1$ and
  $\delta = \frac{1+\sigma^{-1}}{T}$ if $0<\sigma<\frac{1}{2}$.
We bound $\inf_{|\Im(s)|\geq T}|1-G(s)s|$ from below 
by \eqref{eq:fight1} and \eqref{eq:fight2}. Then apply
Proposition~\ref{prop:hippogryph}.
\end{proof}

We will not use Corollary \ref{co:quitegood} in our main results.

\subsection{A better choice of $g$ for $\Re(s)\in\left[\frac{1}{2},1\right]$}\label{subs:better}

The choice of $g$ in \S \ref{subs:optim} is optimal only once we commit
ourselves to bounding $|1-G(s) s|$ as in \eqref{eq:fight1}. Alternatively,
we can choose $g$ from a class of functions whose Mellin transforms
$G(s)$ we can compute explicitly. We can then optimize $g$ within that class.
Consider, for instance,
$g:[0,\infty)\rightarrow\mathbb{R}$ such that $g$ is given by a polynomial
  in the interval $\lbrack 1-\delta,1+\delta\rbrack$, where the transition
  from $1$ to $0$ occurs.
  So that the conditions in Proposition~\ref{prop:hippogryph} are fulfilled,
  we ask for $g$ with $g(x)=1$ for $x<1-\delta$, $g(x)=0$ for $x>1+\delta$, and
\begin{equation}\label{eq:ggeneral}
g(x)=
\frac{1}{2}+\sum_{k=0}^{n}a_{k}\frac{(1+\delta-x)^{k}(1-x)(1-\delta-x)^{k}}{\delta^{2k+1}}\;\;\;\;\; \mathrm{ if \ }1-\delta\leq x\leq 1+\delta
\end{equation}
for some appropriate parameters $n$, $\delta$ and a sequence $\{a_{k}\}_{k=0}^{n}$. This choice in turn will allow us to give the Mellin transform of such $g$ explicitly, according to Lemma~\ref{le:mellinpoly}.

\begin{lemma}\label{le:gisgood}
Let $g:[0,\infty)\rightarrow\mathbb{R}$ be a function of the form \eqref{eq:ggeneral}. Suppose that \textbf{(a)} $a_{0}=\frac{1}{2}$ and $a_{1}=-\frac{1}{4}$,
\textbf{(b)} for every $0\leq k\leq n$ the coefficient $a_{k}$ has sign $(-1)^{k}$,
\textbf{(c)} for every $0\leq k<n$ we have $|a_{k+1}|\leq\frac{2k+1}{2k+2}|a_{k}|$.
Then $g$ is continuously differentiable on $(0,\infty)$ and $0\leq g\leq 1$ everywhere.
\end{lemma}

\begin{proof}
  Each of the three pieces in which $g$ is split by \eqref{eq:ggeneral} is continuously differentiable, so we just have to check the property for the points
  $1-\delta$ and $1+\delta$. We have $g(1\pm\delta)=\frac{1}{2}\mp a_{0}$ and setting $a_{0}=\frac{1}{2}$ makes it so that $g(1-\delta)=1$, $g(1+\delta)=0$, implying the continuity of $g$. Supposing that $a_{0}=\frac{1}{2}$, we also obtain $\lim_{x\rightarrow\delta^{-}}g'(1\pm x)=-\frac{1}{2\delta}-\frac{2a_{1}}{\delta}$ and having $a_{1}=-\frac{1}{4}$ makes it so that this limit becomes $0$, thus giving us the continuity of the first derivative for $g$.

To prove that $0\leq g\leq 1$ in the interval $[1-\delta,1+\delta]$, it is sufficient to show that $g'(x)\leq 0$ in that interval. If we substitute $\varepsilon=1-x$, we have
\begin{align*}
 & \ g'(x)=- \frac{1}{2\delta}-\sum_{k=1}^{n}a_{k}\frac{(\varepsilon^{2}-\delta^{2})^{k-1}}{\delta^{2k+1}}((2k+1)(\varepsilon^{2}-\delta^{2})+2k\delta^{2}) \\
= & \ -\sum_{k=1}^{n-1}\frac{(\varepsilon^{2}-\delta^{2})^{k}}{\delta^{2k+1}}((2k+1)a_{k}+(2k+2)a_{k+1})-a_{n}\frac{(\varepsilon^{2}-\delta^{2})^{n}}{\delta^{2n+1}}(2n+1).
\end{align*}
Since we are working in $[1-\delta,1+\delta]$ we have $\varepsilon^{2}-\delta^{2}\leq 0$. To ensure that the product $a_{n}(\varepsilon^{2}-\delta^{2})^{n}$ in the last term is not negative, it is sufficient to ask for $a_{n}$ to have sign $(-1)^{n}$. We can now proceed backwards by induction on the terms in the sum. Indeed, supposing that $(-1)^{k+1}a_{k+1}\geq 0$, in order to have $(\varepsilon^{2}-\delta^{2})^{k}((2k+1)a_{k}+(2k+2)a_{k+1})\geq 0$ it is enough to ask that $(-1)^{k}a_{k}\geq 0$ and $(2k+1)|a_{k}|\geq (2k+2)|a_{k+1}|$.
\end{proof}

Computing the parameter $\beta$ in  Proposition~\ref{prop:hippogryph} is routine.
\begin{lemma}\label{le:gsquare}
Let $g:[0,\infty)\rightarrow\mathbb{R}$ be a function of the form \eqref{eq:ggeneral} such that $a_{0}=\frac{1}{2}$. Define $\beta=\frac{1}{\delta}\int_{1}^{1+\delta}|g(x)|^{2}dx$. Then
\begin{equation*}
\beta=\sum_{i=0}^{4n+2}\frac{(-1)^{i}}{i+1}\sum_{l=0}^{i}b_{n,l}b_{n,i-l},
\end{equation*}
where
\begin{equation*}
b_{n,j}=\sum_{k=0}^{n}2^{2k-j}\left(2\binom{k}{j-k-1}+\binom{k}{j-k}\right)a_{k}
\end{equation*}
for $1\leq j\leq 2n+1$ and $b_{n,j}=0$ for $j=0$ or $j>2n+1$.
\end{lemma}

\begin{proof}
  We substitute $y=1+\delta-x$ inside the definition of $g(x)$. Then, for
  $1-\delta\leq x\leq 1+\delta$,
\begin{align*}
 & \ g(x)=\frac{1}{2}+\sum_{k=0}^{n}a_{k}\frac{y^{k}(y-\delta)(y-2\delta)^{k}}{\delta^{2k+1}} \\
= & \ \frac{1}{2}+\sum_{k=0}^{n}\sum_{i=0}^{k}\left((-1)^{k-i}\binom{k}{i}\frac{a_{k}2^{k-i}}{\delta^{k+i+1}}y^{k+i+1}-(-1)^{k-i}\binom{k}{i}\frac{a_{k}2^{k-i}}{\delta^{k+i}}y^{k+i}\right).
\end{align*}
Inside the sums, we substitute $j=k+i+1$ in the first term and $j=k+i$ in the second term, we shift one summation symbol outside, with the new index $j$, and we uniformize the range of each of the inner sums. We obtain
\begin{equation*}
g(x)=\frac{1}{2}+\sum_{j=0}^{2n+1}(-1)^{j+1}\frac{y^{j}}{\delta^{j}}\sum_{k=0}^{n}2^{2k-j}\left(2\binom{k}{j-k-1}+\binom{k}{j-k}\right)a_{k}.
\end{equation*}

For $1\leq j\leq 2n+1$, we just define $b_{n,j}$ to be as in the statement. For $j=0$, we include in the definition of $b_{n,0}$ the term $\frac{1}{2}$ that was outside the sums, so that $b_{n,0}=\frac{1}{2}-\sum_{k=0}^{n}2^{2k}\left(2\binom{k}{-k-1}+\binom{k}{-k}\right)a_{k}=\frac{1}{2}-a_{0}=0$. Therefore
\begin{equation}\label{eq:gxasy}
g(x)=\sum_{j=0}^{2n+1}(-1)^{j+1}b_{n,j}\frac{y^{j}}{\delta^{j}}.
\end{equation}
Imposing also $b_{n,j}=0$ for $j>2n+1$, we finally get
\begin{equation*}
\int_{1}^{1+\delta}|g(x)|^{2}dx =\sum_{i=0}^{4n+2}\left(\sum_{l=0}^{i}(-1)^{i}b_{n,l}b_{n,i-l}\right)\frac{\delta^{i+1}}{(i+1)\delta^{i}},
\end{equation*}
which gives $\beta$.
\end{proof}

In order to choose $\delta$ and $g$ optimally, we need to detect first what to minimize.

\begin{proposition}\label{pr:whattomin}
If $0<\sigma\leq 1$, then $\frac{1}{2\pi i}\left(\int_{\sigma-i\infty}^{\sigma-iT}+\int_{\sigma+iT}^{\sigma+i\infty}\right)\left|\frac{\zeta(s)}{s}\right|^{2}ds$ is bounded from above by quantities whose main terms are
\begin{align}
& \frac{2\zeta(2\sigma)r\sum_{i=0}^{4n+2}\frac{(-1)^{i}}{i+1}\sum_{l=0}^{i}b_{n,l}b_{n,i-l}}{\left(1-\sum_{j=1}^{2n+1}\frac{2j!|b_{n,j}|}{r^{j}}\right)^{2}}\cdot\frac{1}{T} & \text{if $\sigma>\frac{1}{2}$,} \nonumber\\
& \frac{2r\sum_{i=0}^{4n+2}\frac{(-1)^{i}}{i+1}\sum_{l=0}^{i}b_{n,l}b_{n,i-l}}{\left(1-\sum_{j=1}^{2n+1}\frac{2j!|b_{n,j}|}{r^{j}}\right)^{2}}\cdot\frac{\log T}{T} & \text{if $\sigma=\frac{1}{2}$,} \nonumber\\
& \label{mainterms}\frac{r^{2\sigma}\!\left(\frac{1}{8\sigma}\!+\!\frac{\alpha}{2\sigma+1}\!+\!\frac{\alpha^{2}}{2\sigma+2}\!+\!\frac{2}{1-2\sigma}\sum_{i=0}^{4n+2}\frac{(-1)^{i}}{i+1}\sum_{l=0}^{i}b_{n,l}b_{n,i-l}\right)}{\left(1-\sum_{j=1}^{2n+1}\frac{2j!|b_{n,j}|}{r^{j}}\right)^{2}}\cdot\frac{1}{T^{2\sigma}} & \text{if $\sigma<\frac{1}{2}$,}
\end{align}
where $g$ is any polynomial as in \eqref{eq:ggeneral}, for any choice of $(n,r,\{a_{k}\}_{k=0}^{n})$ such that $0<r\leq\frac{T}{2}$, $n\geq 1$, $\{a_{k}\}_{k=0}^{n}$ satisfies the conditions of Lemma~\ref{le:gisgood}, the $b_{n,j}$ are defined as in Lemma~\ref{le:gsquare}, $\alpha$ is defined as in Proposition~\ref{prop:hippogryph}, and the expression inside the square in the denominator is positive.
\end{proposition}
\begin{proof}
 Recall inequality \eqref{atmost}. By Lemma \ref{le:gisgood}, all the conditions are met so that we can derive a bound (depending on $\delta$) for its numerator $I(\sigma)$ as given in Proposition~\ref{prop:hippogryph}.

  Let us concentrate on its denominator. For $x\in [1-\delta,1+\delta]$, we write $g(x)$ as in \eqref{eq:gxasy},
  where $y=1+\delta-x$.   We proceed similarly for $z=1-\delta-x$. Observe that, since $g(x)=0$ for all $x>1+\delta$ and $g=1$ in $[0,1-\delta]$,
\begin{equation}\label{eq:gxasyz}
g=\sum_{j=0}^{2n+1}(-1)^{j+1}b_{n,j}\frac{y^{j}}{\delta^{j}}\mathds{1}_{[0,1+\delta]}-\sum_{j=1}^{2n+1}b_{n,j}\frac{z^{j}}{\delta^{j}}\mathds{1}_{[0,1-\delta]},
\end{equation}
where the $b_{n,j}$ are as in Lemma~\ref{le:gsquare}. Now, $g$ is written as linear combination of expressions as in \eqref{eq:mellinpoly} with $a=1\pm\delta$, and, by Lemma~\ref{le:mellinpoly}, its Mellin transform  is
\begin{equation*}
G(s)=\sum_{j=1}^{2n+1}\frac{j!b_{n,j}((-1)^{j+1}(1+\delta)^{s+j}-(1-\delta)^{s+j})}{\delta^{j}s(s+1)...(s+j)}.
\end{equation*}

Furthermore, we have $|s+1|,\ldots,|s+j|>|\Im(s)|^j$, and $\sigma+j\leq j+1$ implies that $|(1+\delta)^{s+j}\pm(1-\delta)^{s+j}|\leq(1+\delta)^{j+1}+(1-\delta)^{j+1}$, since the left hand side is an increasing function of $\sigma$. These two facts imply that
\begin{equation}\label{eq:Gfirstbound}
\inf_{|\Im(s)|\geq T}|1-G(s)s|\geq 1-\sum_{j=1}^{2n+1}\frac{j!|b_{n,j}|}{\delta^{j}T^{j}}\sum_{i=0}^{\left\lfloor\frac{j+1}{2}\right\rfloor}2\binom{j+1}{2i}\delta^{2i}.
\end{equation}

We want $\delta$ to be small, so 
as to keep the upper bound in \eqref{eq:dusino} small, but not too small, since
we want the expression on the right of \eqref{eq:Gfirstbound} to be positive.

The terms $\delta^{j}T^{j}$ in \eqref{eq:Gfirstbound} tell us that we cannot afford more than taking $\delta=\frac{r}{T}$, which we choose, for some $0<r\leq\frac{T}{2}$ large enough (depending only on $n$) to make the right hand side of \eqref{eq:Gfirstbound} positive. Therefore, all conditions requested in the above paragraph hold. Let $D_{\min}$ be the square of the expression on the right of \eqref{eq:Gfirstbound}, so that
$\inf_{|\Im(s)|\geq T}|1-G(s)s|^{2} \geq D_{\min}$.

Now, the substitution $\delta=\frac{r}{T}$ in the bounds \eqref{eq:dusino} makes evident that the obtained main terms, as $T\rightarrow\infty$, are of order $\frac{1}{T}$, $\frac{\log T}{T}$, $\frac{1}{T^{2\sigma}}$ for $\frac{1}{2}<\sigma\leq 1$, $\sigma=\frac{1}{2}$, $0<\sigma<\frac{1}{2}$, respectively. Moreover, thanks to the definitions of $\alpha$, $\beta$, implemented for a function $g$ of the form \eqref{eq:ggeneral}, it is the choice of $\{a_{k}\}_{k=0}^n$ and of $r$ that will determine the optimal constants in front of these main terms.

We derive the result once we put everything together and set aside the summands of order $\frac{1}{T^{2i}}$ that come from the inner sum defining $\sqrt{D_{\min}}$.
\end{proof}

\tiny
\begin{mysage}
n=6
A=[1/2,-1/4,3/16,-533639/10^7,81112/10^7,-7415/10^7,370/10^7]
def compute_coefficients(n,A): #computes the coefficients b_{n,j}
    B=[0]*(4*n+3)
    j=1
    while j<=2*n+1:
        k=0
        while k<=n:
            B[j]=B[j]+(2^(2*k-j))*(2*binomial(k,j-k-1)+binomial(k,j-k))*A[k]
            k=k+1
        j=j+1
    return B
def compute_beta(n,B): #computes beta
    beta=0
    j=0
    while j<=4*n+2:
        k=0
        while k<=j:
            beta=beta+(-1)^j/(j+1)*B[k]*B[j-k]
            k=k+1
        j=j+1
    return beta
B=compute_coefficients(n,A)
beta=compute_beta(n,B)
rvalue=QQ(5.28035)
def compute_sumexpr(n,B): #computes the expression inside Dmin and Dmax, in the variables r,T
    var('r,T')
    sumexpr=0
    j=1
    while j<=2*n+1:
        s=0
        i=0
        while i<=j+1:
            s=s+binomial(j+1,i)*r^(i)/T^(i)
            i=i+2
        sumexpr=sumexpr+2*factorial(j)*abs(B[j])/r^(j)*s
        j=j+1
    return sumexpr
def compute_denomin(n,B): #computes Dmin, in the variables r,T
    var('r,T')
    denomin=(1-compute_sumexpr(n,B))^2
    return denomin
def compute_denomax(n,B): #computes Dmax, in the variables r,T
    var('r,T')
    denomax=(1+compute_sumexpr(n,B))^2
    return denomax
Dmin=RIF(compute_denomin(n,B).substitute(r=rvalue,T=TH))
Dmax=RIF(compute_denomax(n,B).substitute(r=rvalue,T=TH))
def g_function(n,delta,A): #defines g
    var('x')
    g=1/2
    i=0
    while i<=n:
        g=g+A[i]*(1+delta-x)^(i)*(1-x)*(1-delta-x)^(i)/delta^(2*i+1)
        i=i+1
    return g
delta=RIF(rvalue/TH)
alpha=RIF(2*delta*abs(diff(g_function(n,delta,A),x,1).substitute(x=1))/16)
kpar=RIF(rvalue^2)
k111=RIF(1/(8*Dmin))
k112=RIF(alpha/Dmin)
k113=RIF(alpha^2/(2*Dmin))
k114=RIF(2*beta*(1-10^(-5)/kpar)/Dmax)
k12x=RIF(2*beta*rvalue/((1-rvalue^2/TH^2)*Dmin)) #
c21x=RIF(rvalue/(4*Dmin)+alpha*rvalue/(2*Dmin)+alpha^2*rvalue/(3*Dmin)+2*beta*g*rvalue/Dmin)
c22x=RIF(beta*rvalue^2/((1-rvalue^2/TH^2)*Dmin))
c30x=RIF(2*beta*rvalue/Dmax)
k314=RIF(2*beta/Dmin)
\end{mysage}
\normalsize

\begin{proofbold}{Theorem~\ref{thzeta}}
First we bound $\frac{1}{2\pi i}\left(\int_{\sigma-i\infty}^{\sigma-iT}+\int_{\sigma+iT}^{\sigma+i\infty}\right)\left|\frac{\zeta(s)}{s}\right|^{2}ds$ as in Proposition~\ref{pr:whattomin}. As aforementioned, it is the choice of $n$, $a_k$ ($0\leq k\leq n$) and $r$ that suffices to optimize those main terms in each case. For simplicity, we will carry out the optimization process and the corresponding choice of parameters according to \eqref{mainterms} only for $\sigma\geq\frac{1}{2}$, the same choice being used for the remaining cases. 

For $n=2,3$, we determine by computer all possibilities for coefficients of $g$ satisfying the conditions in Lemma~\ref{le:gisgood} with precision $10^{-n-1}$. We then proceed inductively for larger $n$; given an optimized $g=g_n$ for a certain $n$, a better $g=g_{n+1}$ with $n+1$ is found as follows: start with the set of coefficients provided by the original $g$, attaching $a_{n}=0$ as a new variable, and compute the first bound in \eqref{mainterms}, for any fixed $\frac{1}{2}<\sigma\leq 1$ (in fact, $\sigma$ does not participate in our analysis), by adding $\vec{x}$ to the tuple $\vec{a}=(a_{2},\ldots,a_{n})$ ($a_0$, $a_1$ being fixed) for every $\vec{x}\in(\{0,\pm 10^{-n-1}\})^{n-1}$ such the conditions of Lemma~\ref{le:gisgood} hold. We thus determine an optimal $\vec{x}$, call it $\vec{x}_*$, and compute the first bound in \eqref{mainterms} with $\vec{a}+j\vec{x}_*$, $j\geq 1$, as long as we encounter improvements, until we stop and consider the last tuple $\vec{a}_*=\vec{a}+j\vec{x}_*$, that produces an improvement on \eqref{mainterms} (meaning that $\vec{a}+(j+1)\vec{x}_*$ does not). We repeat the described process starting with $\vec{a}_*$ rather than $\vec{a}$ until we find an optimized set of coefficients $a_2,\ldots,a_n$ for which no increment $\vec{x}$ produces any improvement; this final $(a_{2},\ldots,a_{n})$ will define $g_{n+1}$.

By taking $n=\sage{n}$, our parameters are
\begin{align}
a_{0}= & \ \sage{A[0]}, & a_{1}= & \ \sage{A[1]}, & a_{3}= & \ \sage{A[3]}, & a_{5}= & \ \sage{A[5]}, \nonumber \\
\phantom{a_{0}=} & \phantom{\ \sage{A[0]},} & a_{2}= & \ \sage{A[2]}, & a_{4}= & \ \sage{A[4]}, & a_{6}= & \ \sage{A[6]}, \label{eq:gparam}
\end{align}
$r=\sage{roundup(rvalue,5)}$ and $T\geq T_{0}=\sage{TH}$.

Consider $D_{\min}$ and let
\begin{equation*}
  D_{\max} = \left(1+\sum_{j=1}^{2n+1}\frac{2j!|b_{n,j}|}{r^{j}}\left(1+\binom{j+1}{2}\frac{r^{2}}{T^{2}}+...\right)\right)^{2},
\end{equation*}
so that, recalling again Proposition \ref{pr:whattomin}, $\sup_{|\Im(s)|\geq T}|1-G(s)s|^{2}\leq D_{\max}$. 

Given the choice in \eqref{eq:gparam}, we have
\begin{align*}
\alpha= & \ \sage{rounddown(alpha.upper(),5)}..., & \beta= & \ \sage{beta}, \\
D_{\min}> & \ \sage{rounddown(Dmin.lower(),5)}, & D_{\max}< & \ \sage{roundup(Dmax.upper(),5)}.
\end{align*}
Hence, the coefficient of the leading term $\frac{1}{T}$ in the case of $\frac{1}{2}<\sigma\leq 1$ becomes
\begin{equation}\label{eq:3/5}
\frac{2\beta\zeta(2\sigma)r}{D_{\min}}, \ \ \ \ \text{with} \ \ \ \ \frac{2\beta r}{D_{\min}}<\sage{roundup(RIF(2*beta*rvalue/Dmin).upper(),5)}\leq\frac{3}{5},
\end{equation}
and the coefficients of the smaller terms $\frac{1}{T^{2\sigma}},\frac{1}{T^{2\sigma+1}}$ are bounded as follows
\begin{align*}
\kappa:= & \ \sage{roundup(kpar.upper(),5)}\in r^2+[0,10^{-5}], \;\;\;\;
c_{111}:= \kappa^{\sigma}\kappa_{111}:=\kappa^{\sigma}\cdot \sage{roundup(k111.upper(),5)}>\frac{r^{2\sigma}}{8D_{\min}}, \\
c_{112}:= & \ \kappa^{\sigma}\kappa_{112}:=\kappa^{\sigma}\cdot \sage{roundup(k112.upper(),5)}>\frac{\alpha r^{2\sigma}}{D_{\min}}, \\
c_{113}:= & \kappa^{\sigma}\kappa_{113}:=\kappa^{\sigma}\cdot \sage{roundup(k113.upper(),5)}>\frac{\alpha^{2}r^{2\sigma}}{2D_{\min}}, \\
c_{114}:= & \ \kappa^{\sigma}\kappa_{114}:=\kappa^{\sigma}\cdot \sage{rounddown(k114.lower(),5)}<\kappa^{\sigma}\cdot\frac{2\beta\left(1-\frac{10^{-5}}{\kappa}\right)}{D_{\max}}\leq\frac{2\beta r^{2\sigma}}{D_{\max}},\\
c_{12*}:= &\kappa^{\sigma}\kappa_{12*}:=\kappa^{\sigma}\cdot\sage{roundup(k12x.upper(),5)}>\frac{2\beta r^{2\sigma+1}}{\left(1-\frac{r^{2}}{T^{2}}\right)D_{\min}}, 
\end{align*} 
where the numbers $c_{ijk}$ are the ones given in the statement.

In the case of $\sigma=\frac{1}{2}$, the coefficient of the leading term $\frac{\log T}{T}$ is, as in \eqref{eq:3/5}, bounded by $\frac{3}{5}$, while the lower order terms $\frac{1}{T},\frac{1}{T^{2}}$ have their coefficients bounded as follows
\begin{align*}
c_{21*} & :=\sage{roundup(c21x.upper(),5)}>\frac{r}{D_{\min}}\left(2\beta\gamma+\frac{1}{4}+\frac{\alpha}{2}+\frac{\alpha^{2}}{3}\right), \\
c_{22*} & :=\sage{roundup(c22x.upper(),5)}>\frac{\beta r^{2}}{\left(1-\frac{r^{2}}{T^{2}}\right)D_{\min}}.
\end{align*}
Finally, in the case of $0<\sigma<\frac{1}{2}$, the coefficients are bounded in the same way as in the case of $\frac{1}{2}<\sigma\leq 1$, with the exception of
\begin{equation*}
c_{314}:=\kappa^{\sigma}\kappa_{314}:=\kappa^{\sigma}\cdot\sage{roundup(k314.upper(),5)}>\frac{2\beta r^{2\sigma}}{D_{\min}}, \ \ \ \ \ \ c_{30*}:=\sage{rounddown(c30x.lower(),5)}<\frac{2\beta r}{D_{\max}}.
\end{equation*}
\end{proofbold}

{\bf Remarks.} 
The coefficient $\frac{3}{5}=0.6$ appearing in the case $\frac{1}{2}\leq\sigma\leq 1$ is an artificial threshold that the authors have set, $n=\sage{n}$ being the smallest value for which it could be reached for some choice of parameters $a_{k}$. These parameters, together with $r$ and $T_{0}$, were then determined by our choice of threshold and $n$ through computer calculations, as already described during the proof.

The chosen threshold could have been improved by choosing a larger $n$ than $n=\sage{n}$, albeit very slightly; computer investigations up to $n=9$ did not manage to give less than $0.596$. Nevertheless, the correct value in that case, as given in Theorem~\ref{thm:main} and suggested for example by the asymptotics in Theorems~7.2 and 7.3 in \cite{Tit86b}, should have been $\frac{1}{\pi}=\sage{rounddown(1/pi,5)}...$ .

In \S\ref{Opt} we obtain such a coefficient. However, for small values of $T$, the estimations in Theorem~\ref{thzeta}
coming from our work in this section 
are better, whence its importance.

\section{Second approach: Euler-Maclaurin and a standard mean value theorem}\label{Opt}

Rather than working directly with $\zeta$ as in \S\ref{Exp}, we work with its $L^2$ mean through a finite truncation, as given in Lemma \ref{le:353}. We will thus obtain not only bounds of the integral of $t\mapsto\left|\frac{\zeta{(\sigma+it)}}{\sigma+it}\right|^2$ on the tails but also mean square asymptotic expressions for $\zeta$.

\subsection{General bounds}\label{general} 

We start by providing bounds for the integral of $|\zeta(s)|^{2}$ with general extrema. We follow two similar paths, according to whether in Lemma~\ref{le:353} the index $X$ of the sum is chosen to be a constant (as in Proposition~\ref{prop:z-all}) or dependent on $t$ (as in Proposition~\ref{prop:z-allb}): the two choices are advantageous in different situations, as observed in the next subsections.

\begin{proposition}\label{prop:z-all}
Let $\frac{1}{2}\leq\sigma\leq 1$ and $T_{1},T_{2}$ be real numbers such that $1\leq T_{1}\leq T_{2}$. 
Then, for any $\rho>0$,
$\int_{T_{1}}^{T_{2}}|\zeta(\sigma+it)|^{2}dt$ is at most
\begin{equation}\label{eq:z-all}\begin{aligned}
  &(1 + \rho) \left(\left(T_{2}-T_{1}+\frac{E}{2}\right) f_{1,1}^{+}(\sigma,T_2)
  + Ef_{1,2}^{+}(\sigma,T_2)\right) 
  \\ 
  + &\left(1 + \frac{1}{\rho}\right)
  \left(T_2^{2-2\sigma}\left(\frac{1}{T_{1}}-\frac{1}{T_{2}}\right)+\frac{D^2(T_2-T_1)}{T_2^{2\sigma}} 
  +2DT_2^{1-2\sigma}\log\left(\frac{T_{2}}{T_{1}}\right)\right)\end{aligned}\end{equation}
where
\begin{align*}
f_{1,1}^{+}(\sigma,T) = & \begin{cases} \log T+ \gamma + \frac{1}{2 T}
  & \text{if $\sigma=\frac{1}{2}$,}\\
  \zeta(2\sigma)-\frac{1}{(2\sigma-1)T^{2\sigma-1}}+\frac{1}{T^{2\sigma}}  
& \text{if $\frac{1}{2} < \sigma \leq 1$,}\end{cases} \\
f_{1,2}^{+}(\sigma,T) = & \begin{cases}
  \frac{T^{2-2 \sigma}}{2(1-\sigma)} +\frac{1}{2}
  & \text{if $\frac{1}{2} \leq \sigma < 1$,}\\
  \log T+\gamma+\frac{1}{2T} & \text{if $\sigma=1$,}\end{cases} 
\end{align*}
and the constants $D$ and $E$ are as in Lemma~\ref{le:353}, with $C=\lfloor T_{2}\rfloor$, and as in Proposition~\ref{pr:brudern}, respectively. Moreover, for any $-1<\rho<0$, $\int_{T_{1}}^{T_{2}}|\zeta(\sigma+it)|^{2}dt$ is bounded from below by the expression in \eqref{eq:z-all} where $f_{1,1}^{+},f_{1,2}^{+}$ are replaced respectively by
\begin{align*}
f_{1,1}^{-}(\sigma,T) = & \begin{cases} \log T+ \gamma - \frac{c}{T}
  & \text{if $\sigma=\frac{1}{2}$,}\\
  \zeta(2\sigma)-\frac{1}{(2\sigma-1)T^{2\sigma-1}}-\frac{1}{2T^{2\sigma}} 
& \text{if $\frac{1}{2} < \sigma \leq 1$,}\end{cases} \\
f_{1,2}^{-}(\sigma,T) = & \begin{cases}
  \frac{T^{2-2 \sigma}}{2 (1-\sigma)}+\zeta(2\sigma-1)-\frac{1}{2T^{2\sigma-1}}
  & \text{if $\frac{1}{2} \leq \sigma < 1$,}\\
  \log T+\gamma-\frac{c}{T} & \text{if $\sigma=1$,}\end{cases}
\end{align*}
where $c$ is as in Lemma~\ref{le:harmonic}.
\end{proposition}

\begin{proof}
Let $1\leq T_{1}\leq T_{2}$. By Lemma~\ref{le:353}, for any $X\geq T_{2}$ we have
\begin{align}
\int_{T_1}^{T_2}|\zeta(\sigma+it)|^2dt & \leq\int_{T_1}^{T_2} (|Z(t)|+|R(t)|)^2dt, \label{eq:cs} \\
Z(t)=\sum_{n\leq X}\frac{1}{n^{s}}, & \ \ \ \ \ \ \ \ \ \ R(t)=\left|\frac{X^{1-s}}{s-1}\right|+\frac{D}{X^{\sigma}}, \nonumber
\end{align}
where $s = \sigma + i t$. We also obtain a lower bound for the expression above by writing $|Z(t)|-|R(t)|\leq |\zeta(\sigma+it)|$. Hence, by Lemma \ref{lem:handy}, for any $\rho>0$,
\[\int_{T_1}^{T_2} (|Z(t)|+|R(t)|)^2dt
\leq (1 + \rho) \int_{T_1}^{T_2} |Z(t)|^2 dt
+ \left(1 + \frac{1}{\rho}\right)  \int_{T_1}^{T_2} |R(t)|^2 dt,\]
and for any $-1<\rho<0$,
\[\int_{T_1}^{T_2} (|Z(t)|-|R(t)|)^2dt
\geq (1 + \rho) \int_{T_1}^{T_2} |Z(t)|^2 dt
+ \left(1 + \frac{1}{\rho}\right)  \int_{T_1}^{T_2} |R(t)|^2 dt.\]
Applying Proposition~\ref{pr:brudern} with $T=T_{2}-T_{1}$ and $a_{n}=\frac{1}{n^{\sigma+iT_{1}}}$, we see that
\begin{equation}\label{eq:ls}
\int_{T_{1}}^{T_{2}}|Z(t)|^{2}dt=\left(T_{2}-T_{1}+\frac{E}{2}\right)\sum_{n\leq X}|a_{n}|^{2}+O^{*}\left(E \sum_{n\leq X} n |a_{n}|^{2}\right),
\end{equation}
If $\sigma=\frac{1}{2}$, we use Lemma~\ref{le:harmonic} for the first term and $\sum_{n\leq X}1=\lfloor X\rfloor\leq X$ for the second. If $\frac{1}{2}<\sigma<1$ we use Lemma~\ref{recip} for both terms and the inequality $\zeta(2\sigma-1)+\frac{1}{X^{2\sigma-1}}<\zeta(0)+1=\frac{1}{2}$. 
 If $\sigma=1$ we use Lemma~\ref{recip} for the first and Lemma~\ref{le:harmonic} for the second. This analysis gives the following upper bounds for $\int_{T_1}^{T_2}|Z(t)|^2$:
\begin{equation*}
\left(T_{2}-T_{1}+\frac{E}{2}\right)\left(\log X+\gamma+\frac{1}{2X}\right)+EX
\end{equation*}
if $\sigma=\frac{1}{2}$,
\begin{equation*}
\left(T_{2}-T_{1}+\frac{E}{2}\right)\left(\zeta(2\sigma)-\frac{1}{(2\sigma-1)X^{2\sigma-1}}+\frac{1}{X^{2\sigma}}\right)+\frac{EX^{2-2\sigma}}{2(1-\sigma)}+\frac{E}{2} \end{equation*}
if $\frac{1}{2}<\sigma<1$, and
\begin{equation*}
\left(T_{2}-T_{1}+\frac{E}{2}\right)\left(\zeta(2)-\frac{1}{X}+\frac{1}{X^{2}}\right)+E\left(\log X+\gamma+\frac{1}{2X}\right) 
\end{equation*}
if $\sigma=1$. Analogous lower bounds can be deduced respectively, using the same lemmas.

As for the second term in \eqref{eq:cs},
\begin{equation}\label{eq:R}
\int_{T_{1}}^{T_{2}}|R(t)|^{2}dt=\int_{T_{1}}^{T_{2}}\left|\frac{X^{1-s}}{s-1}\right|^2+\frac{D^{2}}{X^{2\sigma}}+\frac{2D}{X^{\sigma}}\left|\frac{X^{1-s}}{s-1}\right|dt. 
\end{equation}
Thanks to our condition $\rho>-1$ for the lower bound, and as we want non-trivial lower bounds, with $R(t)$ being smaller in magnitude than $Z(t)$, it suffices to have only an upper bound for \eqref{eq:R}. Hence, in order to bound the expression on the above right side, we observe that
\begin{equation*}
  \int_{T_1}^{T_2}\left|\frac{X^{1-s}}{s-1}\right|^2dt\leq X^{2-2\sigma}\int_{T_1}^{T_2}\frac{dt}{t^{2}}=X^{2-2\sigma}\left(\frac{1}{T_{1}}-\frac{1}{T_{2}}\right).
\end{equation*}
For the second term we simply have $\int_{T_1}^{T_2}D^2X^{-2\sigma}dt=(T_2-T_1)D^2X^{-2\sigma}$, while the third one is bounded as
\begin{equation*}
\int_{T_1}^{T_2}\frac{2D}{X^{\sigma}}\left|\frac{X^{1-s}}{s-1}\right|dt\leq 2DX^{1-2\sigma}\int_{T_1}^{T_2}\frac{dt}{t}=2DX^{1-2\sigma}\log\left(\frac{T_{2}}{T_{1}}\right).
\end{equation*}
We obtain then
\begin{equation*}
\int_{T_1}^{T_2}|R(t)|^2dt\leq X^{2-2\sigma}\left(\frac{1}{T_{1}}-\frac{1}{T_{2}}\right)+\frac{D^2(T_2-T_1)}{X^{2\sigma}}+2DX^{1-2\sigma}\log\left(\frac{T_{2}}{T_{1}}\right).
\end{equation*}
 
Putting everything together, and imposing $X=T_{2}$ in order to minimize the various terms that arise ($X<T_{2}$ is not possible, by the conditions in Lemma~\ref{le:353}), we obtain the result in the statement.
\end{proof}

\begin{proposition}\label{prop:z-allb}
Let $\frac{1}{2}\leq\sigma\leq 1$ and $T_{1},T_{2}$ be real numbers such that $1\leq T_{1}\leq T_{2}$. 
Then, for any $\rho>0$,
$\int_{T_{1}}^{T_{2}}|\zeta(\sigma+it)|^{2}dt$ is at most
\begin{equation}\label{eq:z-allb} 
(1+\rho)\left(f_{2,1}^{+}(\sigma,T_{1},T_{2})+f_{2,2}^{+}(\sigma,T_{2})\right)+\left(1+\frac{1}{\rho}\right)f_{2,3}^{+}(\sigma,T_{1},T_{2}),
\end{equation}
where
\begin{align*}
f_{2,1}^{+}(\sigma,T_{1},T_{2})\!= & \!\begin{cases}
  T_{2}\log T_{2}\!-\!T_{1}\log T_{1}\!-\!(1\!-\!\gamma)(T_{2}\!-\!T_{1})
  \!+\!\frac{1}{2} \log \frac{T_2}{T_1} & \text{if $\sigma=\frac{1}{2}$,} \\ \\
\zeta(2\sigma)(T_{2}-T_{1}) & \text{if $\sigma>\frac{1}{2}$,}\end{cases} \\ \\
f_{2,2}^{+}(\sigma,T_{2})\!= & \!\begin{cases}
2T_{2}\log T_{2}\!+\!\left(2\gamma\!+\!\frac{E}{2}\!+\!16\right)T_{2}\!+\!\left(\frac{E}{4}\!-\!1\right)\log T_{2} & \\ 
 \ \ \ \ \ +1\!+\!\left(\frac{E}{4}\!-\!1\right)\gamma\!+\!\left(\frac{E}{4}\!-\!1\right)\frac{1}{2T_{2}} & \!\!\!\!\!\!\!\!\!\!\!\!\!\!\!\!\!\!\!\text{if $\sigma=\frac{1}{2}$,} \\ \\ 
\frac{2T_{2}^{2-2\sigma}\log T_{2}}{1-\sigma}\!+\!\left(\frac{E}{2}\!+\!2\gamma\!+\!\frac{4}{1-\sigma}\right)\!\frac{T_{2}^{2-2\sigma}}{1-\sigma}\!+\!\left(\frac{E}{4}\!-\!1\right)\!\zeta(2\sigma) & \\ 
 \ \ \ \ \ +\!\left(\!\frac{1}{1-\sigma}\!-\!\frac{\frac{E}{4}-1}{2\sigma-1}\right)\!\frac{1}{T_{2}^{2\sigma-1}}\!+\!\left(\frac{E}{4}\!-\!1\right)\!\frac{1}{2T_{2}^{2\sigma}} & \!\!\!\!\!\!\!\!\!\!\!\!\!\!\!\!\!\!\!\text{if $\frac{1}{2}<\sigma<1$,} \\ \\ 
3\log^{2}T_{2}\!+\!\left(6\gamma\!+\!\frac{E}{2}\right)\log T_{2}\!+\!3\gamma^{2}\!+\!\frac{E}{2}\gamma & \\
 \ \ \ \ \ +\!\left(\frac{E}{4}\!-\!1\right)\!\zeta(2)\!+\!\frac{3\log T_{2}}{T_{2}}\!+\!\frac{3\gamma+\frac{E}{4}}{T_{2}}\!+\!\frac{3}{4T_{2}^{2}} & \!\!\!\!\!\!\!\!\!\!\!\!\!\!\!\!\!\!\!\text{if $\sigma=1$,}\end{cases} \\ \\ 
f_{2,3}^{+}(\sigma,T_{1},T_{2})= & \begin{cases}
(1+D)^{2}\log\left(\frac{T_{2}}{T_{1}}\right) & \text{if $\sigma=\frac{1}{2}$,} \\ 
\frac{(1+D)^{2}}{2\sigma-1} & \text{if $\frac{1}{2}<\sigma\leq 1$,}\end{cases}
\end{align*}
and the constants $D$ and $E$ are as in Lemma~\ref{le:353} with $C=\lfloor T_{2}\rfloor$ and as in Proposition~\ref{pr:brudern}, respectively. Moreover, for any $-1<\rho<0$, $\int_{T_{1}}^{T_{2}}|\zeta(\sigma+it)|^{2}dt$ is bounded from below by the expression in \eqref{eq:z-allb} where $f_{2,1}^{+}$, $f_{2,2}^{+}$, $f_{2,3}^{+}$ are replaced respectively by
\begin{align*}
f_{2,1}^{-}(\sigma,T_{1},T_{2})\!= & \!\begin{cases}
  T_{2}\log T_{2}\!-\!T_{1}\log T_{1}\!-\!(1\!-\!\gamma)(T_{2}\!-\!T_{1})
  \!-\!c\log \frac{T_{2}}{T_{1}} & \text{if $\sigma=\frac{1}{2}$,} \\ \\
\zeta(2\sigma)(T_{2}\!-\!T_{1})\!-\!\frac{T_{2}^{2-2\sigma}-T_{1}^{2-2\sigma}}{2(1-\sigma)(2\sigma-1)}\!-\!\frac{T_{2}^{1-2\sigma}-T_{1}^{1-2\sigma}}{2(2\sigma-1)} & \text{if $\sigma>\frac{1}{2}$,}\end{cases} \\ 
f_{2,2}^{-}(\sigma,T_{2})\!= & \!-f_{2,2}^{+}(\sigma,T_{2}),\;\;\;\;\;\;\;\;\;\;
f_{2,3}^{-}(\sigma,T_{1},T_{2})\!=\!f_{2,3}^{+}(\sigma,T_{1},T_{2}),
\end{align*}
where $c$ is as in Lemma~\ref{le:harmonic}.
\end{proposition}

\begin{proof}
We start with the bound in Lemma~\ref{le:353}. For $s=\sigma+it$ and $X=t$, by the triangle inequality we get
\begin{align}
\int_{T_{1}}^{T_{2}}|\zeta(\sigma+it)|^{2}dt = & \int_{T_{1}}^{T_{2}}\left|\sum_{n\leq t}\frac{1}{n^{\sigma+it}}\right|^{2}\!\!dt + O^{*}\left(2\int_{T_{1}}^{T_{2}}\left|\sum_{n\leq t}\frac{1}{n^{\sigma+it}}\right|\left|\frac{t^{1-s}}{s-1}\right|dt\right. \nonumber \\
 & + 2D\int_{T_{1}}^{T_{2}}\left|\sum_{n\leq t}\frac{1}{n^{\sigma+it}}\right|t^{-\sigma}dt + \int_{T_{1}}^{T_{2}}\left|\frac{t^{1-s}}{s-1}\right|^{2}dt \nonumber \\
 & + \left.2D\int_{T_{1}}^{T_{2}}\left|\frac{t^{1-s}}{s-1}\right|t^{-\sigma}dt + D^{2}\int_{T_{1}}^{T_{2}}t^{-2\sigma}dt\right). \label{eq:zetalong}
\end{align}
The second and third term in \eqref{eq:zetalong} can be treated using the Cauchy-Schwarz inequality and reduced to the other integrals in the expression. Observe that the integrands $\left|\frac{t^{1-s}}{s-1}\right|t^{-\sigma}$, $\left|\frac{t^{1-s}}{s-1}\right|^{2}$ are both bounded from above by $t^{-2\sigma}$, so that all of their integrals are bounded by $\frac{1}{2\sigma-1}$ if $\frac{1}{2}<\sigma\leq 1$ and by $\log\left(\frac{T_{2}}{T_{1}}\right)$ if $\sigma=\frac{1}{2}$. Using then Lemma~\ref{lem:handy} we get
\begin{equation}\label{eq:firstbreak}
\int_{T_{1}}^{T_{2}}|\zeta(\sigma+it)|^{2}dt\leq(1+\rho)\int_{T_{1}}^{T_{2}}\left|\sum_{n\leq t}\frac{1}{n^{\sigma+it}}\right|^{2}dt+\left(1+\frac{1}{\rho}\right)f_{2,3}^{+}(\sigma,T_{1},T_{2}),
\end{equation}
and an analogous lower bound for $-1<\rho<0$.

We want now to estimate the first term in \eqref{eq:firstbreak}, namely we want bounds for the integral $\int_{T_{1}}^{T_{2}}\left|\sum_{n\leq t}a_{n}e^{i\lambda_{n}t}\right|^{2}dt$,
where in our case $a_{n}=\frac{1}{n^{\sigma}}\in\mathbb{R}^{+}$ and $\lambda_{n}=-\log n$. First, note that
\begin{align}
\int_{T_{1}}^{T_{2}}\left|\sum_{n\leq t}a_{n}e^{i\lambda_{n}t}\right|^{2}dt
 = & \int_{T_{1}}^{T_{2}}\sum_{n\leq t}a_{n}^{2}dt+\int_{T_{1}}^{T_{2}}\sum_{\substack{l,r\leq t \\ l\neq r}}a_{l}a_{r}e^{i(\lambda_{l}-\lambda_{r})t}dt. \label{eq:brudpiece}
\end{align}
If $\frac{1}{2}<\sigma\leq 1$, the first integral in \eqref{eq:brudpiece} is bounded by Lemma~\ref{recip} as
\begin{align*}
 & \int_{T_{1}}^{T_{2}}\left(\zeta(2\sigma)-\frac{t^{1-2\sigma}}{2\sigma-1}-\frac{1}{2t^{2\sigma}}\right)dt \\
\leq & \int_{T_{1}}^{T_{2}}\sum_{n\leq t}a_{n}^{2}dt\leq\int_{T_{1}}^{T_{2}}\left(\zeta(2\sigma)-\frac{t^{1-2\sigma}}{2\sigma-1}+\frac{1}{t^{2\sigma}}\right)dt, 
\end{align*} 
so that
\begin{align*}
 & \zeta(2\sigma)(T_{2}-T_{1})-\frac{T_{2}^{2-2\sigma}-T_{1}^{2-2\sigma}}{2(1-\sigma)(2\sigma-1)}-\frac{T_{2}^{1-2\sigma}-T_{1}^{1-2\sigma}}{2(2\sigma-1)} \\
\leq & \int_{T_{1}}^{T_{2}}\sum_{n\leq t}a_{n}^{2}dt\leq\zeta(2\sigma)(T_{2}-T_{1}), 
\end{align*}
where we use that $-\frac{t^{1-2\sigma}}{2\sigma-1}+t^{-2\sigma}\leq 0$ (under the conditions for $\sigma,t$), and we can extract an analogous lower bound.

If $\sigma=\frac{1}{2}$, the first integral is bounded from above as
\begin{align}
 & \ \int_{T_{1}}^{T_{2}}\sum_{n\leq t}a_{n}^{2}dt\leq \int_{T_{1}}^{T_{2}}\left(\log t+\gamma+\frac{1}{2t}\right)dt \nonumber \\ 
= & \ T_{2}\log T_{2}-T_{1}\log T_{1}-(1-\gamma)(T_{2}-T_{1}) +\frac{1}{2}(\log T_{2}-\log T_{1}), \label{eq:verymain12} 
\end{align}
by Lemma~\ref{le:harmonic}, from which we can derive an analogous lower bound.

As for the second integral in \eqref{eq:brudpiece}, consider first $T_{1},T_{2}$ to be integers for simplicity: we make use of the fact that a sum for $l,r\leq t$ is the same as a sum for $l,r\leq\lfloor t\rfloor$ and get
\begin{align}\label{boundmv} 
\int_{T_{1}}^{T_{2}}\sum_{\substack{l,r\leq t \\ l\neq r}}a_{l}a_{r}e^{i(\lambda_{l}-\lambda_{r})t}dt
 = & \sum_{j=T_{1}}^{T_{2}-1}\sum_{\substack{l,r\leq j \\ l\neq r}}a_{l}a_{r}\frac{e^{i(\lambda_{l}-\lambda_{r})(j+1)}-e^{i(\lambda_{l}-\lambda_{r})j}}{i(\lambda_{l}-\lambda_{r})} \nonumber \\
 = & \sum_{\substack{l,r\leq T_{2}-1 \\ l\neq r}}\sum_{j=\max\{T_{1},l,r\}}^{T_{2}-1}\!\!\!a_{l}a_{r}\frac{e^{i(\lambda_{l}-\lambda_{r})(j+1)}-e^{i(\lambda_{l}-\lambda_{r})j}}{i(\lambda_{l}-\lambda_{r})} \nonumber \\
 = & \sum_{\substack{l,r\leq T_{2}-1 \\ l\neq r}}\frac{a_{l}a_{r}}{\lambda_{l}-\lambda_{r}}\frac{e^{i(\lambda_{l}-\lambda_{r})T_{2}}-e^{i(\lambda_{l}-\lambda_{r})\max\{T_{1},l,r\}}}{i}.
\end{align}
For $T_{1},T_{2}$ general, we have to consider two additional integrals $\int_{T_{1}}^{\lceil T_{1}\rceil},\int_{\lfloor T_{2}\rfloor}^{T_{2}}$; we obtain however the same bound as in \eqref{boundmv}, with the summation going up to $\lfloor T_{2}\rfloor$ and with $T_1$ replaced by $\lfloor T_{1}\rfloor$.

We can divide the last sum in \eqref{boundmv} into two sums, one for each of the summands in the numerator of the second fraction. For the first sum we can reason as in Proposition~\ref{pr:brudern}, using \cite{Pr84} and obtaining
\begin{equation}\label{eq:offdiag1}
\left|\sum_{\substack{l,r\leq T_{2}-1 \\ l\neq r}}\frac{a_{l}a_{r}e^{i(\lambda_{l}-\lambda_{r})T_{2}}}{i(\lambda_{l}-\lambda_{r})}\right|\leq\frac{E}{2}\sum_{n\leq\lfloor T_{2}\rfloor}\frac{a_{n}^{2}}{\min_{n'\neq n}|\lambda_{n}-\lambda_{n'}|}.
\end{equation}
As for the second sum, we can bound the summand in absolute value by $\frac{a_{l}a_{r}}{|\lambda_{l}-\lambda_{r}|}$; then we use classical arguments (see \cite[(3.5)-(3.6)]{Ka13}), and $\sum_{\substack{l,r\leq\lfloor T_{2}\rfloor \\ l\neq r}}\frac{a_{l}a_{r}}{|\lambda_{l}-\lambda_{r}|}$ is at most
\begin{equation}\label{eq:offdiag2}
\sum_{\substack{l,r\leq T_{2} \\ l\neq r}}\frac{a_{l}a_{r}}{|\lambda_{l}-\lambda_{r}|}\leq
\left(\sum_{r\leq T_{2}}\frac{1}{r^{\sigma}}\right)^{2}-\sum_{r\leq T_{2}}\frac{1}{r^{2\sigma}}+2\left(\sum_{r\leq T_{2}}\frac{1}{r^{2\sigma-1}}\right)\left(\sum_{r\leq T_{2}}\frac{1}{r}\right).
\end{equation}  
Upon putting \eqref{eq:offdiag1} and \eqref{eq:offdiag2} together, we resort to Lemmas~\ref{le:harmonic} and \ref{recip} along with the simplifications $\zeta(\alpha)+\frac{1}{T_{2}^{\alpha}}<\frac{1}{2T_{2}^{\alpha}}$, $\sum_{r\leq T_{2}}\frac{1}{r^{\alpha}}\leq\frac{2T_2^{1-\alpha}}{1-\alpha}$ for $0<\alpha<1$ and $T_{2}\geq 1$, and the bound $\sum_{r\leq T_{2}}\frac{1}{r^{2}}\leq\zeta(2)$. Subsequently, we obtain $f_{2,2}^{\pm}(\sigma,T_{2})$ as in the statement. 
\end{proof}

\subsection{Mean value estimates of $\zeta(s)$ for $\Re(s)\in\left[\frac{1}{2},1\right]$}\label{mvb}

\tiny
\begin{mysage}
Dz=Dfunction(floor(TZ))
###
#Given the coefficients of the terms inside f11,f12,f13 when sigma=1/2,
#the following code returns the coefficients of the whole bound as in Prop. 4.1.
#The entry [i][j] of each f1* is the coefficient of T^(1-i)*(log(T))^(1-j),
#while C[i][j] is the coefficient of T^(1-i/2)*(log(T))^(3/2-j/2).
#Here f13 indicates the term multiplied by 1+1/rho in Prop. 4.1 (analogous to f23 in Prop. 4.2).
#We also assume that we are choosing rho of order 1/sqrt(log(T)).
def Cz_12_string(f11,f12,f13,rho):
    C=[[0,0,0,0,0],[0,0,0,0,0],[0,0,0,0,0],[0,0,0,0,0],[0,0,0,0,0],[0,0,0,0,0]]
    i=0
    while i<=2:
        j=0
        while j<=1:
            if 2*i-2>=0:
                #Here we add the terms multiplied by 1*T:
                #the terms of f11 of order Tlog(T) and T do not exist
                #(the result itself would be quite wrong if they did),
                #so asking for 2*i-2>=0 does not remove anything.
                C[2*i-2][2*j+1]=C[2*i-2][2*j+1]+f11[i][j]
                #Here we add the terms multiplied by rho*T.
                C[2*i-2][2*j+2]=C[2*i-2][2*j+2]+rho*f11[i][j]
            #Here we add the terms multiplied by 1/rho.
            C[2*i][2*j]=C[2*i][2*j]+1/rho*f13[i][j]
            #Here we add the terms multiplied by 1.
            C[2*i][2*j+1]=C[2*i][2*j+1]+(E/2-1)*f11[i][j]+E*f12[i][j]+f13[i][j]
            #Here we add the terms multiplied by rho.
            C[2*i][2*j+2]=C[2*i][2*j+2]+rho*((E/2-1)*f11[i][j]+E*f12[i][j])
            j=j+1
        i=i+1
    return C
###
#The following code is analogous to Cz_12_string, but it is designed for sigma=1.
#However, we will later need a version of these bounds with lower extremum kept general:
#therefore we use 'index' to separate the two cases.
#If index=0 we are computing the version with lower extremum=1 (and rho is of order 1/sqrt(T)),
#while if index=1 we have a general lower extremum u (and rho is of order 1/sqrt(uT))
#and then, as 1<=u<=T, in each term we use either the simplification u-->1 or u-->T,
#whichever is the correct one to use for the bound;
#the effect of such simplification is that, in the upper bound constants,
#if index=1 the only thing that changes with respect to index=0
#is that the constants multiplied by 1/rho move up by sqrt(T) (since rho will also
#be multiplied by 1/sqrt(u)), and the same applies to the lower bound constants
#(minorizing 1+rho,1+1/rho and then looking inside the expressions they multiply).
#The entry [i][j] of each f1* is the coefficient of T^(-i)*(log(T))^(1-j),
#while C[i][j] is the coefficient of T^(1-i/2)*(log(T))^(1-j).
#Here f13 indicates the term multiplied by 1+1/rho in Prop. 4.1 (analogous to f23 in Prop. 4.2).
#As already said, we are choosing rho of order 1/sqrt(T) or 1/sqrt(uT) (depending on 'index').
def Cz_1_string(f11,f12,f13,rho,index):
    C=[[0,0],[0,0],[0,0],[0,0],[0,0],[0,0],[0,0],[0,0]]
    i=0
    while i<=2:
        j=0
        while j<=1:
            #Here we add the terms multiplied by 1*T.
            C[2*i][j]=C[2*i][j]+f11[i][j]
            #Here we add the terms multiplied by rho*T.
            C[2*i+1][j]=C[2*i+1][j]+rho*f11[i][j]
            #Here we add the terms multiplied by 1/rho.
            #As already said, the index changes the order of 1/rho after simplification,
            #which is sqrt(T) if index=0 and T if index=1.
            if index==0:
                C[2*i+1][j]=C[2*i+1][j]+1/rho*f13[i][j]
            elif index==1:
                C[2*i][j]=C[2*i][j]+1/rho*f13[i][j]
            else:
                return 'Error'
            #Here we add the terms multiplied by 1.
            C[2*i+2][j]=C[2*i+2][j]+(E/2-1)*f11[i][j]+E*f12[i][j]+f13[i][j]
            #Here we add the terms multiplied by rho.
            C[2*i+3][j]=C[2*i+3][j]+rho*((E/2-1)*f11[i][j]+E*f12[i][j])
            j=j+1
        i=i+1
    return C
###
#The following code returns the coefficient of the resulting error term
#after merging all the small error terms in C,
#where C is the result of Cz_12_string (length=6) or Cz_1_string (length=8);
#the terms to be merged are T^(1-i0/2)*(log(T))^(3/2-j0/2) and lower in the first case
#and T^(1-i0/2)*(log(T))^(1-j0) and lower in the second case
#(and that will be the order of the resulting error term).
#The 'sign' has effect only for length=8 (i.e. the case of Cz_1_string):
#if sign=0, the bound has lower extremum equal to 1 and nothing happens;
#if sign=1 (resp. -1), we are working with the upper (resp. lower) bound with general lower extremum,
#from which we have to remove the small contributions that are not counted in the numerical bound
#(see the code below on how we do it, and the proof of Thm. 4.5 on why we do it).
def Cz_fix_coeff(C,length,i0,j0,sign):
    S=0
    i=i0
    j=j0
    while i<length:
        while j<len(C[i]):
            if length==8 and sign!=0:
                 if sign==1 or sign==-1:
                     #The contribution to remove is zeta(2)*(-u), as seen in (4.17).
                     #The error terms in the upper (resp. lower) bound are positive (resp. negative),
                     #but the coefficients that are being fed to the procedure are provided in absolute value,
                     #so to remove the contribution we have to add (resp. subtract) zeta(2).
                     C[2][1]=C[2][1]+zeta(2)*sign
                 else:
                     return 'Error'
            #Every term that is of smaller order than the highest term we are merging
            #(namely the T^(1-i0/2)*(log(T))^(3/2-j0/2) term or the
            #T^(1-i0/2)*(log(T))^(1-j0) term) comes with a correction that makes it
            #impact less on the final product: for example, if the highest term is Tlog(T)
            #and the term we are considering is sqrt(T), we can say sqrt(T)<=Tlog(T)*1/(sqrt(T0)*log(T0))
            #for any T>=T0 we allow ourselves to consider.
            if length==6:
                correction=TZ^(1/2*(i0-i))*(log(TZ))^(1/2*(j0-j))
            elif length==8:
                correction=TZ^(1/2*(i0-i))*(log(TZ))^(j0-j)
            #For T<e^(a/b) though, (log(T))^a/T^b is not decreasing, so we need a correction
            #to the correction: for all the terms with i0-i<0 and j0-j>0,
            #in case of T low the best we can do is taking the maximum
            #of the function (log(x))^a/x^b as our 'correction',
            #which occurs in e^(a/b), for the appropriate a,b.
            if i0-i<0 and j0-j>0:
                if length==6 and TZ<e^((j0-j)/(i-i0)):
                    correction=((j0-j)/(i-i0)/e)^((j0-j)/2)
                elif length==8 and TZ<e^(2*(j0-j)/(i-i0)):
                    correction=(2*(j0-j)/(i-i0)/e)^(j0-j)
            summand = RIF( C[i][j]*correction )
            #For the terms of smaller order than the highest term we are merging,
            #we must ignore the negative contributions (with a negative coefficient,
            #the correction we pointed out before changes verse in the inequality,
            #for example -sqrt(T)>=-Tlog(T)*1/(sqrt(T0)*log(T0)), so we just cut it off and say -sqrt(T)<=0).
            if i!=i0 or j!=j0:
                summand=max(0,summand)
            S=S+summand
            j=j+1
        j=0
        i=i+1
    return S
###
rho_12_0_m12 = RIF( 1 ) #Coefficient of (log(T))^(-1/2) inside rho for sigma=1/2.
#The following strings collect the coefficients of the f1* when sigma=1/2,
#where the entry [i][j] of each f1* is the coefficient of T^(1-i)*(log(T))^(1-j).
#The 'u,l' indicate upper and lower bounds, respectively.
f11_12_u = [[ RIF( 0 ) , RIF( 0 ) ],[ RIF( 1 ) , RIF( g ) ],[ RIF( 0 ) , RIF( 1/2 ) ]]
f12_12_u = [[ RIF( 0 ) , RIF( 1 ) ],[ RIF( 0 ) , RIF( 1/2 ) ],[ RIF( 0 ) , RIF( 0 ) ]]
f13_12_u = [[ RIF( 0 ) , RIF( 1 ) ],[ RIF( 2*Dz ) , RIF( -1+Dz^2 ) ],[ RIF( 0 ) , RIF( -Dz^2 ) ]]
f11_12_l = [[ RIF( 0 ) , RIF( 0 ) ],[ RIF( -1 ) , RIF( -g ) ],[ RIF( 0 ) , RIF( lower_harm ) ]]
f12_12_l = [[ RIF( 0 ) , RIF( -1 ) ],[ RIF( 0 ) , RIF( 1/2-zeta(0) ) ],[ RIF( 0 ) , RIF( 0 ) ]]
f13_12_l = [[ RIF( 0 ) , RIF( -1 ) ],[ RIF( -2*Dz ) , RIF( 1-Dz^2 ) ],[ RIF( 0 ) , RIF( Dz^2 ) ]]
#The terms Cz_12_*a are those of order T*sqrt(log(T)),
#and the Cz_12_*b collect the error terms from sqrt(T)*log(T) downwards.
Cz_12_ua = Cz_12_string(f11_12_u,f12_12_u,f13_12_u,rho_12_0_m12)[0][2]
Cz_12_ub = Cz_fix_coeff(Cz_12_string(f11_12_u,f12_12_u,f13_12_u,rho_12_0_m12),6,0,3,0)
Cz_12_la = Cz_12_string(f11_12_l,f12_12_l,f13_12_l,-rho_12_0_m12)[0][2]
Cz_12_lb = Cz_fix_coeff(Cz_12_string(f11_12_l,f12_12_l,f13_12_l,-rho_12_0_m12),6,0,3,0)
###
rho_1_m12_0 = RIF( 1/sqrt(zeta(2)) ) #Coefficient of T^(-1/2) inside rho for sigma=1.
#The following strings collect the coefficients of the f1* when sigma=1,
#where the entry [i][j] of each f1* is the coefficient of T^(-i)*(log(T))^(1-j).
#The 'u,l' indicate upper and lower bounds, respectively.
f11_1_u = [[ RIF( 0 ) , RIF( zeta(2) ) ],[ RIF( 0 ) , RIF( -1 ) ],[ RIF( 0 ) , RIF( 1 ) ]]
f12_1_u = [[ RIF( 1 ) , RIF( g ) ],[ RIF( 0 ) , RIF( 1/2 ) ],[ RIF( 0 ) , RIF( 0 ) ]]
f13_1_u = [[ RIF( 0 ) , RIF( 1 ) ],[ RIF( 2*Dz ) , RIF( -1+Dz^2 ) ],[ RIF( 0 ) , RIF( -Dz^2 ) ]]
f11_1_l = [[ RIF( 0 ) , RIF( -zeta(2) ) ],[ RIF( 0 ) , RIF( 1 ) ],[ RIF( 0 ) , RIF( 1/2 ) ]]
f12_1_l = [[ RIF( -1 ) , RIF( -g ) ],[ RIF( 0 ) , RIF( lower_harm ) ],[ RIF( 0 ) , RIF( 0 ) ]]
f13_1_l = [[ RIF( 0 ) , RIF( -1 ) ],[ RIF( -2*Dz ) , RIF( 1-Dz^2 ) ],[ RIF( 0 ) , RIF( Dz^2 ) ]]
#The terms Cz_1_* collect the error terms from log(T) downwards.
Cz_1_u = Cz_fix_coeff(Cz_1_string(f11_1_u,f12_1_u,f13_1_u,rho_1_m12_0,0),8,2,0,0)
Cz_1_l = Cz_fix_coeff(Cz_1_string(f11_1_l,f12_1_l,f13_1_l,-rho_1_m12_0,0),8,2,0,0)
###
#Given the coefficients of the terms inside f21,f22,f23 when 1/2<sigma<1,
#the following code returns the coefficients of the whole bound as in Prop. 4.2.
#The entry [i][j][l][k] of each f2* is the coefficient of T^(1-i+(1-2*sigma)j)*(log(T))^(1-l)
#times the following functions of sigma: 1 (for k=0), 1/(1-sigma) (for k=1), 1/(1-sigma)^2 (for k=2)
#1/(2*sigma-1) (for k=3), 1/((1-sigma)*(2*sigma-1)) (for k=4), zeta(2*sigma) (for k=5);
#C[i][j][l][k] is the coefficient of T^(1-i/2+(1-2*sigma)j)*(log(T))^(1-l) times the same functions,
#although keep in mind that k=5 is removed by the use of Lemma 2.12.
#We also assume that we are choosing rho of order 1/sqrt((2*sigma-1)*zeta(2*sigma)*T).
def Cz_121_string(f21,f22,f23,rho):
    C=[[[[0,0,0,0,0],[0,0,0,0,0]],[[0,0,0,0,0],[0,0,0,0,0]]], \
      [[[0,0,0,0,0],[0,0,0,0,0]],[[0,0,0,0,0],[0,0,0,0,0]]], \
      [[[0,0,0,0,0],[0,0,0,0,0]],[[0,0,0,0,0],[0,0,0,0,0]]], \
      [[[0,0,0,0,0],[0,0,0,0,0]],[[0,0,0,0,0],[0,0,0,0,0]]], \
      [[[0,0,0,0,0],[0,0,0,0,0]],[[0,0,0,0,0],[0,0,0,0,0]]], \
      [[[0,0,0,0,0],[0,0,0,0,0]],[[0,0,0,0,0],[0,0,0,0,0]]]]
    i=0
    while i<=2:
        j=0
        while j<=1:
            l=0
            while l<=1:
                k=0
                #We remark that k=5 is treated separately, as it involves zeta.
                while k<=4:
                    #Here we add the terms multiplied by 1.
                    C[2*i][j][l][k]=C[2*i][j][l][k]+f21[i][j][l][k]+f22[i][j][l][k]+f23[i][j][l][k]
                    #Here we add the terms multiplied by rho.
                    #For the terms with k!=5 and multiplied by rho there is zeta, but we have
                    #(2*sigma-1)*zeta(2*sigma)>1 in the denominator, so the issue solves itself:
                    #the factor in the denominator with zeta and sigma is just bounded by 1.
                    C[2*i+1][j][l][k]=C[2*i+1][j][l][k]+rho*(f21[i][j][l][k]+f22[i][j][l][k])
                    if 2*i-1>=0:
                        #Here we add the terms multiplied by 1/rho:
                        #the terms of f23 of order Tlog(T) and T do not exist
                        #(the result itself would be quite wrong if they did),
                        #so asking for 2*i-1>=0 does not remove anything.
                        #For the terms multiplied by 1/rho, we have
                        #sqrt(zeta(2*sigma)*(2*sigma-1))<sqrt(1+H), so this is how zeta is treated.
                        C[2*i-1][j][l][k]=C[2*i-1][j][l][k]+1/rho*sqrt(1+H)*f23[i][j][l][k]
                    k=k+1
                #For the terms with k=5 (i.e. multiplied by zeta(2*sigma)), we have
                #zeta(2*sigma)<1/(2*sigma-1)+H, so we can absorb them into the terms with k=0,3.
                C[2*i][j][l][3]=C[2*i][j][l][3]+f21[i][j][l][5]+f22[i][j][l][5]+f23[i][j][l][5]
                C[2*i][j][l][0]=C[2*i][j][l][0]+H*(f21[i][j][l][5]+f22[i][j][l][5]+f23[i][j][l][5])
                #For the terms with k=5 and multiplied by rho, we have
                #sqrt(zeta(2*sigma)/(2*sigma-1))<sqrt(1+H)/(2*sigma-1),
                #so we can absorb them into the terms with k=3.
                C[2*i+1][j][l][3]=C[2*i+1][j][l][3]+sqrt(1+H)*rho*(f21[i][j][l][5]+f22[i][j][l][5])
                if 2*i-1>=0:
                    #For the terms with k=5 and multiplied by 1/rho, we have
                    #zeta(2*sigma)^(3/2)*sqrt(2*sigma-1)<sqrt(1+H)*(1/(2*sigma-1)+H),
                    #so we can absorb them into the terms with k=0,3.
                    C[2*i-1][j][l][3]=C[2*i-1][j][l][3]+sqrt(1+H)*1/rho*f23[i][j][l][5]
                    C[2*i-1][j][l][0]=C[2*i-1][j][l][0]+sqrt(1+H)*H*1/rho*f23[i][j][l][5]
                l=l+1
            j=j+1
        i=i+1
    return C
###
#The following code returns the resulting coefficients after merging all the small
#error terms in C, where C is the result of Cz_121_string. The coefficients are those
#of max(T^(2-2*sigma)*log(T),sqrt(T)) times the functions in the definition of Cz_121_string,
#and C[i][j][l][k] is the coefficient of T^(1-i/2+(1-2*sigma)j)*(log(T))^(1-l) times the same functions.
def Cz_var_coeff(C):
    S=[0,0,0,0,0]
    i=0
    j=1
    while i<=5:
        while j<=1:
            l=0
            while l<=1:
                k=0
                while k<=4:
                    #Every term that is of smaller order than the highest term we are merging
                    #(namely either T^(2-2*sigma)*log(T) or sqrt(T)) comes with a correction that makes it
                    #impact less on the final product: for example, if the highest term is sqrt(T)
                    #and the term we are considering is 1, we can say 1<=sqrt(T)/sqrt(T0)
                    #for any T>=T0 we allow ourselves to consider.
                    #We also take advantage of the fact that all the terms with j=0,l=0
                    #have coefficient=0, to make our life easier and avoid having to deal with
                    #corrections to the corrections as we had to do in Cz_fix_coeff.
                    if j==0 and l==0 and C[i][j][l][k]!=0:
                        return 'Error'
                    if j==0:
                        correction=TZ^((1-i-j)/2)
                    elif j==1:
                        correction=TZ^((1-i-j)/2)*(log(TZ))^(1-l)
                    #We must also ignore the negative contributions.
                    S[k]=S[k]+max(0, RIF( C[i][j][l][k]*correction ) )
                    k=k+1
                l=l+1
            j=j+1
        j=0
        i=i+1
    return S
###
rho_121_m12_0_s12msz = RIF( 1+Dz ) #Coefficient of T^(-1/2)/sqrt((2*sigma-1)*zeta(2*sigma))
#inside rho for 1/2<sigma<1.
#The following strings collect the coefficients of the f2* when 1/2<sigma<1,
#where the entry [i][j][l][k] of each f2* is the coefficient of T^(1-i+(1-2*sigma)j)*(log(T))^(1-l)
#times the following functions of sigma: 1 (for k=0), 1/(1-sigma) (for k=1), 1/(1-sigma)^2 (for k=2)
#1/(2*sigma-1) (for k=3), 1/((1-sigma)*(2*sigma-1)) (for k=4), zeta(2*sigma) (for k=5).
#The 'u,l' indicate upper and lower bounds, respectively.
f21_121_u = [[[[ RIF( 0 ) , RIF( 0 ) , RIF( 0 ) , RIF( 0 ) , RIF( 0 ) , RIF( 0 ) ], \
            [ RIF( 0 ) , RIF( 0 ) , RIF( 0 ) , RIF( 0 ) , RIF( 0 ) , RIF( 1 ) ]], \
            [[ RIF( 0 ) , RIF( 0 ) , RIF( 0 ) , RIF( 0 ) , RIF( 0 ) , RIF( 0 ) ], \
            [ RIF( 0 ) , RIF( 0 ) , RIF( 0 ) , RIF( 0 ) , RIF( 0 ) , RIF( 0 ) ]]], \
            [[[ RIF( 0 ) , RIF( 0 ) , RIF( 0 ) , RIF( 0 ) , RIF( 0 ) , RIF( 0 ) ], \
            [ RIF( 0 ) , RIF( 0 ) , RIF( 0 ) , RIF( 0 ) , RIF( 0 ) , RIF( -1 ) ]], \
            [[ RIF( 0 ) , RIF( 0 ) , RIF( 0 ) , RIF( 0 ) , RIF( 0 ) , RIF( 0 ) ], \
            [ RIF( 0 ) , RIF( 0 ) , RIF( 0 ) , RIF( 0 ) , RIF( 0 ) , RIF( 0 ) ]]], \
            [[[ RIF( 0 ) , RIF( 0 ) , RIF( 0 ) , RIF( 0 ) , RIF( 0 ) , RIF( 0 ) ], \
            [ RIF( 0 ) , RIF( 0 ) , RIF( 0 ) , RIF( 0 ) , RIF( 0 ) , RIF( 0 ) ]], \
            [[ RIF( 0 ) , RIF( 0 ) , RIF( 0 ) , RIF( 0 ) , RIF( 0 ) , RIF( 0 ) ], \
            [ RIF( 0 ) , RIF( 0 ) , RIF( 0 ) , RIF( 0 ) , RIF( 0 ) , RIF( 0 ) ]]]]
f22_121_u = [[[[ RIF( 0 ) , RIF( 0 ) , RIF( 0 ) , RIF( 0 ) , RIF( 0 ) , RIF( 0 ) ], \
            [ RIF( 0 ) , RIF( 0 ) , RIF( 0 ) , RIF( 0 ) , RIF( 0 ) , RIF( 0 ) ]], \
            [[ RIF( 0 ) , RIF( 2 ) , RIF( 0 ) , RIF( 0 ) , RIF( 0 ) , RIF( 0 ) ], \
            [ RIF( 0 ) , RIF( E/2+2*g ) , RIF( 4 ) , RIF( 0 ) , RIF( 0 ) , RIF( 0 ) ]]], \
            [[[ RIF( 0 ) , RIF( 0 ) , RIF( 0 ) , RIF( 0 ) , RIF( 0 ) , RIF( 0 ) ], \
            [ RIF( 0 ) , RIF( 0 ) , RIF( 0 ) , RIF( 0 ) , RIF( 0 ) , RIF( E/4-1 ) ]], \
            [[ RIF( 0 ) , RIF( 0 ) , RIF( 0 ) , RIF( 0 ) , RIF( 0 ) , RIF( 0 ) ], \
            [ RIF( 0 ) , RIF( 1 ) , RIF( 0 ) , RIF( -(E/4-1) ) , RIF( 0 ) , RIF( 0 ) ]]], \
            [[[ RIF( 0 ) , RIF( 0 ) , RIF( 0 ) , RIF( 0 ) , RIF( 0 ) , RIF( 0 ) ], \
            [ RIF( 0 ) , RIF( 0 ) , RIF( 0 ) , RIF( 0 ) , RIF( 0 ) , RIF( 0 ) ]], \
            [[ RIF( 0 ) , RIF( 0 ) , RIF( 0 ) , RIF( 0 ) , RIF( 0 ) , RIF( 0 ) ], \
            [ RIF( 1/2*(E/4-1) ) , RIF( 0 ) , RIF( 0 ) , RIF( 0 ) , RIF( 0 ) , RIF( 0 ) ]]]]
f23_121_u = [[[[ RIF( 0 ) , RIF( 0 ) , RIF( 0 ) , RIF( 0 ) , RIF( 0 ) , RIF( 0 ) ], \
            [ RIF( 0 ) , RIF( 0 ) , RIF( 0 ) , RIF( 0 ) , RIF( 0 ) , RIF( 0 ) ]], \
            [[ RIF( 0 ) , RIF( 0 ) , RIF( 0 ) , RIF( 0 ) , RIF( 0 ) , RIF( 0 ) ], \
            [ RIF( 0 ) , RIF( 0 ) , RIF( 0 ) , RIF( 0 ) , RIF( 0 ) , RIF( 0 ) ]]], \
            [[[ RIF( 0 ) , RIF( 0 ) , RIF( 0 ) , RIF( 0 ) , RIF( 0 ) , RIF( 0 ) ], \
            [ RIF( 0 ) , RIF( 0 ) , RIF( 0 ) , RIF( (1+Dz)^2 ) , RIF( 0 ) , RIF( 0 ) ]], \
            [[ RIF( 0 ) , RIF( 0 ) , RIF( 0 ) , RIF( 0 ) , RIF( 0 ) , RIF( 0 ) ], \
            [ RIF( 0 ) , RIF( 0 ) , RIF( 0 ) , RIF( 0 ) , RIF( 0 ) , RIF( 0 ) ]]], \
            [[[ RIF( 0 ) , RIF( 0 ) , RIF( 0 ) , RIF( 0 ) , RIF( 0 ) , RIF( 0 ) ], \
            [ RIF( 0 ) , RIF( 0 ) , RIF( 0 ) , RIF( 0 ) , RIF( 0 ) , RIF( 0 ) ]], \
            [[ RIF( 0 ) , RIF( 0 ) , RIF( 0 ) , RIF( 0 ) , RIF( 0 ) , RIF( 0 ) ], \
            [ RIF( 0 ) , RIF( 0 ) , RIF( 0 ) , RIF( 0 ) , RIF( 0 ) , RIF( 0 ) ]]]]
f21_121_l = [[[[ RIF( 0 ) , RIF( 0 ) , RIF( 0 ) , RIF( 0 ) , RIF( 0 ) , RIF( 0 ) ], \
            [ RIF( 0 ) , RIF( 0 ) , RIF( 0 ) , RIF( 0 ) , RIF( 0 ) , RIF( -1 ) ]], \
            [[ RIF( 0 ) , RIF( 0 ) , RIF( 0 ) , RIF( 0 ) , RIF( 0 ) , RIF( 0 ) ], \
            [ RIF( 0 ) , RIF( 0 ) , RIF( 0 ) , RIF( 0 ) , RIF( 1/2 ) , RIF( 0 ) ]]], \
            [[[ RIF( 0 ) , RIF( 0 ) , RIF( 0 ) , RIF( 0 ) , RIF( 0 ) , RIF( 0 ) ], \
            [ RIF( 0 ) , RIF( 0 ) , RIF( 0 ) , RIF( -1 ) , RIF( -1/2 ) , RIF( 1 ) ]], \
            [[ RIF( 0 ) , RIF( 0 ) , RIF( 0 ) , RIF( 0 ) , RIF( 0 ) , RIF( 0 ) ], \
            [ RIF( 0 ) , RIF( 0 ) , RIF( 0 ) , RIF( 1 ) , RIF( 0 ) , RIF( 0 ) ]]], \
            [[[ RIF( 0 ) , RIF( 0 ) , RIF( 0 ) , RIF( 0 ) , RIF( 0 ) , RIF( 0 ) ], \
            [ RIF( 0 ) , RIF( 0 ) , RIF( 0 ) , RIF( 0 ) , RIF( 0 ) , RIF( 0 ) ]], \
            [[ RIF( 0 ) , RIF( 0 ) , RIF( 0 ) , RIF( 0 ) , RIF( 0 ) , RIF( 0 ) ], \
            [ RIF( 0 ) , RIF( 0 ) , RIF( 0 ) , RIF( 0 ) , RIF( 0 ) , RIF( 0 ) ]]]]
f22_121_l = [[[[ RIF( 0 ) , RIF( 0 ) , RIF( 0 ) , RIF( 0 ) , RIF( 0 ) , RIF( 0 ) ], \
            [ RIF( 0 ) , RIF( 0 ) , RIF( 0 ) , RIF( 0 ) , RIF( 0 ) , RIF( 0 ) ]], \
            [[ RIF( 0 ) , RIF( 2 ) , RIF( 0 ) , RIF( 0 ) , RIF( 0 ) , RIF( 0 ) ], \
            [ RIF( 0 ) , RIF( E/2+2*g ) , RIF( 4 ) , RIF( 0 ) , RIF( 0 ) , RIF( 0 ) ]]], \
            [[[ RIF( 0 ) , RIF( 0 ) , RIF( 0 ) , RIF( 0 ) , RIF( 0 ) , RIF( 0 ) ], \
            [ RIF( 0 ) , RIF( 0 ) , RIF( 0 ) , RIF( 0 ) , RIF( 0 ) , RIF( E/4-1 ) ]], \
            [[ RIF( 0 ) , RIF( 0 ) , RIF( 0 ) , RIF( 0 ) , RIF( 0 ) , RIF( 0 ) ], \
            [ RIF( 0 ) , RIF( 1 ) , RIF( 0 ) , RIF( -(E/4-1) ) , RIF( 0 ) , RIF( 0 ) ]]], \
            [[[ RIF( 0 ) , RIF( 0 ) , RIF( 0 ) , RIF( 0 ) , RIF( 0 ) , RIF( 0 ) ], \
            [ RIF( 0 ) , RIF( 0 ) , RIF( 0 ) , RIF( 0 ) , RIF( 0 ) , RIF( 0 ) ]], \
            [[ RIF( 0 ) , RIF( 0 ) , RIF( 0 ) , RIF( 0 ) , RIF( 0 ) , RIF( 0 ) ], \
            [ RIF( 1/2*(E/4-1) ) , RIF( 0 ) , RIF( 0 ) , RIF( 0 ) , RIF( 0 ) , RIF( 0 ) ]]]]
f23_121_l = [[[[ RIF( 0 ) , RIF( 0 ) , RIF( 0 ) , RIF( 0 ) , RIF( 0 ) , RIF( 0 ) ], \
            [ RIF( 0 ) , RIF( 0 ) , RIF( 0 ) , RIF( 0 ) , RIF( 0 ) , RIF( 0 ) ]], \
            [[ RIF( 0 ) , RIF( 0 ) , RIF( 0 ) , RIF( 0 ) , RIF( 0 ) , RIF( 0 ) ], \
            [ RIF( 0 ) , RIF( 0 ) , RIF( 0 ) , RIF( 0 ) , RIF( 0 ) , RIF( 0 ) ]]], \
            [[[ RIF( 0 ) , RIF( 0 ) , RIF( 0 ) , RIF( 0 ) , RIF( 0 ) , RIF( 0 ) ], \
            [ RIF( 0 ) , RIF( 0 ) , RIF( 0 ) , RIF( -(1+Dz)^2 ) , RIF( 0 ) , RIF( 0 ) ]], \
            [[ RIF( 0 ) , RIF( 0 ) , RIF( 0 ) , RIF( 0 ) , RIF( 0 ) , RIF( 0 ) ], \
            [ RIF( 0 ) , RIF( 0 ) , RIF( 0 ) , RIF( 0 ) , RIF( 0 ) , RIF( 0 ) ]]], \
            [[[ RIF( 0 ) , RIF( 0 ) , RIF( 0 ) , RIF( 0 ) , RIF( 0 ) , RIF( 0 ) ], \
            [ RIF( 0 ) , RIF( 0 ) , RIF( 0 ) , RIF( 0 ) , RIF( 0 ) , RIF( 0 ) ]], \
            [[ RIF( 0 ) , RIF( 0 ) , RIF( 0 ) , RIF( 0 ) , RIF( 0 ) , RIF( 0 ) ], \
            [ RIF( 0 ) , RIF( 0 ) , RIF( 0 ) , RIF( 0 ) , RIF( 0 ) , RIF( 0 ) ]]]]
#The terms Cz_121_* collect the error terms from T^(2-2*sigma)*log(T) and sqrt(T) downwards.
Cz_121_u = Cz_var_coeff(Cz_121_string(f21_121_u,f22_121_u,f23_121_u,rho_121_m12_0_s12msz))
Cz_121_l = Cz_var_coeff(Cz_121_string(f21_121_l,f22_121_l,f23_121_l,-rho_121_m12_0_s12msz))
\end{mysage}
\normalsize

\begin{theorem}\label{th:z->=1/2}
Let $T\geq T_{0}=\sage{TZ}$. Then
\begin{align*}
\int_{1}^{T}\left|\zeta\left(\frac{1}{2}+it\right)\right|^{2}dt\leq & \ T\log T+\sage{roundup(Cz_12_ua.upper(),5)}\cdot T\sqrt{\log T}+\sage{roundup(Cz_12_ub.upper(),5)}\cdot T \\
\int_{1}^{T}\left|\zeta\left(\frac{1}{2}+it\right)\right|^{2}dt\geq & \ T\log T-\sage{roundup(Cz_12_la.upper(),5)}\cdot T\sqrt{\log T}-\sage{roundup(Cz_12_lb.upper(),5)}\cdot T.
\end{align*}
Moreover, for $\frac{1}{2}<\sigma<1$,
\begin{align*}
\int_{1}^{T}|\zeta(\sigma+it)|^{2}dt\leq & \ \zeta(2\sigma)T+C^{+}(\sigma)\cdot\max\{T^{2-2\sigma}\log T,\sqrt{T}\} \\
\int_{1}^{T}|\zeta(\sigma+it)|^{2}dt\geq & \ \zeta(2\sigma)T-C^{-}(\sigma)\cdot\max\{T^{2-2\sigma}\log T,\sqrt{T}\}
\end{align*}
with
\begin{align*}
C^{+}(\sigma)= & \ \frac{\sage{roundup(Cz_121_u[2].upper(),5)}}{(1-\sigma)^{2}}+\frac{\sage{roundup(Cz_121_u[1].upper(),5)}}{1-\sigma}+\frac{\sage{roundup(Cz_121_u[3].upper(),5)}}{2\sigma-1}+\sage{roundup(Cz_121_u[0].upper(),5)} \\
C^{-}(\sigma)= & \ \frac{\sage{roundup(Cz_121_l[2].upper(),5)}}{(1-\sigma)^{2}}+\frac{\sage{roundup(Cz_121_l[4].upper(),5)}}{(1-\sigma)(2\sigma-1)}+\frac{\sage{roundup(Cz_121_l[1].upper(),5)}}{1-\sigma}+\frac{\sage{roundup(Cz_121_l[3].upper(),5)}}{2\sigma-1}+\sage{roundup(Cz_121_l[0].upper(),5)}.
\end{align*}
Finally,
\begin{align*}
\int_{1}^{T}|\zeta(1+it)|^{2}dt\leq & \ \frac{\pi^{2}}{6}T+\pi\sqrt{\frac{2}{3}}\sqrt{T}+\sage{roundup(Cz_1_u.upper(),5)}\cdot\log T \\
\int_{1}^{T}|\zeta(1+it)|^{2}dt\geq & \ \frac{\pi^{2}}{6}T-\pi\sqrt{\frac{2}{3}}\sqrt{T}+\sage{-roundup(Cz_1_l.upper(),5)}\cdot\log T. 
\end{align*}
\end{theorem}

\begin{proof}
We substitute $T_{1}=1$ inside either Proposition~\ref{prop:z-all} or Proposition~\ref{prop:z-allb}, according to which one gives us the best result. Our choice of $\rho$ for the upper bound will be the square root of the ratio between the leading terms of the expressions multiplying $1+\frac{1}{\rho}$ and $1+\rho$ respectively, the same choice with a negative sign corresponding to the lower bound. Such choice will be very close to the optimal one highlighted by Lemma~\ref{lem:handy}, but simpler and easier to handle.

For $\frac{1}{2}<\sigma<1$, Proposition~\ref{prop:z-allb} is the better alternative, as $\rho$ will be qualitatively smaller than in Proposition~\ref{prop:z-all} and the second order term will be of smaller order (the error term arising in the alternative case being of order $T^{\frac{3}{2}-\sigma}$). We set $\rho=\frac{1+D}{\sqrt{(2\sigma-1)\zeta(2\sigma)T}}$ (where $D$ is as in the proof of Lemma~\ref{le:353}, choosing $C=\lfloor T_{0}\rfloor$) and by imposing $T\geq T_{0}$ we merge all lower order terms, observing that the bound on $\zeta(2\sigma)$ given in Lemma~\ref{le:stieltjes} is being used; the condition $T_{0}=\sage{floor(TZ)}$ is employed to make sure that we actually get $-\rho>-1$, in order to apply Proposition~\ref{prop:z-allb} in the lower bound correctly.

When $\sigma\in\left\{\frac{1}{2},1\right\}$ the better alternative is Proposition~\ref{prop:z-all}: in the first case, the main terms obtained through Propositions \ref{prop:z-all} and \ref{prop:z-allb} are qualitatively the same but worse constants arise from Proposition~\ref{prop:z-allb}, while in the second case the same situation occurs for the error terms. For $\sigma=\frac{1}{2}$ we set $\rho=\frac{1}{\sqrt{\log T}}$ and for $\sigma=1$ we set $\rho=\frac{1}{\sqrt{\zeta(2)T}}$, and then impose $T\geq T_{0}$ to simplify the second order terms.
\end{proof}

\subsection{Extension of asymptotic formulas}

We prove here a proposition that allows us to extend the asymptotic formulas in the previous subsection to the case $\sigma<\frac{1}{2}$, via the functional equation \eqref{functional}.

\begin{proposition}\label{extension}
Let $\mathds{I}=[a_{0},a_{1}]$ be an interval of the real line ($a_{i}=\pm\infty$ is allowed). Let $Z:\mathds{I}\to\mathbb{R}_{\geq 0}$ be an integrable function such that, for every $T_1,T_2\in \mathds{I}$ with $T_1\leq T_2$,
\begin{align}\label{Asymp1}
F(T_1,T_2)-r^-(T_1,T_2)\ \leq\int_{T_1}^{T_2}Z(t)dt&\leq\  F(T_1,T_2)+r^+(T_1,T_2),
\end{align}
where $F$, $r^+$ and $r^-$ are non-negative real functions, such that $F$ is differentiable and, for every pair $T_1,T_2\in \mathds{I}$, $F(T_2,T_2)=F(T_1,T_1)=0$.

Let $f:\mathds{I}\to\mathbb{R}_{\geq 0}$ be a differentiable function with $f'$ integrable satisfying either $f'\geq 0$ or $f'\leq 0$ and such that either $f(a_{0})=0$ or $f(a_{1})=0$. We have the following cases.\\
\textbf{(i)}\ \ If $f(a_{0})=0$ (so $f'\geq 0$) and $\int_{T_1}^{T_2}\int_{a_{0}}^{T_2}|f'(u)|Z(t)dudt$ converges for every $T_1,T_2\in \mathds{I}$ with $T_1\leq T_2$, then
\begin{align*}
\int_{T_1}^{T_2}f(t)Z(t)dt&\leq\int_{a_{0}}^{T_2}\!\left(-f(u)\frac{\partial F(u,T_2)}{\partial u}+f'(u)r^{+}(u,T_2)\right)du,\\
\int_{T_1}^{T_2}f(t)Z(t)dt&\geq\int_{a_{0}}^{T_2}\!\left(-f(u)\frac{\partial F(u,T_2)}{\partial u}-f'(u)r^{-}(u,T_2)\right)du.
\end{align*}
\textbf{(ii)}\ \ If $f(a_{1})=0$ (so $f'\leq 0$) and $\int_{T_1}^{T_2}\int_{T_1}^{a_{1}}|f'(u)|Z(t)dudt$ converges for every $T_1,T_2\in \mathds{I}$ with $T_1\leq T_2$ and $\lim_{u\to a_{1}}f(u)F(T_1,u)=0$, then
\begin{align*}
\int_{T_1}^{T_2}f(t)Z(t)dt&\leq\int_{T_1}^{a_{1}}\left(f(u)\frac{\partial F(T_1,u)}{\partial u}-f'(u)r^{+}(T_1,u)\right)du,\\
\int_{T_1}^{T_2}f(t)Z(t)dt&\geq\int_{T_1}^{a_{1}}\left(f(u)\frac{\partial F(T_1,u)}{\partial u}+f'(u)r^{-}(T_1,u)\right)du.
\end{align*}
\end{proposition}

\begin{proof} As $f'$ is integrable, so is $|f|$. Suppose first that $f(a_{0})=0$; by the Fundamental Theorem of Calculus, for every $t\in [T_1,T_2]$, $f(t)=\int_{a_{0}}^tf'(u)du$. Then
\begin{equation*}\int_{T_1}^{T_2}f(t)Z(t)dt=\int_{T_1}^{T_2}\!\!\int_{a_{0}}^tf'(u)Z(t)dudt=\int_{T_1}^{T_2}\!\!\int_{T_1}^{T_2}\mathds{1}_{[a_{0},t]}(u)f'(u)Z(t)dudt.
\end{equation*}
Observe that, under the above conditions, $\mathds{1}_{[a_{0},t]}(u)=\mathds{1}_{[u,T_2]}(t)\mathds{1}_{[a_{0},T_2]}(u)$. Since the double integral $\int_{T_1}^{T_2}\int_{a_{0}}^{T_2}|f'(x)|Z(t)dxdt$ converges, by Fubini's Theorem, we can exchange the order of integration in the above equation and obtain
\begin{align*}
\int_{T_1}^{T_2}\!\!f(t)Z(t)dt & =\!\int_{a_{0}}^{T_2}\!\!f'(u)\!\int_{u}^{T_2}\!\!Z(t)dtdu\leq\!\int_{a_{0}}^{T_2}\!\!f'(u)(F(u,T_2)\!+\!r^{+}(u,T_2))du\\
& =-\int_{a_{0}}^{T_2}\!\!f(u)\frac{\partial F(u,T_2)}{\partial u}du+\int_{a_{0}}^{T_2}\!\!f'(u)r^{+}(u,T_2)du,
\end{align*}
where we have used integration by parts in the last step. We also derive the lower bound
\begin{align*}
&-\int_{a_{0}}^{T_2}f(u)\frac{\partial F(u,T_2)}{\partial u}du-\int_{a_{0}}^{T_2}f'(u)r^{-}(u,T_2)du.
\end{align*}

Case (ii) is obtained by proceeding in a similar manner as above, keeping in mind that $f(t)=-\int_t^{a_{1}}f'(u)du$ for $t\in [T_1,T_2]$ and $\mathds{1}_{[t,a_{1}]}(u)=\mathds{1}_{[T_1,u]}(t)\mathds{1}_{[T_1,a_{1}]}(u)$, and then using Fubini's Theorem and integration by parts. Here, the condition $\lim_{u\to a_{1}}f(u)F(T_1,u)=0$ is employed so as to make sure that if $a_{1}=\infty$, integration by parts is well-performed.
\end{proof}

The sign condition on $f'$ in Proposition \ref{extension} is not necessary; under the other conditions, one can derive an analogous result by writing $f'=f'_+-f'_-$, where $f_{\pm}=\max\{\pm f',0\}$. In that case, the point $a\in\mathds{I}$ such that $f(a)=0$ need not be an extremum of $\mathds{I}$, and if $T_1< a<T_2$ one can derive bounds by applying case (i) to $\int_{a}^{T_2}f(t)Z(t)dt$ and case (ii) to $\int_{T_1}^{a}f(t)Z(t)dt$.

\subsection{Mean value estimates of $\zeta(s)$ for $\Re(s)\in\left[0,\frac{1}{2}\right)$}

Thanks to Proposition~\ref{extension}, we are going to give asymptotic formulas for the integral of $|\zeta(\sigma+it)|^2$ in the case $0\leq\sigma<\frac{1}{2}$.

\tiny
\begin{mysage}
K1=RIF(e^(3/4)-1)
K2=RIF(2*K1+K1^2+2/e^pi+4*K1/e^pi+2*K1^2/e^pi)
#For checks about the following constants, see at the end of the file.
L_1_u_main = round( 74.37875 ,ndigits=5)
L_11_u_main = round( 27.84712 ,ndigits=5)
L_12_u_main = round( 42.13295 ,ndigits=5)
L_0_u_main = round( 35.53078 ,ndigits=5)
L_1_l_main = round( 66.11381 ,ndigits=5)
L_11_l_main = round( 27.84712 ,ndigits=5)
L_12_l_main = round( 40.43424 ,ndigits=5)
L_0_l_main = round( 15.44357 ,ndigits=5)
Cz_0_u_a_main = round( 4.01901 ,ndigits=5)
Cz_0_u_b_main = round( 7.41013 ,ndigits=5)
Cz_0_l_a_main = round( 0.09277 ,ndigits=5)
Cz_0_l_b_main = round( 7.41013 ,ndigits=5)
\end{mysage}
\normalsize

\begin{theorem}\label{asymptotic<1/2}
If $0<\sigma<\frac{1}{2}$ and $T\geq T_{0}=\sage{TZ}$, then
\begin{align*}
\int_1^T|\zeta(\sigma+it)|^2dt\leq & \ \frac{\zeta(2-2\sigma)}{(2\pi)^{1-2\sigma}(2-2\sigma)}T^{2-2\sigma}+L^{+}(\sigma)T, \\
\int_1^T|\zeta(\sigma+it)|^2dt\geq & \ \frac{\zeta(2-2\sigma)}{(2\pi)^{1-2\sigma}(2-2\sigma)}T^{2-2\sigma}-L^{-}(\sigma)T,
\end{align*}
where
\begin{align*}
L^{+}(\sigma)= & \ \frac{1}{(2\pi)^{1-2\sigma}}\left(\frac{\sage{L_11_u_main}}{\sigma^{2}}+\frac{\sage{L_1_u_main}}{\sigma}+\frac{\sage{L_12_u_main}}{1-2\sigma}+\sage{L_0_u_main}\right), \\
L^{-}(\sigma)= & \ \frac{1}{(2\pi)^{1-2\sigma}}\left(\frac{\sage{L_11_l_main}}{\sigma^{2}}+\frac{\sage{L_1_l_main}}{\sigma}+\frac{\sage{L_12_l_main}}{1-2\sigma}+\sage{L_0_l_main}\right).
\end{align*}

If $\sigma=0$ and $T\geq T_{0}=\sage{TZ}$, then
\begin{align*}
\int_1^T|\zeta(it)|^2dt\leq&\ \frac{\pi}{24}T^{2}+\sage{Cz_0_u_a_main}\cdot T\log T+\sage{Cz_0_u_b_main}\cdot T, \\
\int_1^T|\zeta(it)|^2dt\geq&\ \frac{\pi}{24}T^{2}-\sage{Cz_0_l_a_main}\cdot T\log T-\sage{Cz_0_l_b_main}\cdot T.
\end{align*}
\end{theorem}

\begin{proof}
Consider $\sigma$ such that $0\leq\sigma<\frac{1}{2}$. By using the functional equation \eqref{functional} of $\zeta$ and knowing that $|\zeta(s)|=|\zeta(\overline{s})|,|\Gamma(s)|=|\Gamma(\overline{s})|$, we readily see that
\begin{equation}\label{A}\int_1^T|\zeta(\sigma+it)|^2dt=\frac{1}{(2\pi)^{2-2\sigma}}\int_1^T\left|2\sin\left(\frac{\pi s}{2}\right)\Gamma(1-\sigma+it)\zeta(1-\sigma+it)\right|^2dt
\end{equation}
Let $s=\sigma+it$ with $t\geq 1$. For every complex number $z$ we have the identity $|\sin(z)|^{2}=\cosh^2(\Im(z))-\cos^2(\Re(z))$ (combine 4.5.7 and 4.5.54 in \cite{AS72}). Hence
\begin{align}\label{B}
\left|\sin\left(\frac{\pi s}{2}\right)\right|^{2}
&=\frac{e^{\pi t}}{4}\left(1+\frac{1}{e^{\pi t}}\left(2+\frac{1}{e^{\pi t}}-4\cos^2\left(\frac{\pi\sigma}{2}\right)\right)\right) \nonumber \\
&=\frac{e^{\pi t}}{4}\left(1+O^*\left(\frac{2}{e^{\pi t}}\right)\right),
\end{align}
since $\frac{1}{2}<\cos^2\left(\frac{\pi\sigma}{2}\right)\leq 1$ for the choice of $\sigma$. Moreover, using Corollary \ref{decay},
$
|\Gamma(1-\sigma+it)|=\sqrt{2\pi}t^{\frac{1}{2}-\sigma}e^{-\frac{\pi t}{2}}\exp\left(O^*\left(\frac{G_{1-\sigma}}{t}\right)\right),
$
where $G_{1-\sigma}=\frac{(1-\sigma)^3}{3}+\frac{(1-\sigma)^2}{2}\left(\frac{1}{2}-\sigma\right)+\frac{1}{6}\leq\frac{1}{3}+\frac{1}{4}+\frac{1}{6}=\frac{3}{4}$. We then verify that $\exp\left(O^*\left(\frac{G_{1-\sigma}}{t}\right)\right) =1+O^*\left(\frac{K_{1}}{t}\right)$, where $K_{1}=e^{\frac{3}{4}}-1$, as $t(e^{\frac{3}{4t}}-1)$ is decreasing for $t\geq 1$. This observation and \eqref{B} allow us to derive in \eqref{A} that
\begin{equation*}\label{C}\int_1^T|\zeta(\sigma+it)|^2dt=\frac{1}{(2\pi)^{1-2\sigma}}\int_1^T t^{1-2\sigma}\left|\zeta(1-\sigma+it)\right|^2\left(1+O^*\left(\frac{K_{2}}{t}\right)\right)dt,
\end{equation*}
where $K_{2}$ is defined as below
\begin{equation*}\left(1+\frac{2}{e^{\pi t}}\right)\!\left(1+\frac{K_{1}}{t}\right)^2\!\!\leq 1+\left(2K_{1}+K_{1}^2+\frac{2}{e^{\pi}}+\frac{4K_{1}}{e^{\pi}}+\frac{2K_{1}^2}{e^{\pi}}\right)\frac{1}{t}=1+\frac{K_{2}}{t},
\end{equation*}
since $\frac{e^{\pi t}}{t}$ is increasing for $t\geq 1$; we could do better, since the worst cases of \eqref{B} and $G_{1-\sigma}$ happen at different $\sigma$, but the advantage would be negligible. We conclude that
\begin{align}\label{awesome}
\int_1^T|\zeta(\sigma+it)|^2dt=&\ \frac{1}{(2\pi)^{1-2\sigma}}\int_1^T t^{1-2\sigma}\left|\zeta(1-\sigma+it)\right|^2dt\nonumber \\
&+O^*\left(\frac{K_{2}}{(2\pi)^{1-2\sigma}}\int_1^T \frac{\left|\zeta(1-\sigma+it)\right|^2}{t^{2\sigma}}dt\right).
\end{align}

\tiny
\begin{mysage}
#For explanations about these two lines, see the comments on 'Cz_1_string'.
Cz_1gen_u = Cz_fix_coeff(Cz_1_string(f11_1_u,f12_1_u,f13_1_u,rho_1_m12_0,1),8,2,0,1)
Cz_1gen_l = Cz_fix_coeff(Cz_1_string(f11_1_l,f12_1_l,f13_1_l,-rho_1_m12_0,1),8,2,0,-1)
###
#In the following, N is similar in spirit to Cz_121.
#However, we use Prop. 4.1 instead of Prop. 4.2 and the substitution sigma<->1-sigma,
#and more importantly the lower extremum u is generic. We have 1<=u<=T,
#and in each term we use either the simplification u-->1 or u-->T,
#whichever is the correct one to use for the bound; in practice, the latter is used
#only inside 1/rho (as rho is multiplied by 1/sqrt(u)).
###
#The entry [i][j][l][k] of each f1* is the coefficient of T^(-i+2j*sigma)*(log(T))^(1-l) times
#the following functions of sigma: 1 (for k=0), 1/sigma (for k=1), 1/(1-2*sigma) (for k=2),
#zeta(2-2*sigma) (for k=3). The entry [i][j][l][k] of N is the coefficient of
#T^(1-i/2+j*sigma)*(log(T))^(1-l) times the same functions of sigma as in f1* up to k=2
#(whereas zeta(2-2*sigma) can and will be removed through Lemma 2.12).
#We also assume that we are choosing rho of order T^(sigma-1/2)/sqrt(zeta(2-2*sigma)*u).
def N_string(f11,f12,f13,rho):
    N=[[[[0,0,0],[0,0,0]],[[0,0,0],[0,0,0]],[[0,0,0],[0,0,0]],[[0,0,0],[0,0,0]]], \
      [[[0,0,0],[0,0,0]],[[0,0,0],[0,0,0]],[[0,0,0],[0,0,0]],[[0,0,0],[0,0,0]]], \
      [[[0,0,0],[0,0,0]],[[0,0,0],[0,0,0]],[[0,0,0],[0,0,0]],[[0,0,0],[0,0,0]]], \
      [[[0,0,0],[0,0,0]],[[0,0,0],[0,0,0]],[[0,0,0],[0,0,0]],[[0,0,0],[0,0,0]]], \
      [[[0,0,0],[0,0,0]],[[0,0,0],[0,0,0]],[[0,0,0],[0,0,0]],[[0,0,0],[0,0,0]]], \
      [[[0,0,0],[0,0,0]],[[0,0,0],[0,0,0]],[[0,0,0],[0,0,0]],[[0,0,0],[0,0,0]]], \
      [[[0,0,0],[0,0,0]],[[0,0,0],[0,0,0]],[[0,0,0],[0,0,0]],[[0,0,0],[0,0,0]]], \
      [[[0,0,0],[0,0,0]],[[0,0,0],[0,0,0]],[[0,0,0],[0,0,0]],[[0,0,0],[0,0,0]]]]
    i=0
    while i<=2:
        j=0
        while j<=1:
            l=0
            while l<=1:
                #Before we fill N, we take care of zeta through Lemma 2.12.
                #For the terms with k=3, we have zeta(2-2*sigma)<1/(1-2*sigma)+H,
                #so we can absorb them into the terms with k=0,2.
                f11[i][j][l][2]=f11[i][j][l][2]+f11[i][j][l][3]
                f11[i][j][l][0]=f11[i][j][l][0]+H*f11[i][j][l][3]
                f12[i][j][l][2]=f12[i][j][l][2]+f12[i][j][l][3]
                f12[i][j][l][0]=f12[i][j][l][0]+H*f12[i][j][l][3]
                f13[i][j][l][2]=f13[i][j][l][2]+f13[i][j][l][3]
                f13[i][j][l][0]=f13[i][j][l][0]+H*f13[i][j][l][3]
                k=0
                while k<=2:
                    #Here we add the terms multiplied by 1*T.
                    N[2*i][2*j][l][k]=N[2*i][2*j][l][k]+f11[i][j][l][k]
                    #Here we add the terms multiplied by rho*T.
                    #Inside rho, we operate the substitution u-->1
                    #and we have 1/sqrt(zeta(2-2*sigma))<1 by Lemma 2.12.
                    N[2*i+1][2*j+1][l][k]=N[2*i+1][2*j+1][l][k]+rho*f11[i][j][l][k]
                    #Here we add the terms multiplied by 1.
                    N[2*i+2][2*j][l][k]=N[2*i+2][2*j][l][k]+ \
                      (E/2-1)*f11[i][j][l][k]+E*f12[i][j][l][k]+f13[i][j][l][k]
                    #Here we add the terms multiplied by rho.
                    #Inside rho, we operate the substitution u-->1
                    #and we have 1/sqrt(zeta(2-2*sigma))<1 by Lemma 2.12.
                    N[2*i+3][2*j+1][l][k]=N[2*i+3][2*j+1][l][k]+ \
                      rho*((E/2-1)*f11[i][j][l][k]+E*f12[i][j][l][k]+f13[i][j][l][k])
                    if 2*j-1>=0 and k==0:
                         #Here we add the terms multiplied by 1/rho.
                         #We remark that 2*j-1>=0 and k=0 are both satisfied
                         #for any nonzero coefficient we are provided with
                         #(the result would be quite wrong if they were not).
                         #Also, for 1/rho we remind that we use the substitution u-->T,
                         #so that its order with respect to the variable T is T^(1-sigma);
                         #finally, we have sqrt(zeta(2-2*sigma))<1/(1-2*sigma)+sqrt(H)
                         #using Lemma 2.12, so that we can absorb these terms
                         #into the entries of N with k=0,2.
                        N[2*i][2*j-1][l][2]=N[2*i][2*j-1][l][2]+1/rho*f13[i][j][l][k]
                        N[2*i][2*j-1][l][0]=N[2*i][2*j-1][l][0]+1/rho*sqrt(H)*f13[i][j][l][k]
                    k=k+1
                l=l+1
            j=j+1
        i=i+1
    return N
###
#The following code returns the resulting coefficients after merging all the small
#error terms in N, where N is the result of N_string. The coefficients are those
#of T^(2*sigma) times the functions in the definition of N_string.
#The entry [i][j][l][k] of N is the coefficient of T^(1-i/2+j*sigma)*(log(T))^(1-l)
#times the following functions of sigma: 1 (for k=0), 1/sigma (for k=1), 1/(1-2*sigma) (for k=2).
#In the N in the input, T^(2*sigma) is the order of the entries N[2][2][1][k].
def N_coeff(N,sign): #If sign=1 we are looking for the upper bound,
                     #if sign=-1 for the lower bound.
    S=[0,0,0]
    i=0
    while i<=7:
        j=0
        while j<=3:
            l=0
            while l<=1:
                k=0
                while k<=2:
                    if i<2 or j-i>0 or ((i==2 or j-i==0) and l==0):
                        #We first remove the terms of order higher than T^(2*sigma).
                        #The terms of order arbitrarily close from below to T^(2*sigma)
                        #are 1, T^(sigma), T^(2*sigma), T^(3*sigma-1/2), so they all sit in
                        #either i=2 or j-i=0; therefore we remove the terms with higher
                        #exponents, i.e. i<2 or j-i>0, or with the same exponents but
                        #one more log(T) factor, i.e. also l=0.
                        summand=0
                    elif i==2 and j==2 and l==1:
                        #This is the exact term T^(2*sigma).
                        summand=N[i][j][l][k]
                    else:
                        #Now we remove the small contributions that are not in N.
                        #The following two are zeta(2-2*sigma)*(-u).
                        #The error terms in the upper (resp. lower) bound are positive
                        #(resp. negative), but the coefficients that are being fed
                        # to the procedure are provided in absolute value,
                        #while the sign of zeta(2-2*sigma)*(-u)
                        #does not change in the two directions of the bound:
                        #so to remove the contribution we have to add (resp. subtract) it,
                        #and 'sign' is 1 (resp. -1) for the upper (resp. lower) bound.
                        #Also, by Lemma 2.12 we have zeta(2-2*sigma)<1/(1-2*sigma)+H,
                        #so that the removal happens inside the entries of N with k=0,2:
                        #that is where the term we are removing had been sent
                        #during the production of N inside the procedure N_string.
                        if i==2 and j==0 and l==1 and k==2:
                            N[i][j][l][k]=N[i][j][l][k]+1*sign
                        if i==2 and j==0 and l==1 and k==0:
                            N[i][j][l][k]=N[i][j][l][k]+H*sign
                        #As for the other small contributions not in N,
                        #they are not small: as observed in the proof of Thm. 4.5,
                        #the terms depending on u inside (4.18) are exactly those
                        #that would give error terms larger than T^(2*sigma).
                        #2*sqrt(zeta(2-2*sigma))*T^(1/2+sigma)/sqrt(u) is large.
                        #2*sqrt(zeta(2-2*sigma))*D*sqrt(u)/T^(1/2-sigma)*log(T/u) is large
                        #because sqrt(u)-->sqrt(T), since it comes from 1/rho.
                        #Finally we deal with all the other terms.
                        #Every term that is of smaller order than T^(2*sigma)
                        #comes with a correction that makes it
                        #impact less on the final product.
                        correction=TZ^(1/2*max(2-i,j-i))*(log(TZ))^(1-l)
                        #For T<e^2 though, log(T)/sqrt(T) is not decreasing,
                        #so we need a correction to the correction:
                        #in this case, the best we can do is taking the maximum
                        #of the function log(x)/sqrt(x) as our 'correction',
                        #which occurs in e^2.
                        if min(2-i,j-i)==-1 and l==0 and TZ<e^2:
                            correction=2/e
                        #We must also ignore the negative contributions.
                        summand=max(0, N[i][j][l][k]*correction )
                    S[k]=S[k]+summand
                    k=k+1
                l=l+1
            j=j+1
        i=i+1
    return [RIF(S[0]),RIF(S[1]),RIF(S[2])]
###
rhoN_sm12_0_msz = RIF( 1 ) #Coefficient of T^(sigma-1/2)/sqrt(u*zeta(2-2*sigma)) inside rho.
#The following strings collect the coefficients of the f1* from Prop. 4.1 when 1/2<sigma<1
#(and after the use of the proposition we apply the substitution sigma<->1-sigma),
#where the entry [i][j][l][k] of each f1* is the coefficient of
#T^(-i+2j*sigma)*(log(T))^(1-l) times the following functions of sigma:
#1 (for k=0), 1/sigma (for k=1), 1/(1-2*sigma) (for k=2), zeta(2-2*sigma) (for k=3).
#The 'u,l' indicate upper and lower bounds, respectively.
f11_121_u = [[[[ RIF( 0 ) , RIF( 0 ) , RIF( 0 ) , RIF( 0 ) ], \
[ RIF( 0 ) , RIF( 0 ) , RIF( 0 ) , RIF( 1 ) ]], \
[[ RIF( 0 ) , RIF( 0 ) , RIF( 0 ) , RIF( 0 ) ], \
[ RIF( 0 ) , RIF( 0 ) , RIF( 0 ) , RIF( 0 ) ]]], \
[[[ RIF( 0 ) , RIF( 0 ) , RIF( 0 ) , RIF( 0 ) ], \
[ RIF( 0 ) , RIF( 0 ) , RIF( 0 ) , RIF( 0 ) ]], \
[[ RIF( 0 ) , RIF( 0 ) , RIF( 0 ) , RIF( 0 ) ], \
[ RIF( 0 ) , RIF( 0 ) , RIF( -1 ) , RIF( 0 ) ]]], \
[[[ RIF( 0 ) , RIF( 0 ) , RIF( 0 ) , RIF( 0 ) ], \
[ RIF( 0 ) , RIF( 0 ) , RIF( 0 ) , RIF( 0 ) ]], \
[[ RIF( 0 ) , RIF( 0 ) , RIF( 0 ) , RIF( 0 ) ], \
[ RIF( 1 ) , RIF( 0 ) , RIF( 0 ) , RIF( 0 ) ]]]]
f12_121_u = [[[[ RIF( 0 ) , RIF( 0 ) , RIF( 0 ) , RIF( 0 ) ], \
[ RIF( 1/2 ) , RIF( 0 ) , RIF( 0 ) , RIF( 0 ) ]], \
[[ RIF( 0 ) , RIF( 0 ) , RIF( 0 ) , RIF( 0 ) ], \
[ RIF( 0 ) , RIF( 1/2 ) , RIF( 0 ) , RIF( 0 ) ]]], \
[[[ RIF( 0 ) , RIF( 0 ) , RIF( 0 ) , RIF( 0 ) ], \
[ RIF( 0 ) , RIF( 0 ) , RIF( 0 ) , RIF( 0 ) ]], \
[[ RIF( 0 ) , RIF( 0 ) , RIF( 0 ) , RIF( 0 ) ], \
[ RIF( 0 ) , RIF( 0 ) , RIF( 0 ) , RIF( 0 ) ]]], \
[[[ RIF( 0 ) , RIF( 0 ) , RIF( 0 ) , RIF( 0 ) ], \
[ RIF( 0 ) , RIF( 0 ) , RIF( 0 ) , RIF( 0 ) ]], \
[[ RIF( 0 ) , RIF( 0 ) , RIF( 0 ) , RIF( 0 ) ], \
[ RIF( 0 ) , RIF( 0 ) , RIF( 0 ) , RIF( 0 ) ]]]]
f13_121_u = [[[[ RIF( 0 ) , RIF( 0 ) , RIF( 0 ) , RIF( 0 ) ], \
[ RIF( 0 ) , RIF( 0 ) , RIF( 0 ) , RIF( 0 ) ]], \
[[ RIF( 0 ) , RIF( 0 ) , RIF( 0 ) , RIF( 0 ) ], \
[ RIF( 1 ) , RIF( 0 ) , RIF( 0 ) , RIF( 0 ) ]]], \
[[[ RIF( 0 ) , RIF( 0 ) , RIF( 0 ) , RIF( 0 ) ], \
[ RIF( 0 ) , RIF( 0 ) , RIF( 0 ) , RIF( 0 ) ]], \
[[ RIF( 2*Dz ) , RIF( 0 ) , RIF( 0 ) , RIF( 0 ) ], \
[ RIF( -1+Dz^2 ) , RIF( 0 ) , RIF( 0 ) , RIF( 0 ) ]]], \
[[[ RIF( 0 ) , RIF( 0 ) , RIF( 0 ) , RIF( 0 ) ], \
[ RIF( 0 ) , RIF( 0 ) , RIF( 0 ) , RIF( 0 ) ]], \
[[ RIF( 0 ) , RIF( 0 ) , RIF( 0 ) , RIF( 0 ) ], \
[ RIF( -Dz^2 ) , RIF( 0 ) , RIF( 0 ) , RIF( 0 ) ]]]]
f11_121_l = [[[[ RIF( 0 ) , RIF( 0 ) , RIF( 0 ) , RIF( 0 ) ], \
[ RIF( 0 ) , RIF( 0 ) , RIF( 0 ) , RIF( -1 ) ]], \
[[ RIF( 0 ) , RIF( 0 ) , RIF( 0 ) , RIF( 0 ) ], \
[ RIF( 0 ) , RIF( 0 ) , RIF( 0 ) , RIF( 0 ) ]]], \
[[[ RIF( 0 ) , RIF( 0 ) , RIF( 0 ) , RIF( 0 ) ], \
[ RIF( 0 ) , RIF( 0 ) , RIF( 0 ) , RIF( 0 ) ]], \
[[ RIF( 0 ) , RIF( 0 ) , RIF( 0 ) , RIF( 0 ) ], \
[ RIF( 0 ) , RIF( 0 ) , RIF( 1 ) , RIF( 0 ) ]]], \
[[[ RIF( 0 ) , RIF( 0 ) , RIF( 0 ) , RIF( 0 ) ], \
[ RIF( 0 ) , RIF( 0 ) , RIF( 0 ) , RIF( 0 ) ]], \
[[ RIF( 0 ) , RIF( 0 ) , RIF( 0 ) , RIF( 0 ) ], \
[ RIF( 1/2 ) , RIF( 0 ) , RIF( 0 ) , RIF( 0 ) ]]]]
f12_121_l = [[[[ RIF( 0 ) , RIF( 0 ) , RIF( 0 ) , RIF( 0 ) ], \
[ RIF( 0 ) , RIF( 1/2 ) , RIF( 0 ) , RIF( 0 ) ]], \
[[ RIF( 0 ) , RIF( 0 ) , RIF( 0 ) , RIF( 0 ) ], \
[ RIF( 0 ) , RIF( -1/2 ) , RIF( 0 ) , RIF( 0 ) ]]], \
[[[ RIF( 0 ) , RIF( 0 ) , RIF( 0 ) , RIF( 0 ) ], \
[ RIF( 0 ) , RIF( 0 ) , RIF( 0 ) , RIF( 0 ) ]], \
[[ RIF( 0 ) , RIF( 0 ) , RIF( 0 ) , RIF( 0 ) ], \
[ RIF( 1/2 ) , RIF( 0 ) , RIF( 0 ) , RIF( 0 ) ]]], \
[[[ RIF( 0 ) , RIF( 0 ) , RIF( 0 ) , RIF( 0 ) ], \
[ RIF( 0 ) , RIF( 0 ) , RIF( 0 ) , RIF( 0 ) ]], \
[[ RIF( 0 ) , RIF( 0 ) , RIF( 0 ) , RIF( 0 ) ], \
[ RIF( 0 ) , RIF( 0 ) , RIF( 0 ) , RIF( 0 ) ]]]]
#In f12_121_l[0][0][1][1] we have used
#zeta(1-2*sigma)>-1/(2*sigma) for 0<sigma<1/2 by Lemma 2.12.
f13_121_l = [[[[ RIF( 0 ) , RIF( 0 ) , RIF( 0 ) , RIF( 0 ) ], \
[ RIF( 0 ) , RIF( 0 ) , RIF( 0 ) , RIF( 0 ) ]], \
[[ RIF( 0 ) , RIF( 0 ) , RIF( 0 ) , RIF( 0 ) ], \
[ RIF( -1 ) , RIF( 0 ) , RIF( 0 ) , RIF( 0 ) ]]], \
[[[ RIF( 0 ) , RIF( 0 ) , RIF( 0 ) , RIF( 0 ) ], \
[ RIF( 0 ) , RIF( 0 ) , RIF( 0 ) , RIF( 0 ) ]], \
[[ RIF( -2*Dz ) , RIF( 0 ) , RIF( 0 ) , RIF( 0 ) ], \
[ RIF( 1-Dz^2 ) , RIF( 0 ) , RIF( 0 ) , RIF( 0 ) ]]], \
[[[ RIF( 0 ) , RIF( 0 ) , RIF( 0 ) , RIF( 0 ) ], \
[ RIF( 0 ) , RIF( 0 ) , RIF( 0 ) , RIF( 0 ) ]], \
[[ RIF( 0 ) , RIF( 0 ) , RIF( 0 ) , RIF( 0 ) ], \
[ RIF( Dz^2 ) , RIF( 0 ) , RIF( 0 ) , RIF( 0 ) ]]]]
#The terms N_* collect the error terms from T^(2*sigma) downwards.
N_u = N_coeff(N_string(f11_121_u,f12_121_u,f13_121_u,rhoN_sm12_0_msz),1)
N_l = N_coeff(N_string(f11_121_l,f12_121_l,f13_121_l,-rhoN_sm12_0_msz),-1)
\end{mysage}
\normalsize

To estimate \eqref{awesome}, we could resort to Proposition \ref{extension}, using the functions $t\mapsto t^{1-2\sigma}$, $t\mapsto t^{-2\sigma}$ and the bounds for $\int_1^T|\zeta(1-\sigma+it)|^2dt$ given in Theorem \ref{th:z->=1/2}. This approach, while simpler, produces less accurate second order terms. One can do better by studying $\int_u^T|\zeta(1-\sigma+it)|^2dt$ for $1\leq u\leq T$. We proceed as in Theorem~\ref{th:z->=1/2}, with the general bound of Proposition~\ref{prop:z-all}. Set $\rho=\frac{T^{\sigma-1/2}}{\sqrt{\zeta(2-2\sigma)u}}$: the dependence on $u$ allows us to use Proposition~\ref{extension} non-trivially, yielding sharper estimates, while $\rho$ in Proposition~\ref{prop:z-allb} depends solely on $T$. Afterwards, we merge the second order terms according to either $u\geq 1$ or $u\leq T$, recalling Lemma~\ref{le:stieltjes} and $T\geq T_{0}$. The final bounds are
\begin{equation}\label{asyy}
-r^-(u,T)\leq\int_u^T|\zeta(1-\sigma+it)|^2dt-\zeta(2-2\sigma)(T-u)\leq r^+(u,T),
\end{equation} 
where
\begin{align}\label{remainder}
r^{\pm}(u,T)\!= & \begin{cases}2\sqrt{\zeta(2-2\sigma)}\left(\frac{T^{\frac{1}{2}+\sigma}}{\sqrt{u}}+\frac{D\sqrt{u}}{T^{\frac{1}{2}-\sigma}}\log\left(\frac{T}{u}\right)\right)\!+\!N^{\pm}(\sigma)T^{2\sigma} & \text{if $\sigma>0$}, \\
\pi\sqrt{\frac{2T}{3u}}+W^{\pm}\log T & \text{if $\sigma=0$},
\end{cases}
\end{align}
and
\begin{align*}
N^{+}(\sigma)= & \ \frac{\sage{roundup(N_u[1].upper(),5)}}{\sigma}+\frac{\sage{roundup(N_u[2].upper(),5)}}{1-2\sigma}+\sage{roundup(N_u[0].upper(),5)}, & W^{+}= & \ \sage{roundup(Cz_1gen_u.upper(),5)}, \\
N^{-}(\sigma)= & \ \frac{\sage{roundup(N_l[1].upper(),5)}}{\sigma}+\frac{\sage{roundup(N_l[2].upper(),5)}}{1-2\sigma}+\sage{roundup(N_l[0].upper(),5)}, & W^{-}= & \ \sage{roundup(Cz_1gen_l.upper(),5)}.
\end{align*}
As remarked, the terms in $u$ in the definition of $r^{\pm}(u,T)$ are those that would have otherwise given larger error terms if we had taken $r^{\pm}$ independent of $u$.

We further verify by \eqref{asyy} that the conditions of Proposition \ref{extension} are met with the increasing function $f(t)=t^{1-2\sigma}-1$, $Z(t)=|\zeta(1-\sigma+it)|^2$ and $a_{0}=1$ (we cannot use $f(t)=t^{1-2\sigma}$ directly as \eqref{asyy} is only valid for $u\geq 1$). We split the integral as
\begin{equation*}
\int_{1}^{T}\!\!t^{1-2\sigma}|\zeta(1-\sigma+it)|^2dt=\!\int_{1}^{T}\!\!(t^{1-2\sigma}-1)|\zeta(1-\sigma+it)|^2dt+\int_{1}^{T}\!\!|\zeta(1-\sigma+it)|^2 dt,
\end{equation*}
and the second integral is already bounded by \eqref{asyy}. For the first, we thus apply Proposition~\ref{extension}(i) using the bound in \eqref{asyy} as
\begin{align*}
& \ \int_{1}^{T}(t^{1-2\sigma}-1)|\zeta(1-\sigma+it)|^2dt\leq\int_{0}^{T}u^{1-2\sigma}\zeta(2-2\sigma) \\
& \ +(1-2\sigma)\left(2\sqrt{\zeta(2-2\sigma)}T^{\frac{1}{2}+\sigma}u^{-\sigma-\frac{1}{2}}
+ \ N^{+}(\sigma)T^{2\sigma}u^{-2\sigma}\right)du \\&+\int_{1}^{T}2\sqrt{\zeta(2-2\sigma)}DT^{\frac{1}{2}-\sigma}u^{\frac{1}{2}}\log\left(\frac{T}{u}\right)-\zeta(2-2\sigma)du \\
= & \ \frac{\zeta(2-2\sigma)}{2-2\sigma}T^{2-2\sigma}+4\sqrt{\zeta(2-2\sigma)}T+N^{+}(\sigma)T \\
& + \ (1-2\sigma)2\sqrt{\zeta(2-2\sigma)}D\frac{T^{1-\sigma}-T^{\sigma-\frac{1}{2}}}{\left(\frac{3}{2}-2\sigma\right)^{2}}-\zeta(2-2\sigma)(T-1).
\end{align*}

\tiny
\begin{mysage}
#The following constants come from the reasoning leading from N^+,N^-
#to S^+,S^-: read the proof for details. By Lemma 2.12 we have
#zeta(2-2*sigma)<1/(1-2*sigma)+H and sqrt(zeta(2-2*sigma))<1/(1-2*sigma)+sqrt(H).
#Also (1-2*sigma)/(3/2-2*sigma)^2<=1/2, and in the lower bound 1/(2-2*sigma)<1.
S_u___1= RIF( 2*N_u[1] )
S_u___12 = RIF( 2*N_u[2]+6+3*Dz )
S_u___0 = RIF( 2*N_u[0]+6*sqrt(H)+3*Dz*sqrt(H) )
S_l___1 = RIF( 2*N_l[1] )
S_l___12 = RIF( 2*N_l[2]+6+3*Dz+1 )
S_l___0 = RIF( 2*N_l[0]+6*sqrt(H)+3*Dz*sqrt(H)+H )
\end{mysage}
\normalsize

We proceed similarly for the lower bound (a term $-\frac{\zeta(2-2\sigma)}{2-2\sigma}$ emerges in that case from the approximations) and for $\sigma=0$. Using also Lemma~\ref{le:stieltjes} and $\frac{1-2\sigma}{\left(\frac{3}{2}-2\sigma\right)^{2}}\leq\frac{1}{2}$, we obtain
\begin{align*}
\int_{1}^{T}t^{1-2\sigma}|\zeta(1-\sigma+it)|^2dt&\leq\frac{\zeta(2-2\sigma)}{2-2\sigma}T^{2-2\sigma}+S^{+}(\sigma,T), \\
\int_{1}^{T}t^{1-2\sigma}|\zeta(1-\sigma+it)|^2dt&\geq\frac{\zeta(2-2\sigma)}{2-2\sigma}T^{2-2\sigma}-S^{-}(\sigma,T)
\end{align*}
where
\begin{align*}
S^{+}(\sigma,T)= & \begin{cases}\left(\frac{\sage{roundup(S_u___1.upper(),5)}}{\sigma}+\frac{\sage{roundup(S_u___12.upper(),5)}}{1-2\sigma}+\sage{roundup(S_u___0.upper(),5)}\right)T & \text{if $0<\sigma<\frac{1}{2}$},\\
W^{+} T\log T+2\pi\sqrt{\frac{2}{3}}T & \text{if $\sigma=0$},
\end{cases} \\
S^{-}(\sigma,T)= & \begin{cases}\left(\frac{\sage{roundup(S_l___1.upper(),5)}}{\sigma}+\frac{\sage{roundup(S_l___12.upper(),5)}}{1-2\sigma}+\sage{roundup(S_l___0.upper(),5)}\right)T & \text{if $0<\sigma<\frac{1}{2}$},\\
 W^{-} T\log T+2\pi\sqrt{\frac{2}{3}}T & \text{if $\sigma=0$}.
\end{cases}
\end{align*}

Finally, for the error term of \eqref{awesome}, the conditions of Proposition~\ref{extension} are not met with $f(t)=t^{-2\sigma}$ and $0<\sigma<\frac{1}{2}$. Instead, we apply the weaker bound $t^{-2\sigma}<1$, sufficient to have an error term of order $T$, and use Theorem~\ref{th:z->=1/2} with $1-\sigma$ instead of $\sigma$.
\end{proof}

\subsection{Square mean of $\frac{\zeta(s)}{s}$ on tails: asymptotically sharp bounds}

We will use the bounds for $\zeta(s)$ given in the previous sections and the machinery of Proposition~\ref{extension} to retrieve upper bounds for $\frac{\zeta(s)}{s}$.

\tiny
\begin{mysage}
#For checks about the following constants, see at the end of the file.
C_12_a_main = round( 4.0 ,ndigits=5)
C_12_b_main = round( 49.8039 ,ndigits=5)
C_1_a_main = round( 25.25217 ,ndigits=5)
C_1_b_main = round( 16.04622 ,ndigits=5)
\end{mysage}
\normalsize

\begin{theorem}\label{th:z/s-all}
Let $T\geq T_{0}=\sage{TZ}$; let $L^{+},N^{+}$ be as in Theorem~\ref{asymptotic<1/2}, and let $D$ be defined as in Lemma~\ref{le:353} with $C=\lfloor T_{0}\rfloor$. Then $\int_T^\infty\frac{|\zeta(s)|^2}{|s|^2}dt$ is bounded from above by
\begin{align*}
& \frac{\zeta(2-2\sigma)}{2\sigma(2\pi)^{1-2\sigma}}\cdot\frac{1}{T^{2\sigma}}+2L^{+}(\sigma)\cdot\frac{1}{T} & & \text{if \ $0<\sigma<\frac{1}{2}$,} \\
& \frac{\log T}{T}+\sage{C_12_a_main}\cdot\frac{\sqrt{\log T}}{T}+\sage{C_12_b_main}\cdot\frac{1}{T} & & \text{if \ $\sigma=\frac{1}{2},$} \\
& \zeta(2\sigma)\cdot\frac{1}{T}+\left(2N^{+}(1-\sigma)+(D+4)\sqrt{\zeta(2\sigma)}\right)\cdot\frac{1}{T^{2\sigma}} & & \text{if \ $\frac{1}{2}<\sigma<1$,} \\
& \frac{\pi^2}{6}\cdot\frac{1}{T}+\sage{C_1_a_main}\cdot \frac{\log T}{T^2}+\sage{C_1_b_main}\cdot \frac{1}{T^{2}} & & \text{if \ $\sigma=1$.}
\end{align*}
\end{theorem}

\begin{proof}
We apply case (ii) of Proposition~\ref{extension} by taking $T_1=T\geq 1$, $T_2=\infty$, $f(t)=\frac{1}{\sigma^2+t^2}$, and $a_{1}=\infty$. Using the bounds $u^{2}<\sigma^{2}+u^{2}$ and $(-(\sigma^2+u^2)^{-1})'<2u^{-3}$, we get
\begin{equation*}
\int_{T}^{\infty}\frac{|\zeta(s)|^2}{|s|^2}dt<\int_{T}^{\infty}\left(\frac{1}{u^{2}}\frac{\partial F(T,u)}{\partial u}+\frac{2}{u^{3}}r^{+}(T,u)\right)du,
\end{equation*}
for appropriate choices of $F$ and $r^{+}$, which are taken as follows. 

For $0<\sigma<\frac{1}{2}$, we use Theorem~\ref{asymptotic<1/2} and the observation that the integral of $|\zeta(\sigma+it)|^{2}$ in $[T,u]$ is bounded by the integral in $[1,u]$. For $\sigma=\frac{1}{2}$, we use Theorem~\ref{th:z->=1/2} and the same observation. For $\frac{1}{2}<\sigma\leq 1$, we use the upper bound in \eqref{asyy} replacing $\sigma$ by $1-\sigma$.
\end{proof}

When $\sigma=0$, note by Proposition~\ref{extension} that the main term of $\int_T^\infty\frac{|\zeta(s)|^2}{|s|^2}dt$ is $\frac{\pi^{2}}{12}\int_T^{\infty}\frac{u}{\sigma^2+u^2}du=\frac{\pi^{2}}{12}\cdot\left.\frac{1}{2}\log(\sigma^2+u^2)\right|_T^\infty=\infty$, so that the integral is divergent.

\section{Numerical considerations}\label{Num}

\tiny
\begin{mysage}
def bisect(f1,f2,Tmin,Tmax,sigma):
    while Tmax-Tmin>1:
        Tmed=floor((Tmax+Tmin)/2)
        if RIF(f1.substitute(T=Tmed))/RIF(f2.substitute(T=Tmed))>RIF(1):
            Tmax=Tmed
        else:
            Tmin=Tmed
        if Tmin>=10^10 and sigma!=1/2:
            siz=floor(log(Tmax)/log(10))
            piemax=ceil(Tmax/10^(siz-2))
            piemin=floor(Tmin/10^(siz-2))
            if piemax-piemin<=1:
                Tmin=Tmax
    if Tmax>10^50:
        return '>10^50'
    elif Tmax<200:
        return '<200'
    else:
        if Tmax<10^10 or sigma==1/2:
            return Tmax
        else:
            return [piemax/100,siz]
def find_threshold(ind,sigma,Tmin,Tmax):
    var('s,T')
    if sigma>1/2 and sigma<1:
        if ind==0: #if the threshold is for the versions in Thm. 1.1
            Numer=3/5*pi*zeta(2*sigma)+ \
                  (C_121_num_a_main-C_121_num_b_main/(sigma-1/2))*T^(1-2*sigma)
            Asymp=zeta(2*sigma)+C_121_opt_main/((sigma-1/2)*(1-sigma))*T^(1-2*sigma)
        elif ind==1: #if the threshold is for the versions in Thms. 3.1 and 4.6
            Numer=pi*(3/5*zeta(2*sigma)+kpar_main^sigma* \
                  (k111_main/sigma+k112_main/(2*sigma+1)+k113_main/(sigma+1)-k114_main/(2*sigma-1))* \
                  1/T^(2*sigma-1)+kpar_main^sigma*k12x_main/T^(2*sigma))
            Asymp=zeta(2*sigma)+(2*(roundup(N_u[1].upper(),5)/sigma+ \
                  roundup(N_u[2].upper(),5)/(1-2*sigma)+roundup(N_u[0].upper(),5))+ \
                  (Dz+4)*sqrt(zeta(2*sigma)))*T^(1-2*sigma)
        else:
            return 'Error'
    elif sigma>0 and sigma<1/2:
        if ind==0: #if the threshold is for the version in Thm. 1.1
            Numer=(C_012_num_a_main/sigma+C_012_num_b_main/(1/2-sigma)+C_012_num_c_main)+ \
                  C_012_num_d_main*abs(zeta(2*sigma))*T^(2*sigma-1)
            Asymp=zeta(2-2*sigma)/(2*sigma*(2*pi)^(1-2*sigma))+ \
                  C_012_opt_main/(sigma^2*(1/2-sigma))*T^(2*sigma-1)
        elif ind==1: #if the threshold is for the versions in Thms. 3.1 and 4.6
            Numer=pi*(kpar_main^sigma*(k111_main/sigma+k112_main/(2*sigma+1)+ \
                  k113_main/(sigma+1)+k314_main/(1-2*sigma))+c30x_main*zeta(2*sigma)* \
                  1/T^(1-2*sigma)+kpar_main^sigma*k12x_main/T)
            Asymp=zeta(2-2*sigma)/(2*sigma*(2*pi)^(1-2*sigma))+2*1/(2*pi)^(1-2*sigma)* \
                  (L_11_u_main/sigma^2+L_1_u_main/sigma+L_12_u_main/(1-2*sigma)+L_0_u_main)*T^(2*sigma-1)
        else:
            return 'Error'
    elif sigma==1/2: #the ind can be anything
        Numer=pi*(3/5*log(T)+c21x_main+c22x_main/T)
        Asymp=log(T)+C_12_a_main*sqrt(log(T))+C_12_b_main
    else:
        return 'Error'
    if (RIF(Numer.substitute(T=Tmin))/RIF(Asymp.substitute(T=Tmin))>=RIF(1) and Tmin>=200) \
    or RIF(Numer.substitute(T=Tmax))/RIF(Asymp.substitute(T=Tmax))<=RIF(1):
        return 'Error'
    elif RIF(Numer.substitute(T=Tmin))/RIF(Asymp.substitute(T=Tmin))>=RIF(1) and Tmin<200:
        return '<200'
    return bisect(Numer,Asymp,Tmin,Tmax,sigma)
\end{mysage}
\normalsize

In case \eqref{bigthm1} of Theorem~\ref{thm:main}, we only show the bound from Theorem~\ref{th:z/s-all}, since it is always stronger than the one from Theorem~\ref{thzeta}. In case \eqref{bigthm12}, we chose $T=10^{\sage{Exp_12}}$ because the threshold where the second bound is better than the first sits in $(10^{\sage{Exp_12-1}},10^{\sage{Exp_12}}]$.

In case \eqref{bigthm121}, $T=\sage{T0_121}$ is the lowest integer at which for some $\sigma\in\left(\frac{1}{2},1\right)$ the second bound is stronger than the first. In Table~\ref{ta:sigma>1/2} we give thresholds $T$ for all $\sigma\in\frac{1}{20}\mathbb{N}\cap\left(\frac{1}{2},1\right)$. We also present thresholds for $\sigma\in\frac{1}{100}\mathbb{N}\cap\left(\frac{1}{2},\frac{2}{5}\right)$ for the bounds of Theorems~\ref{thzeta} and \ref{th:z/s-all}: for the same $\sigma$, these sharper bounds yield a lower $T$ than the ones from Theorem~\ref{thm:main}.

In case \eqref{bigthm012}, $T=\sage{T0_012}$ is the lowest integer at which for some $\sigma\in\left(0,\frac{1}{2}\right)$ the second bound is stronger than the first. Table~\ref{ta:sigma<1/2} gives thresholds between the bounds of Theorem~\ref{thm:main} or between those in Theorems~\ref{thzeta} and~\ref{th:z/s-all}.

In the tables, the significant digits of the higher entries of $T$ have been reduced for simplicity. To obtain the reported approximations, the threshold has been rounded up.

Lastly: the loss of precision in Theorems~\ref{thm:main} and \ref{thm:mainz} with respect to Theorems~\ref{thzeta}, \ref{th:z/s-all}, \ref{th:z->=1/2} and \ref{asymptotic<1/2} may be significant, especially for $\sigma\not\in\left\{0,\frac{1}{2},1\right\}$. In \S\ref{Int}, we favored simplicity in the statements, provided that they showed the correct asymptotics for the main terms and the correct order of the error terms for $T\rightarrow\infty$ and $\sigma$ tending to $0,\frac{1}{2},1$. Readers wanting sharper bounds are advised to rely on the stronger estimates of \S\ref{Exp} and \S\ref{Opt}.
\bigskip
\\
\begin{minipage}[c]{0.53\linewidth}
\centering
\begin{tabular}{|l|l|l|l|l|}
\hline
$\sigma$ & Th.~\ref{thm:main} & \!\!\!\! & $\sigma$ & Th.~\ref{thzeta}-\ref{th:z/s-all} \\ \hline
$0.55$\! & $\approx\sage{round( find_threshold(0,55/100,100,10^25)[0] ,ndigits=2)}\!\cdot\! 10^{\sage{find_threshold(0,55/100,100,10^25)[1]}}$ & \!\!\!\! & $0.51$\! & $\approx\sage{round( find_threshold(1,51/100,100,10^50)[0] ,ndigits=2)}\!\cdot\! 10^{\sage{find_threshold(1,51/100,100,10^50)[1]}}$\! \\ \hline
$0.6$ & $\sage{find_threshold(0,60/100,100,10^10)}$ & \!\!\!\! & $0.52$\! & $\approx\sage{round( find_threshold(1,52/100,100,10^30)[0] ,ndigits=2)}\!\cdot\! 10^{\sage{find_threshold(1,52/100,100,10^30)[1]}}$\! \\ \hline
$0.65$\! & $\sage{find_threshold(0,65/100,100,10^10)}$ & \!\!\!\! & $0.53$\! & $\approx\!\sage{round( find_threshold(1,53/100,100,10^18)[0] ,ndigits=2)}\!\cdot\! 10^{\sage{find_threshold(1,53/100,100,10^18)[1]}}$\! \\ \hline
$0.7$ & $\sage{find_threshold(0,70/100,100,10^7)}$ & \!\!\!\! & $0.54$\! & $\sage{find_threshold(1,54/100,100,10^18)}$ \\ \hline
$0.75$\! & $\sage{find_threshold(0,75/100,100,10^6)}$ & \!\!\!\! & $0.55$\! & $\sage{find_threshold(1,55/100,100,10^18)}$ \\ \hline
$0.8$ & $\sage{find_threshold(0,80/100,100,10^6)}$ & \!\!\!\! & $0.56$\! & $(<200)$ \\ \hline
$0.85$\! & $\sage{find_threshold(0,85/100,100,10^6)}$ & \!\!\!\! & $0.57$\! & $(<200)$ \\ \hline
$0.9$ & $\sage{find_threshold(0,90/100,100,10^6)}$ & \!\!\!\! & $0.58$\! & $(<200)$ \\ \hline
$0.95$\! & $\sage{find_threshold(0,95/100,100,10^6)}$ & \!\!\!\! & $0.59$\! & $\sage{find_threshold(1,59/100,4,10^8)}$ \\ \hline
\end{tabular}
\smallskip
\captionof{table}{$T$ for $\sigma>\frac{1}{2}$.}
\label{ta:sigma>1/2}
\end{minipage}
\hfill
\begin{minipage}[c]{0.45\linewidth}
\centering
\begin{tabular}{|l|l|l|}
\hline
$\sigma$ & Th.~\ref{thm:main} & Th.~\ref{thzeta}-\ref{th:z/s-all} \\ \hline
$0.05$\! & $\sage{find_threshold(0,5/100,100,10^6)}$ & $\sage{find_threshold(1,5/100,10,10^6)}$ \\ \hline
$0.1$ & $\sage{find_threshold(0,10/100,100,10^6)}$ & $\sage{find_threshold(1,10/100,100,10^6)}$ \\ \hline
$0.15$\! & $\sage{find_threshold(0,15/100,100,10^6)}$ & $\sage{find_threshold(1,15/100,100,10^6)}$ \\ \hline
$0.2$ & $\sage{find_threshold(0,20/100,100,10^6)}$ & $\sage{find_threshold(1,20/100,100,10^6)}$ \\ \hline
$0.25$\! & $\sage{find_threshold(0,25/100,100,10^6)}$ & $\sage{find_threshold(1,25/100,100,10^6)}$ \\ \hline
$0.3$ & $\sage{find_threshold(0,30/100,100,10^8)}$ & $\sage{find_threshold(1,30/100,100,10^8)}$ \\ \hline
$0.35$\! & $\sage{find_threshold(0,35/100,100,10^10)}$ & $\sage{find_threshold(1,35/100,100,10^10)}$ \\ \hline
$0.4$ & $\approx\sage{round( find_threshold(0,40/100,100,10^15)[0] ,ndigits=2)}\!\cdot\! 10^{\sage{find_threshold(0,40/100,100,10^15)[1]}}$\! & $\approx\sage{round( find_threshold(1,40/100,100,10^15)[0] ,ndigits=2)}\!\cdot\! 10^{\sage{find_threshold(1,40/100,100,10^15)[1]}}$\! \\ \hline
$0.45$\! & $\approx\sage{round( find_threshold(0,45/100,100,10^25)[0] ,ndigits=2)}\!\cdot\! 10^{\sage{find_threshold(0,45/100,100,10^25)[1]}}$\! & $\approx\sage{round( find_threshold(1,45/100,100,10^25)[0] ,ndigits=2)}\!\cdot\! 10^{\sage{find_threshold(1,45/100,100,10^25)[1]}}$\! \\ \hline
\end{tabular}
\smallskip
\smallskip
\smallskip
\captionof{table}{$T$ for $\sigma<\frac{1}{2}$.}
\label{ta:sigma<1/2}
\end{minipage}

\section*{Acknowledgements}

Thanks are due to F.~Aryan, J.~Bajpai, J.~Br\"udern, F.~Petrov, O.~Ramar\'e, A.~Simoni\v{c}, and M.~Young.

D.~Dona was supported by the European Research Council under Programme H2020-EU.1.1., ERC Grant ID: 648329 (codename GRANT). H.~Helfgott was
supported by the same ERC grant and by his Humboldt professorship. S.~Zuniga Alterman was supported by the fund of CONICYT PFCHA/DBCh/2015 - 72160520.

\Addresses

\tiny
\begin{mysage}
###
#For sigma=1, the bound in Thm. 1.1 has terms 1/T and log(T)/T^2,
#the bound in Thm. 4.6 has 1/T, log(T)/T^2 and 1/T^2:
#here we collect the coefficients of log(T)/T^2 and 1/T^2.
C_1_real=roundup(RIF(C_1_a_main+C_1_b_main/log(T0_1)).upper(),2)
if C_1_main!=C_1_real:
    String_C_1='Error in C_1. It should be '+ \
    str(C_1_real)+ \
    ' instead of '+str(C_1_main)+'. '
else:
    String_C_1=''
###
#For sigma in (1/2,1), the bound in Thm. 1.1 has terms 1/T and 1/T^(2*sigma),
#the bound in Thm. 3.1 has 1/T, 1/T^(2*sigma) and 1/T^(2*sigma+1):
#here we collect the coefficients of 1/T^(2*sigma) and 1/T^(2*sigma+1).
#The two coefficients C_121_num_a,b are multiplied by 1 and 1/(sigma-1/2) respectively.
max_111=max( kpar^(1/2)/(1/2), kpar^1/1 )
max_112=max( kpar^(1/2)/(2*1/2+1), kpar^1/(2*1+1) )
max_113=max( kpar^(1/2)/(1/2+1), kpar^1/(1+1) )
max_12x=max( kpar^(1/2), kpar^1 )
C_121_num_a_real=roundup(RIF(pi*(k111*max_111+k112*max_112+k113*max_113+k12x*max_12x/TH)).upper(),2)
if C_121_num_a_main!=C_121_num_a_real:
    String_C_121_num_a='Error in C_121_num_a. It should be '+ \
    str(C_121_num_a_real)+ \
    ' instead of '+str(C_121_num_a_main)+'. '
else:
    String_C_121_num_a=''
min_114=min( kpar^(1/2), kpar^1 )
C_121_num_b_real=rounddown(RIF(pi*k114*min_114/2).lower(),2)
if C_121_num_b_main!=C_121_num_b_real:
    String_C_121_num_b='Error in C_121_num_b. It should be '+ \
    str(C_121_num_b_real)+ \
    ' instead of '+str(C_121_num_b_main)+'. '
else:
    String_C_121_num_b=''
###
#For sigma in (1/2,1), the bound in Thm. 1.1 has terms 1/T and 1/T^(2*sigma),
#the bound in Thm. 4.6 has 1/T and 1/T^(2*sigma):
#here we collect the coefficient of 1/T^(2*sigma), multiplied by 1/((sigma-1/2)*(1-sigma)).
max_N1=1-1/2 #Maximum of sigma-1/2 in (1/2,1).
max_N2=1-1/2 #Maximum of 1-sigma in (1/2,1).
max_N0=(3/4-1/2)*(1-3/4) #Maximum of (sigma-1/2)*(1-sigma) in (1/2,1).
max_sqrtzeta=sqrt(2/3-1/2)*(1-2/3) #Maximum of sqrt(sigma-1/2)*(1-sigma) in (1/2,1), after using Lemma 2.12.
C_121_opt_real=roundup(RIF(2*(N_u[1]*max_N1+1/2*N_u[2]*max_N2+N_u[0]*max_N0)+ \
               (Dz+4)*(1/sqrt(2)*max_sqrtzeta+sqrt(H)*max_N0)).upper(),2)
if C_121_opt_main!=C_121_opt_real:
    String_C_121_opt='Error in C_121_opt. It should be '+ \
    str(C_121_opt_real)+ \
    ' instead of '+str(C_121_opt_main)+'. '
else:
    String_C_121_opt=''
###
#For sigma=1/2, the bound in Thm. 1.1 has terms log(T)/T and 1/T,
#the bound in Thm. 3.1 has log(T)/T, 1/T and 1/T^2:
#here we collect the coefficients of 1/T and 1/T^2.
C_12_num_real=roundup(RIF(pi*(c21x+c22x/TH)).upper(),2)
if C_12_num_main!=C_12_num_real:
    String_C_12_num='Error in C_12_num. It should be '+ \
    str(C_12_num_real)+ \
    ' instead of '+str(C_12_num_main)+'. '
else:
    String_C_12_num=''
###
#For sigma=1/2, the bound in Thm. 1.1 has terms log(T)/T and sqrt(log(T))/T,
#the bound in Thm. 4.6 has log(T)/T, sqrt(log(T))/T and 1/T:
#here we collect the coefficients of sqrt(log(T))/T and 1/T.
C_12_opt_real=roundup(RIF(C_12_a_main+C_12_b_main/sqrt(log(T0_12))).upper(),2)
if C_12_opt_main!=C_12_opt_real:
    String_C_12_opt='Error in C_12_opt. It should be '+ \
    str(C_12_opt_real)+ \
    ' instead of '+str(C_12_opt_main)+'. '
else:
    String_C_12_opt=''
###
#For sigma in (0,1/2), the bound in Thm. 1.1 has terms 1/T^(2*sigma) and 1/T,
#the bound in Thm. 3.1 has 1/T^(2*sigma), 1/T and 1/T^(2*sigma+1):
#here we collect first the coefficient of 1/T^(2*sigma) in C_012_num_a,b,c,
#multiplied by 1/sigma, 1/(1/2-sigma) and 1 respectively;
#then we collect the coefficients of 1/T and 1/T^(2*sigma+1).
#We can use kpar^sigma/sigma=1/sigma+(kpar^sigma-1)/sigma and then maximize
#the second term in sigma=1/2, to get a better coefficient for sigma going to 0.
max_311=kpar^0
max_311bis=(kpar^(1/2)-1)/(1/2)
C_012_num_a_real=roundup(RIF(pi*k111*max_311).upper(),2)
if C_012_num_a_main!=C_012_num_a_real:
    String_C_012_num_a='Error in C_012_num_a. It should be '+ \
    str(C_012_num_a_real)+ \
    ' instead of '+str(C_012_num_a_main)+'. '
else:
    String_C_012_num_a=''
max_314=kpar^(1/2)
C_012_num_b_real=roundup(RIF(pi*k314*max_314/2).upper(),2)
if C_012_num_b_main!=C_012_num_b_real:
    String_C_012_num_b='Error in C_012_num_b. It should be '+ \
    str(C_012_num_b_real)+ \
    ' instead of '+str(C_012_num_b_main)+'. '
else:
    String_C_012_num_b=''
max_312=max( kpar^0/(2*0+1), kpar^(1/2)/(2*1/2+1) )
max_313=max( kpar^0/(0+1), kpar^(1/2)/(1/2+1) )
C_012_num_c_real=roundup(RIF(pi*(k111*max_311bis+k112*max_312+k113*max_313)).upper(),2)
if C_012_num_c_main!=C_012_num_c_real:
    String_C_012_num_c='Error in C_012_num_c. It should be '+ \
    str(C_012_num_c_real)+ \
    ' instead of '+str(C_012_num_c_main)+'. '
else:
    String_C_012_num_c=''
max_32x=max( (kpar/TH^2)^0, (kpar/TH^2)^(1/2) )/abs(zeta(2*0))
C_012_num_d_real=roundup(RIF(pi*max( 0, -c30x+k12x*max_32x )).upper(),2)
if C_012_num_d_main!=C_012_num_d_real:
    String_C_012_num_d='Error in C_012_num_d. It should be '+ \
    str(C_012_num_d_real)+ \
    ' instead of '+str(C_012_num_d_main)+'. '
else:
    String_C_012_num_d=''
###
#For sigma in (0,1/2), the bound in Thm. 1.1 has terms 1/T^(2*sigma) and 1/T,
#the bound in Thm. 4.6 has 1/T^(2*sigma) and 1/T:
#here we collect the coefficient of 1/T, multiplied by 1/(sigma^2*(1/2-sigma)).
si_L1=(log(2*pi)-2+sqrt((log(2*pi))^2+4))/(4*log(2*pi))
max_L1=si_L1*(1/2-si_L1)/(2*pi)^(1-2*si_L1) #Maximum of sigma*(1/2-sigma)/(2*pi)^(1-2*sigma) in (0,1/2).
si_L11=1/(2*log(2*pi))
max_L11=(1/2-si_L1)/(2*pi)^(1-2*si_L1) #Maximum of (1/2-sigma)/(2*pi)^(1-2*sigma) in (0,1/2).
max_L12=(1/2)^2/(2*pi)^(1-2*1/2) #Maximum of sigma^2/(2*pi)^(1-2*sigma) in (0,1/2).
si_L0=(log(2*pi)-3+sqrt((log(2*pi))^2+2*log(2*pi)+9))/(4*log(2*pi))
max_L0=si_L0^2*(1/2-si_L0)/(2*pi)^(1-2*si_L0) #Maximum of sigma^2*(1/2-sigma)/(2*pi)^(1-2*sigma) in (0,1/2).
C_012_opt_real=roundup(RIF(2*(L_1_u_main*max_L1+L_11_u_main*max_L11+1/2*L_12_u_main*max_L12+L_0_u_main*max_L0)).upper(),2)
if C_012_opt_main!=C_012_opt_real:
    String_C_012_opt='Error in C_012_opt. It should be '+ \
    str(C_012_opt_real)+ \
    ' instead of '+str(C_012_opt_main)+'. '
else:
    String_C_012_opt=''
###
#For sigma=1, the bound in Thm. 1.2 has terms T and sqrt(T),
#the bound in Thm. 4.3 has T, sqrt(T) and log(T):
#here we collect the coefficient of sqrt(T) (see the code in section 4.2 for explanations).
Cz_1_real=roundup(Cz_fix_coeff(Cz_1_string(f11_1_u,f12_1_u,f13_1_u,rho_1_m12_0,0),8,1,1,0).upper(),2)
if Cz_1_main!=Cz_1_real:
    String_Cz_1='Error in Cz_1. It should be '+ \
    str(Cz_1_real)+ \
    ' instead of '+str(Cz_1_main)+'. '
else:
    String_Cz_1=''
###
#For sigma in (1/2,1), the bound in Thm. 1.2 has terms T and max(T^(2-2*sigma)*log(T),sqrt(T)),
#the bound in Thm. 4.3 has T and max(T^(2-2*sigma)*log(T),sqrt(T)):
#here we collect the coefficient of max(T^(2-2*sigma)*log(T),sqrt(T)), multiplied by 1/((sigma-1/2)*(1-sigma)^2).
max_C1=(3/4-1/2)*(1-3/4) #Maximum of (sigma-1/2)*(1-sigma) in (1/2,1).
max_C11=1-1/2 #Maximum of sigma-1/2 in (1/2,1).
max_C2=(1-1/2)^2 #Maximum of (1-sigma)^2 in (1/2,1).
max_C0=(2/3-1/2)*(1-2/3)^2 #Maximum of (sigma-1/2)*(1-sigma)^2 in (1/2,1).
Cz_121_real=roundup(RIF(Cz_121_u[2]*max_C11+Cz_121_u[1]*max_C1+1/2*Cz_121_u[3]*max_C2+Cz_121_u[0]*max_C0).upper(),2)
if Cz_121_main!=Cz_121_real:
    String_Cz_121='Error in Cz_121. It should be '+ \
    str(Cz_121_real)+ \
    ' instead of '+str(Cz_121_main)+'. '
else:
    String_Cz_121=''
###
#For sigma=1/2, the bound in Thm. 1.2 has terms T*log(T), T*sqrt(log(T)) and T, like Thm. 4.3:
#here we simplify the coefficients of T*sqrt(log(T)) and T.
Cz_12_a_real=roundup(Cz_12_ua.upper(),2)
if Cz_12_a_main!=Cz_12_a_real:
    String_Cz_12_a='Error in Cz_12_a. It should be '+ \
    str(Cz_12_a_real)+ \
    ' instead of '+str(Cz_12_a_main)+'. '
else:
    String_Cz_12_a=''
Cz_12_b_real=roundup(Cz_12_ub.upper(),2)
if Cz_12_b_main!=Cz_12_b_real:
    String_Cz_12_b='Error in Cz_12_b. It should be '+ \
    str(Cz_12_b_real)+ \
    ' instead of '+str(Cz_12_b_main)+'. '
else:
    String_Cz_12_b=''
###
#For sigma in (0,1/2), the bound in Thm. 1.2 has terms T^(2-2*sigma) and T,
#the bound in Thm. 4.3 has T^(2-2*sigma) and T:
#here we collect the coefficient of T, multiplied by 1/(sigma^2*(1/2-sigma)).
#The calculations are identical to those for C_012_opt up to a factor of 2.
Cz_012_real=roundup(RIF(L_1_u_main*max_L1+L_11_u_main*max_L11+1/2*L_12_u_main*max_L12+L_0_u_main*max_L0).upper(),2)
if Cz_012_main!=Cz_012_real:
    String_Cz_012='Error in Cz_012. It should be '+ \
    str(Cz_012_real)+ \
    ' instead of '+str(Cz_012_main)+'. '
else:
    String_Cz_012=''
###
#For sigma=0, the bound in Thm. 1.2 has terms T^2 and T*log(T),
#the bound in Thm. 4.5 has T^2, T*log(T) and T:
#here we collect the coefficients of T*log(T) and T.
Cz_0_real=roundup(RIF(Cz_0_u_a_main+Cz_0_u_b_main/log(TZ)).upper(),2)
if Cz_0_main!=Cz_0_real:
    String_Cz_0='Error in Cz_0. It should be '+ \
    str(Cz_0_real)+ \
    ' instead of '+str(Cz_0_main)+'. '
else:
    String_Cz_0=''
###
#Here we write down the constants listed in Thm. 3.1.
#The code that computes them is in section 3.3.
if kpar_main!=roundup(kpar.upper(),5):
    String_kpar='Error in kpar. It should be '+ \
    str(roundup(kpar.upper(),5))+ \
    ' instead of '+str(kpar_main)+'. '
else:
    String_kpar=''
if k111_main!=roundup(k111.upper(),5):
    String_k111='Error in k111. It should be '+ \
    str(roundup(k111.upper(),5))+ \
    ' instead of '+str(k111_main)+'. '
else:
    String_k111=''
if k112_main!=roundup(k112.upper(),5):
    String_k112='Error in k112. It should be '+ \
    str(roundup(k112.upper(),5))+ \
    ' instead of '+str(k112_main)+'. '
else:
    String_k112=''
if k113_main!=roundup(k113.upper(),5):
    String_k113='Error in k113. It should be '+ \
    str(roundup(k113.upper(),5))+ \
    ' instead of '+str(k113_main)+'. '
else:
    String_k113=''
if k114_main!=rounddown(k114.lower(),5):
    String_k114='Error in k114. It should be '+ \
    str(rounddown(k114.lower(),5))+ \
    ' instead of '+str(k114_main)+'. '
else:
    String_k114=''
if k12x_main!=roundup(k12x.upper(),5):
    String_k12x='Error in k12x. It should be '+ \
    str(roundup(k12x.upper(),5))+ \
    ' instead of '+str(k12x_main)+'. '
else:
    String_k12x=''
if c21x_main!=roundup(c21x.upper(),5):
    String_c21x='Error in c21x. It should be '+ \
    str(roundup(c21x.upper(),5))+ \
    ' instead of '+str(c21x_main)+'. '
else:
    String_c21x=''
if c22x_main!=roundup(c22x.upper(),5):
    String_c22x='Error in c22x. It should be '+ \
    str(roundup(c22x.upper(),5))+ \
    ' instead of '+str(c22x_main)+'. '
else:
    String_c22x=''
if k314_main!=roundup(k314.upper(),5):
    String_k314='Error in k314. It should be '+ \
    str(roundup(k314.upper(),5))+ \
    ' instead of '+str(k314_main)+'. '
else:
    String_k314=''
if c30x_main!=rounddown(c30x.lower(),5):
    String_c30x='Error in c30x. It should be '+ \
    str(rounddown(c30x.lower(),5))+ \
    ' instead of '+str(c30x_main)+'. '
else:
    String_c30x=''
###
#Here we compute the value of the constants defining L^+,L^- as in Thm. 4.5:
#see the proof of the theorem, starting with (4.16).
#In brief, it will be (1/(2*pi)^(1-2*sigma))*S+(K2/(2*pi)^(1-2*sigma))*(whole bound in Thm. 4.3),
#where we use Lemma 2.12 for the main term of the bound in Thm. 4.3.
#Given the O* in (4.16), we must use the upper bound in Thm. 4.3 for both L^+ and L^-.
L_1_u_real=roundup(RIF(S_u___1+K2*Cz_121_u[1]).upper(),5)
if L_1_u_main!=L_1_u_real:
    String_L_1_u='Error in L_1_u. It should be '+ \
    str(L_1_u_real)+ \
    ' instead of '+str(L_1_u_main)+'. '
else:
    String_L_1_u=''
L_11_u_real=roundup(RIF(K2*Cz_121_u[2]).upper(),5)
if L_11_u_main!=L_11_u_real:
    String_L_11_u='Error in L_11_u. It should be '+ \
    str(L_11_u_real)+ \
    ' instead of '+str(L_11_u_main)+'. '
else:
    String_L_11_u=''
L_12_u_real=roundup(RIF(S_u___12+K2*(1+Cz_121_u[3])).upper(),5)
if L_12_u_main!=L_12_u_real:
    String_L_12_u='Error in L_12_u. It should be '+ \
    str(L_12_u_real)+ \
    ' instead of '+str(L_12_u_main)+'. '
else:
    String_L_12_u=''
L_0_u_real=roundup(RIF(S_u___0+K2*(H+Cz_121_u[0])).upper(),5)
if L_0_u_main!=L_0_u_real:
    String_L_0_u='Error in L_0_u. It should be '+ \
    str(L_0_u_real)+ \
    ' instead of '+str(L_0_u_main)+'. '
else:
    String_L_0_u=''
L_1_l_real=roundup(RIF(S_l___1+K2*Cz_121_u[1]).upper(),5)
if L_1_l_main!=L_1_l_real:
    String_L_1_l='Error in L_1_l. It should be '+ \
    str(L_1_l_real)+ \
    ' instead of '+str(L_1_l_main)+'. '
else:
    String_L_1_l=''
L_11_l_real=roundup(RIF(K2*Cz_121_u[2]).upper(),5)
if L_11_l_main!=L_11_l_real:
    String_L_11_l='Error in L_11_l. It should be '+ \
    str(L_11_l_real)+ \
    ' instead of '+str(L_11_l_main)+'. '
else:
    String_L_11_l=''
L_12_l_real=roundup(RIF(S_l___12+K2*(1+Cz_121_u[3])).upper(),5)
if L_12_l_main!=L_12_l_real:
    String_L_12_l='Error in L_12_l. It should be '+ \
    str(L_12_l_real)+ \
    ' instead of '+str(L_12_l_main)+'. '
else:
    String_L_12_l=''
L_0_l_real=roundup(RIF(S_l___0+K2*(H+Cz_121_u[0])).upper(),5)
if L_0_l_main!=L_0_l_real:
    String_L_0_l='Error in L_0_l. It should be '+ \
    str(L_0_l_real)+ \
    ' instead of '+str(L_0_l_main)+'. '
else:
    String_L_0_l=''
###
#Here we compute the value of the constants for sigma=0 in Thm. 4.5:
#see the proof of the theorem, starting with (4.16).
#In brief, the coefficient of T*log(T) will be (1/(2*pi))*W,
#while the coefficient of T will be (1/(2*pi))*2*pi*sqrt(2/3)+(K2/(2*pi))*(whole bound in Thm. 4.3).
#Given the O* in (4.16), we must use the upper bound in Thm. 4.3 for both bounds here.
Cz_0_u_a_real=roundup(RIF(Cz_1gen_u/(2*pi)).upper(),5)
if Cz_0_u_a_main!=Cz_0_u_a_real:
    String_Cz_0_u_a='Error in Cz_0_u_a. It should be '+ \
    str(Cz_0_u_a_real)+ \
    ' instead of '+str(Cz_0_u_a_main)+'. '
else:
    String_Cz_0_u_a=''
Cz_0_u_b_real=roundup(RIF(sqrt(2/3)+K2/(2*pi)*(pi^2/6+pi*sqrt(2/3)/sqrt(TZ)+Cz_1_u*log(TZ)/TZ)).upper(),5)
if Cz_0_u_b_main!=Cz_0_u_b_real:
    String_Cz_0_u_b='Error in Cz_0_u_b. It should be '+ \
    str(Cz_0_u_b_real)+ \
    ' instead of '+str(Cz_0_u_b_main)+'. '
else:
    String_Cz_0_u_b=''
Cz_0_l_a_real=roundup(RIF(Cz_1gen_l/(2*pi)).upper(),5)
if Cz_0_l_a_main!=Cz_0_l_a_real:
    String_Cz_0_l_a='Error in Cz_0_l_a. It should be '+ \
    str(Cz_0_l_a_real)+ \
    ' instead of '+str(Cz_0_l_a_main)+'. '
else:
    String_Cz_0_l_a=''
Cz_0_l_b_real=roundup(RIF(sqrt(2/3)+K2/(2*pi)*(pi^2/6+pi*sqrt(2/3)/sqrt(TZ)+Cz_1_u*log(TZ)/TZ)).upper(),5)
if Cz_0_l_b_main!=Cz_0_l_b_real:
    String_Cz_0_l_b='Error in Cz_0_l_b. It should be '+ \
    str(Cz_0_l_b_real)+ \
    ' instead of '+str(Cz_0_l_b_main)+'. '
else:
    String_Cz_0_l_b=''
###
#Here we compute the value of the constants for sigma=1/2 in Thm. 4.6:
#see the proof of the theorem.
#In brief, the integral of (log(x)+1)/x^2 is -log(x)/x-2/x
#and the integral of 2/x^2 is -2/x;
#for the integral of 2*sqrt(log(x))/x^2, we use the following simplification:
#2*sqrt(log(x))/x^2<2*(sqrt(log(x))/x^2-1/(2*sqrt(log(x))*x^2)+1/(2*sqrt(log(T0))*x^2)
#for any x>T0, and the integral of the RHS is 2*(-sqrt(log(x))/x-1/(2*sqrt(log(T0))*x).
C_12_a_real=roundup(RIF(2*Cz_12_ua).upper(),5)
if C_12_a_main!=C_12_a_real:
    String_C_12_a='Error in C_12_a. It should be '+ \
    str(C_12_a_real)+ \
    ' instead of '+str(C_12_a_main)+'. '
else:
    String_C_12_a=''
C_12_b_real=roundup(RIF(2+2*Cz_12_ub+1/sqrt(log(TZ))*Cz_12_ua).upper(),5)
if C_12_b_main!=C_12_b_real:
    String_C_12_b='Error in C_12_b. It should be '+ \
    str(C_12_b_real)+ \
    ' instead of '+str(C_12_b_main)+'. '
else:
    String_C_12_b=''
###
#Here we compute the value of the constants for sigma=1 in Thm. 4.6:
#see the proof of the theorem.
#In brief, the integral of zeta(2)/x^2 is -zeta(2)/x,
#the integral of 2/x^(5/2) is -4/3*1/x^(3/2),
#and the integral of 2*log(x)/x^3 is -log(x)/x^2-1/2*1/x^2.
C_1_a_real=roundup(Cz_1gen_u.upper(),5)
if C_1_a_main!=C_1_a_real:
    String_C_1_a='Error in C_1_a. It should be '+ \
    str(C_1_a_real)+ \
    ' instead of '+str(C_1_a_main)+'. '
else:
    String_C_1_a=''
C_1_b_real=roundup(RIF(pi*sqrt(2/3)*4/3+Cz_1gen_u/2).upper(),5)
if C_1_b_main!=C_1_b_real:
    String_C_1_b='Error in C_1_b. It should be '+ \
    str(C_1_b_real)+ \
    ' instead of '+str(C_1_b_main)+'. '
else:
    String_C_1_b=''
###
#Here we compute the threshold between the bounds in Thm. 3.1 and Thm. 4.6 for sigma=1/2:
#see section 5 for details.
T0_12_tooexact=find_threshold(0,1/2,100,10^40)
if T0_12_exact<=T0_12_tooexact or T0_12_exact-10^(Exp_12-6)>=T0_12_tooexact:
    String_T0_12='Error in T0_12. It should be '+ \
    str(T0_12_tooexact)+ \
    ' instead of '+str(T0_12_exact)+'. '
else:
    String_T0_12=''
###
#Here we ensure that the "<200" entries in Table 1 are correct:
#see section 5 for details.
Tab1=find_threshold(1,56/100,100,10^8)
Tab2=find_threshold(1,57/100,10,10^8)
Tab3=find_threshold(1,58/100,4,10^8)
if Tab1!='<200' or Tab2!='<200' or Tab3!='<200':
    String_Tab='Error in Table 1, in some of the <200 entries. '
else:
    String_Tab=''
###
\end{mysage}
\normalsize

$\sage{String_C_1}\sage{String_C_121_num_a}\sage{String_C_121_num_b}\sage{String_C_121_opt}\sage{String_C_12_num}\sage{String_C_12_opt}\sage{String_C_012_opt}\sage{String_C_012_num_a}\sage{String_C_012_num_b}\sage{String_C_012_num_c}\sage{String_C_012_num_d}\sage{String_Cz_1}\sage{String_Cz_121}\sage{String_Cz_12_a}\sage{String_Cz_12_b}\sage{String_Cz_012}\sage{String_Cz_0}\sage{String_kpar}\sage{String_k111}\sage{String_k112}\sage{String_k113}\sage{String_k114}\sage{String_k12x}\sage{String_c21x}\sage{String_c22x}\sage{String_k314}\sage{String_c30x}\sage{String_L_1_u}\sage{String_L_11_u}\sage{String_L_12_u}\sage{String_L_0_u}\sage{String_L_1_l}\sage{String_L_11_l}\sage{String_L_12_l}\sage{String_L_0_l}\sage{String_Cz_0_u_a}\sage{String_Cz_0_u_b}\sage{String_Cz_0_l_a}\sage{String_Cz_0_l_b}\sage{String_C_12_a}\sage{String_C_12_b}\sage{String_C_1_a}\sage{String_C_1_b}\sage{String_T0_12}\sage{String_Tab}$

\end{document}